\documentclass[11pt,reqno,twoside]{amsart}

\usepackage{amsfonts,amsmath,amssymb}
\usepackage{mathrsfs,mathtools}
\usepackage{enumerate}
\usepackage{hyperref}
\usepackage{esint}
\usepackage{graphicx}
\usepackage{bm}
\usepackage{commath}
\usepackage{esint}
\DeclareMathAlphabet{\mathpzc}{OT1}{pzc}{m}{it}

\usepackage{placeins} 
\usepackage{flafter} 

\usepackage[textsize=small]{todonotes}
\usepackage{caption}

\numberwithin{equation}{section}

\usepackage[dvips,bottom=1.4in,right=1in,top=1in, left=1in]{geometry}

\makeatletter
\def\eqnarray{\stepcounter{equation}\let\@currentlabel=\theequation
\global\@eqnswtrue
\tabskip\@centering\let\\=\@eqncr
$$\halign to \displaywidth\bgroup\hfil\global\@eqcnt\z@
  $\displaystyle\tabskip\z@{##}$&\global\@eqcnt\@ne
  \hfil$\displaystyle{{}##{}}$\hfil
  &\global\@eqcnt\tw@ $\displaystyle{##}$\hfil
  \tabskip\@centering&\llap{##}\tabskip\z@\cr}

\def\endeqnarray{\@@eqncr\egroup
      \global\advance\c@equation\m@ne$$\global\@ignoretrue}

\newtheorem{theorem}{Theorem}[section]

\newtheorem{definition}[theorem]{Definition}
\newtheorem{example}[theorem]{Example}
\newtheorem{lemma}[theorem]{Lemma}

\newtheorem{proposition}[theorem]{Proposition}
\newtheorem{assumption}[theorem]{Assumption}
\newtheorem{remark}[theorem]{Remark}
\numberwithin{equation}{section}

\def\Omc{\mathbb{R}^N\setminus\Omega}

\def\RR{{\mathbb{R}}}
\def\NN{{\mathbb{N}}}

\def\Om{\Omega}

\def\bOm{\overline{\Om}}
\def\pOm{\partial\Omega}

\title{External optimal control of nonlocal PDEs}

\author{Harbir Antil}
\address{Department of Mathematical Sciences, George Mason University, Fairfax, VA 22030, USA.}
\email{hantil@gmu.edu}

\author{Ratna Khatri}
\address{Department of Mathematical Sciences, George Mason University, Fairfax, VA 22030, USA.}
\email{rkhatri3@gmu.edu}

\author{Mahamadi Warma}
\address{University of Puerto Rico  (Rio Piedras Campus), College of Natural Sciences,
Department of Mathematics, PO Box 70377 San Juan PR
00936-8377 (USA). }
\email{mahamadi.warma1@upr.edu, mjwarma@gmail.com}

\thanks{The first and second authors are partially supported by NSF grant DMS-1521590 
and DMS-1818772 and Air Force Office of Scientic Research under Award NO: FA9550-19-1-0036. 
The third author is partially supported by the Air Force Office of Scientific Research under 
Award NO:  FA9550-18-1-0242}

\keywords{Fractional Laplacian, interaction operator, weak and very-weak solutions, Dirichlet control problem, Robin control problem, external control. }

\subjclass[2010]{49J20, 49K20, 35S15, 65R20, 65N30}

\begin{document}

\begin{abstract}
Very recently Warma \cite{warma2018approximate} has shown that for nonlocal PDEs associated with the fractional Laplacian, the classical notion of controllability from the boundary does not make sense and therefore it must be replaced by a control that is localized outside the open set where the PDE is solved.
Having learned from the above mentioned result, in this paper we introduce a new class
of source identification and optimal control problems where the source/control is located outside the observation domain where
the PDE is satisfied. The classical diffusion models lack this flexibility as they assume that the source/control is located either inside
or on the boundary. This is essentially due to the locality property of the underlying 
operators. 
We use the nonlocality of the fractional operator to create a framework
that now allows placing a source/control outside the observation domain. 
We consider the Dirichlet, Robin and Neumann source identification or optimal control problems. 
These problems require dealing with the nonlocal normal derivative (that we shall call interaction 
operator). We create a functional analytic framework and show well-posedness and derive the first order 
optimality conditions for these problems. We introduce a new approach to approximate, with 
convergence rate, the Dirichlet problem with nonzero exterior condition. The numerical examples 
confirm our theoretical findings and illustrate the practicality of our approach.
\end{abstract}

\maketitle

\section{Introduction and Motivation}

In many real life applications a source or a control is placed outside 
(disjoint from) the observation domain $\Om$ where PDE is satisfied. 
Some examples of inverse and optimal control problems where this 
situation can arise are: 
(i) Acoustic testing, when the loudspeakers are placed far from the 
aerospace structures \cite{larkin1999direct}; 
(ii) Magnetotellurics (MT), which is a technique to infer earth's subsurface
electrical conductivity from surface measurements \cite{unsworth2005new, CWeiss_BvBWaanders_HAntil_2018a}; 
(iii) Magnetic drug targeting (MDT), where drugs with ferromagnetic particles 
in suspension are injected into the body and the external magnetic field is 
then used to steer the drug to relevant areas, for example, solid tumors \cite{lubbe1996clinical,HAntil_RHNochetto_PVenegas_2018a, HAntil_RHNochetto_PVenegas_2018b}; 
(iv) Electroencephalography (EEG) is used to record electrical activity in brain \cite{williams1974electroencephalography, niedermeyer2005electroencephalography}, 
in case one accounts for the neurons disjoint from the brain, we will obtain
an external source problem.

This is different from the  traditional approaches where the source/control
is placed either inside the domain $\Om$ or on the boundary 
$\pOm$ of $\Om$.
This is not surprising since
in many cases we do not have 
a direct access to $\pOm$. See for instance, the setup in Figure~\ref{f:twodomains}.
%
 \begin{figure}[!h]
    \centering
    \includegraphics[width=0.3\textwidth]{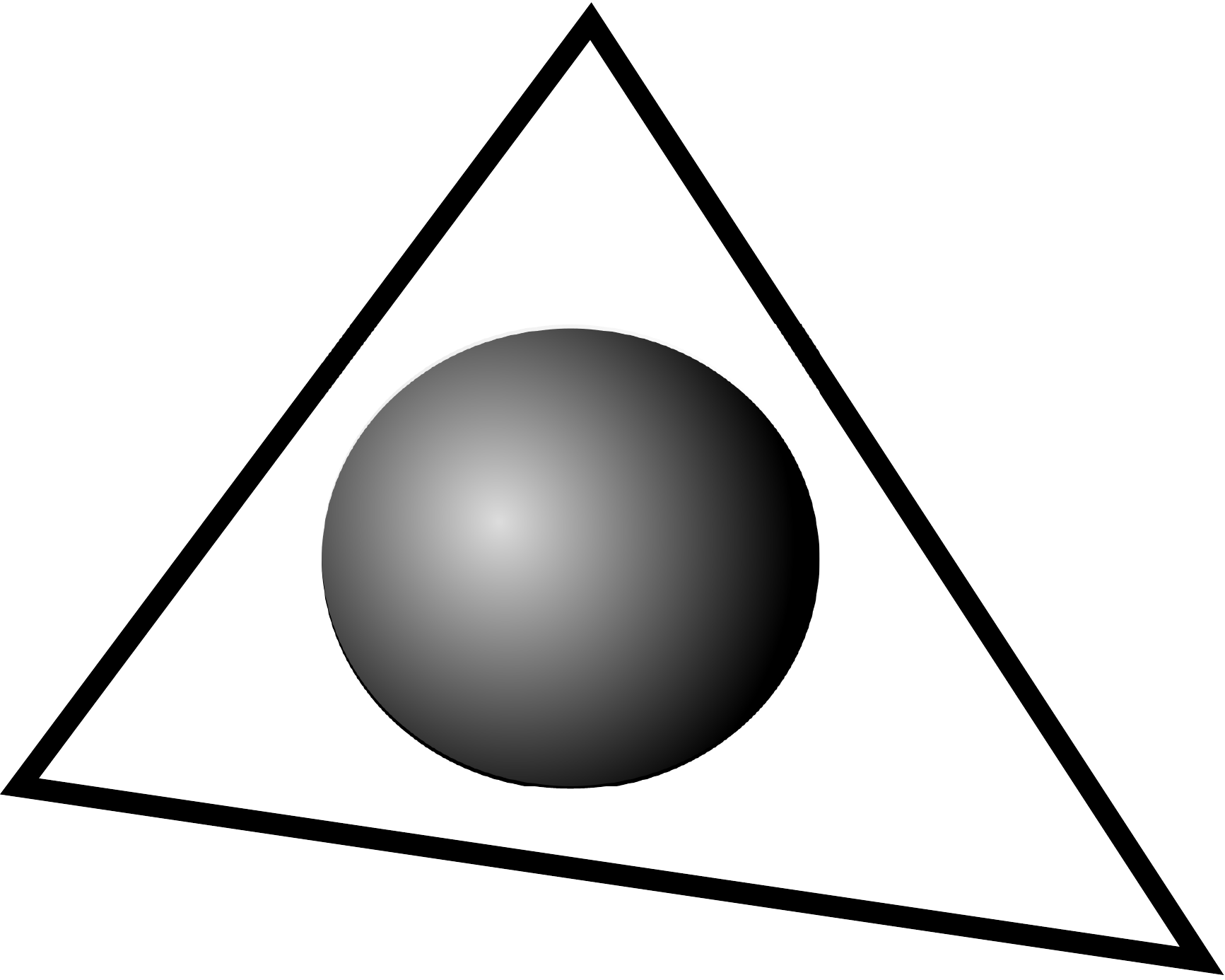}  \qquad 
    \includegraphics[width = 0.28\textwidth]{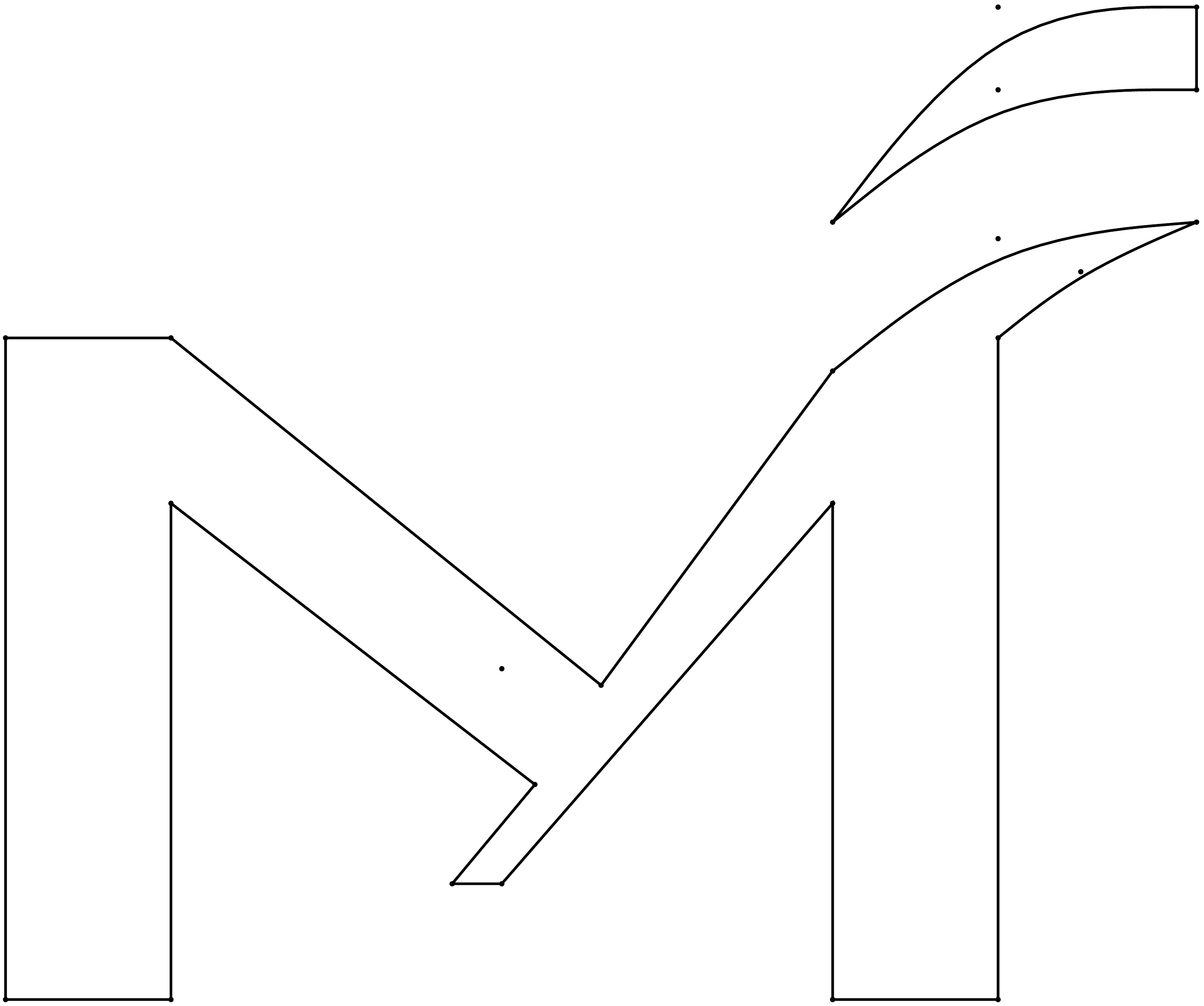}
    \caption{\label{f:twodomains}
    Let a diffussion process occurs inside a domain $\Om$ which is  
    the sphere in the left panel and the letter M in the right panel.     
    We are interested in the source idenfication or controlling 
    this diffusion process by placing the source/control in a set $\widehat\Om$
    which is disjoint from $\Om$. In the above figure $\widehat\Om$ is 
    the triangular pipe in the left panel and the structure on top of the letter 
    M in the right panel.}    
 \end{figure}  
In such applications the existing models can be 
ineffective due to their strict requirements. Indeed think of the 
source identification problem for the most basic Poisson equation:
 \begin{equation}\label{eq:Poisson}
  -\Delta u = f \quad \mbox{in } \Om, \quad u = z \quad \mbox{on } \pOm ,
 \end{equation}
where the source is either $f$ (force or load) or $z$ (boundary control) see \cite{antil2018frontiers,JLLions_1971a,FTroeltzsch_2010a}. In \eqref{eq:Poisson} there
is no provision to place the source in $\widehat\Om$ (cf.~Figure~\ref{f:twodomains}). 
The issue is that the operator $\Delta$ has ``lesser reach", in other words, it is a 
local operator. On the other hand the fractional Laplacian $(-\Delta)^s$ with $0< s <1$ 
(see \eqref{eq11}) is a nonlocal operator. This difference in behavior can be
easily seen in our numerical examples in Section~\ref{s:source} where we observe that we cannot
see the external source as $s$ approaches 1.

Recently, nonlocal diffusion operators such as the fractional Laplacian $(-\Delta)^s$ have
emerged as an excellent alternative to model diffusion. Under a probabilistic framework
this operator can be derived as a limit of the so-called \emph{long jump} random walk \cite{EValdinoci_2009a}.
Recall that $\Delta$ is the limit of the classical random walk or the Brownian motion. 
More applications of these models appear in (but not limited to) image denoising and phase field modeling
\cite{antil2017spectral}; image denosing where $s$ is allowed to be spatially 
dependent \cite{HAntil_CNRautenberg_2018b}; fractional diffusion maps  (data analysis) \cite{antil2018fractional}; 
magnetotellurics (geophysics) \cite{CWeiss_BvBWaanders_HAntil_2018a}.


Coming back to the question of source/control placement, we next state the exterior value 
problem corresponding to $(-\Delta)^s$. Find $u$ in an appropriate function space satisfying
 \begin{equation}\label{eq:fracPoisson}
  (-\Delta)^s u = f \quad \mbox{in } \Om, \quad u = z \quad \mbox{on } \RR^N\setminus\Om .
 \end{equation} 
As in the case of \eqref{eq:Poisson}, besides $f$ being the source/control in $\Om$ we can 
also place the source/control $z$ in the exterior domain $\RR^N\setminus\Om$. However, the action of 
$z$ in \eqref{eq:fracPoisson} is 
significantly different from \eqref{eq:Poisson}. Indeed, the source/control in \eqref{eq:Poisson} is placed on the boundary $\pOm$, but the source/control $z$ in \eqref{eq:fracPoisson} is placed outside in $\RR^N\setminus\Om$ which is what we wanted to 
achieve in Figure~\ref{f:twodomains}. For completeness, we refer to \cite{antil2017optimal} for 
the optimal control problem, with $f$ being the source/control and \cite{antil2018b,antil2016optimal}  for another inverse problem to identify the coefficients in the fractional $p$-Laplacian.

The purpose of this paper is to introduce and study a new class of the Dirichlet, Robin, 
and Neumann source identification problems or the optimal control problems. We shall use 
these terms interchangeably but we will make a distinction in our numerical experiments. We emphasize that yet another class of identification 
where the unknown is the fractional exponent $s$ for the spectral fractional Laplacian
(which is different from the operator under consideration) was recently considered in 
\cite{sprekels2016new}. We shall describe our problems next.


Let $\Om \subset \RR^N$, $N \ge 1$, be a bounded open set with boundary $\pOm$. 
Let $(Z_D,U_D)$ and $(Z_{R},U_R)$, where subscripts $D$ and $R$ indicate Dirichlet 
and Robin, be Banach spaces. The goal of this 
paper is to consider the following two external control or source 
identification problems. The source/control in our case is denoted by $z$.

\begin{itemize}
 \item \emph{\bf Fractional Dirichlet exterior control problem:} Given $\xi \ge 0$
 a constant penalty parameter we consider the minimization problem:
 \begin{subequations}\label{eq:dcp}
 \begin{equation}\label{eq:Jd}
    \min_{(u,z)\in (U_D,Z_D)} J(u) + \frac{\xi}{2} \|z\|^2_{Z_D} ,
 \end{equation}
 subject to the fractional Dirichlet exterior value problem: Find $u \in U_D$ solving 
 \begin{equation}\label{eq:Sd}
 \begin{cases}
    (-\Delta)^s u &= 0 \quad \mbox{in } \Om, \\
                u &= z \quad \mbox{in } \RR^N\setminus \Om  ,
 \end{cases}                
 \end{equation} 
 and the control constraints 
 \begin{equation}\label{eq:Zd}
    z \in Z_{ad,D} ,
 \end{equation}
 \end{subequations}
 with $Z_{ad,D} \subset Z_D$ being a closed and convex subset. 

\item \emph{\bf Fractional Robin exterior control problem:} Given $\xi \ge 0$
 a constant penalty parameter we consider the minimization problem
 \begin{subequations}\label{eq:ncp}
 \begin{equation}\label{eq:Jn}
    \min_{(u,z)\in (U_R, Z_R)} J(u) + \frac{\xi}{2} \|z\|^2_{Z_R} ,
 \end{equation}
 subject to the fractional Robin exterior value problem: Find $u \in U_R$ solving 
 \begin{equation}\label{eq:Sn}
 \begin{cases}
    (-\Delta)^s u &= 0 \quad \mbox{in } \Om, \\
    \mathcal{N}_s u  +\kappa u &= \kappa z \quad \mbox{in } \RR^N\setminus \Om ,
 \end{cases}                
 \end{equation} 
 and the control constraints 
 \begin{equation}\label{eq:Zn}
    z \in Z_{ad,R} ,
 \end{equation}
 \end{subequations}
 with $Z_{ad,R} \subset Z_R$ being a closed and convex subset. In \eqref{eq:Sn}, $\mathcal N_su$ is the nonlocal normal derivative of $u$ given in \eqref{NLND} below, $\kappa\in L^1(\Omc)\cap L^\infty(\Omc)$ and is non-negative. We notice that the latter assumption is not a restriction since otherwise we can replace $\kappa$ throughout by $|\kappa|$.
 \end{itemize}

The precise conditions on $\Om$, $J$ and the Banach spaces involved will be
given in the subsequent sections. Notice that both the exterior value
problems \eqref{eq:Sd} and \eqref{eq:Sn} are ill-posed if the  conditions 
are enforced on $\pOm$. 
The main difficulties in \eqref{eq:dcp} 
and \eqref{eq:ncp} stem from the following facts.

 \begin{enumerate}[$\bullet$]
    \item{\bf Nonlocal operator.}
     The fractional Laplacian $(-\Delta)^s$ is a nonlocal operator. 
     This can be easily seen from the definition of $(-\Delta)^s$ in \eqref{eq11}.
     
    \item {\bf Double nonlocality.}
     The first order optimality conditions for \eqref{eq:dcp} and the Robin exterior
     value problem \eqref{eq:Sn} require to study $\mathcal{N}_s u$ which is the 
     so-called nonlocal-normal derivative of $u$ (see \eqref{NLND}).
     Thus we not only have the nonlocal operator $(-\Delta)^s$ on the domain but also on the 
     exterior $\RR^N\setminus\overline\Om$, i.e., a double nonlocality.  
    
    \item {\bf Exterior conditions in $\RR^N\setminus\Om$ and not boundary conditions on $\pOm$.}
     The conditions in \eqref{eq:Sd} and \eqref{eq:Sn} need to be specified in 
     $\RR^N\setminus \Omega$ instead on $\pOm$ as otherwise the problems 
     \eqref{eq:dcp} and \eqref{eq:ncp} are ill-posed as we have already mentioned above. 
     
    \item {\bf Very-weak solutions of nonlocal exterior value problems.}
     A typical choice for $Z$ is $L^2(\RR^N\setminus\Om)$. As a result, the Dirichlet
     exterior value problem \eqref{eq:Sd}  can only have very-weak solutions 
     (cf.~\cite{apelnicaisepfefferer:2015-1,apelnicaisepfefferer:2015-2,MR2084239} for the 
     case $s=1$). To the best of our knowledge this is the first work that considers the 
     notion of very-weak solutions for nonlocal (fractional) exterior value problems 
     associated with the fractional Laplace operator.     
     
    \item {\bf Regularity of optimization variables.} The standard 
     shift-theorem which holds for local operators such as $\Delta$ does not hold always
     hold for nonlocal operators such as $(-\Delta)^s$ (see for example \cite{Grub}).
 \end{enumerate}
In view of all these aforementioned challenges it is clear that the standard techniques
which are now well established for local problems do not directly extend to the nonlocal
problems investigated in the present paper.

The purpose of this paper is to discuss our approach to deal with these nontrivial 
issues. We emphasize that to the best of our knowledge this is the first work that considers 
the optimal control of problems (source identification problems) \eqref{eq:Sd} and \eqref{eq:Sn} where the control/source is applied 
from the outside.  Let us also mention that 
this notion of controllability of PDEs from the exterior has been introduced by M. Warma in \cite{warma2018approximate} for the nonlocal heat equation associated with 
the fractional Laplacian and in \cite{MW-CLR} for the wave type equation with 
the fractional Laplace operator to study their controllability properties. The case of the strong damping wave equation is included in \cite{WaZa} where some controllability results have been obtained. In case of problems with the spectral fractional Laplacian 
the boundary control has been established in \cite{antil2017fractional}.

We mention that we can also deal with the {\bf fractional Neumann exterior control problem}. That is, given $\xi \ge 0$
 a constant penalty parameter,
 \begin{equation*}
    \min_{(u,z)\in (U_N, Z_N)} J(u) + \frac{\xi}{2} \|z\|^2_{Z_N} ,
 \end{equation*}
 subject to the fractional Neumann exterior value problem: Find $u \in U_N$ solving 
 \begin{equation}\label{eq:SN}
 \begin{cases}
    (-\Delta)^s u +u&= 0 \quad \mbox{in } \Om, \\
    \mathcal{N}_s u  &= z \quad \mbox{in } \RR^N\setminus \Om ,
 \end{cases}                
 \end{equation} 
 and the control constraints 
 \begin{equation*}
    z \in Z_{ad,N} .
 \end{equation*}
 The term $u$ is added in \eqref{eq:SN} just to ensure the uniqueness of solutions. The proofs follow similarly as the two cases we consider in the present paper with very minor changes. Since the paper is already long, we shall not give any details on this case.


Below we mention further the key-novelties of the present paper: 

 \begin{enumerate}[$(i)$]
  \item For the first time, we introduce and study the notion of very-weak solutions to the Dirichlet exterior value problem \eqref{eq:Sd} which is suitable for optimal control problems. We also study weak solutions of the Robin exterior value problem \eqref{eq:Sn}.
        
  \item We approximate the weak solutions of nonhomogeneous Dirichlet exterior value problem by using a suitable Robin 
   exterior value problem. This allows us to circumvent approximating the nonlocal normal derivative.    
   This is a new approach to impose non-zero exterior conditions for
  the  fractional Dirichlet exterior value problem. We refer to an alternative approach 
   \cite{acosta2017finite} 
   where the authors use the Lagrange multipliers to impose nonzero Dirichlet exterior 
   conditions.
   
  \item We study both Dirichlet and Robin exterior control problems. 
  
  \item We approximate (with rate) the Dirichlet exterior control problem by a suitable Robin exterior control problem.
 \end{enumerate}
%

The rest of the paper is organized as follows. We begin with Section~\ref{s:not} where we 
introduce the relevant notations and function spaces. The material in this section is 
well-known. Our main work starts from Section~\ref{s:state} where at first we study
the weak and very-weak solutions for the Dirichlet exterior value problem in Subsection~\ref{s:dbvp}.
This is followed by the well-posedness of the Robin exterior value problem in Subsection~\ref{s:rbvp}. 
The Dirichlet exterior control problem is considered in Section~\ref{s:dcp} and Robin in 
Section~\ref{s:ncp}. We show how to approximate the weak solutions to Dirichlet problem and
the solutions to Dirichlet exterior control problem in Section~\ref{s:DirRob}. 
Subsection~\ref{s:RobinNum} is devoted to the experimental rate of convergence to approximate
the Dirichlet exterior value problem using the Robin problem. In Subsection~\ref{s:source} we consider a source identification problem in the classical sense, however our source is located outside the observation domain where the PDE is satisfied. Subsection~\ref{s:dcpNum} is devoted to two optimal control problems.

 \begin{remark}[\bf Practical aspects]
  {\rm From a practical point of view, having the source/control over the entire $\RR^N\setminus\Om$ 
  can be very expensive. But this can be easily fixed by appropriately describing 
  $Z_{ad}$. Indeed 
  in case of Figure~\ref{f:twodomains} we can set the support of functions in $Z_{ad}$  
  to be in $\widehat{\Om}\setminus\Omega$. 
  }
 \end{remark}

\section{Notation and Preliminaries}\label{s:not}

Unless otherwise stated, $\Om \subset\RR^N$ ($N \ge 1$) is a bounded  open set and $0 < s < 1$. 
We let 
 \[
    W^{s,2}(\Om) := \left\{ u \in L^2(\Om) \;:\; 
            \int_\Om\int_\Om \frac{u(x)-u(y)}{|x-y|^{N+2s}}\;dxdy < \infty \right\} ,
 \]
and we endow it with the norm defined by

 \[
    \|u\|_{W^{s,2}(\Om)} := \left(\int_\Om |u|^2\;dx 
        + \int_\Om\int_\Om  \frac{|u(x)-u(y)|^2}{|x-y|^{N+2s}}\;dxdy \right)^{\frac12}.      
 \]

In order to study \eqref{eq:Sd} we also need to define 
 \[
    W^{s,2}_0(\overline\Om) := \left\{ u \in W^{s,2}(\RR^N) \;:\; u = 0 \mbox{ in } 
            \RR^N\setminus\Om \right\} . 
 \]
Then
\begin{align*}
\|u\|_{W_0^{s,2}(\bOm)}:=\left(\int_{\RR^N}\int_{\RR^N}\frac{|u(x)-u(y)|^2}{|x-y|^{N+2s}}\;dxdy\right)^{\frac 12}
\end{align*}
defines an equivalent norm on $W_0^{s,2}(\bOm)$. 

We shall use $W^{-s,2}(\RR^N)$ and $W^{-s,2}(\overline\Om)$ to denote the dual spaces of $W^{s,2}(\RR^N)$
and $W_0^{s,2}(\overline\Om)$, respectively, and $\langle\cdot,\cdot\rangle$, to denote their
duality pairing whenever it is clear from the context. 
  
We also define the  local fractional order Sobolev space
 \begin{equation}\label{eq:Ws2loc}
    W^{s,2}_{\rm loc}(\Omc) := \left\{ u \in L^2(\Omc) \;:\; u\varphi \in W^{s,2}(\Omc), 
         \ \forall \ \varphi \in \mathcal{D}(\Omc) \right\}.  
 \end{equation}

To introduce the fractional Laplace operator, we let $0<s<1$, and we set 
\begin{equation*}
\mathbb{L}_s^{1}(\RR^N):=\left\{u:\RR^N\rightarrow
\mathbb{R}\;\mbox{
measurable, }\;\int_{\RR^N}\frac{|u(x)|}{(1+|x|)^{N+2s}}%
\;dx<\infty \right\}.
\end{equation*}%
For $u\in \mathbb{L}_s^{1}(\RR^N)$ and $ 
\varepsilon >0$, we let
\begin{equation*}
(-\Delta )_{\varepsilon }^{s}u(x)=C_{N,s}\int_{\{y\in \RR^N,|y-x|>\varepsilon \}}
\frac{u(x)-u(y)}{|x-y|^{N+2s}}dy,\;\;x\in\RR^N,
\end{equation*}%
where the normalized constant $C_{N,s}$ is given by
\begin{equation}\label{CN}
C_{N,s}:=\frac{s2^{2s}\Gamma\left(\frac{2s+N}{2}\right)}{\pi^{\frac
N2}\Gamma(1-s)},
\end{equation}%
and $\Gamma $ is the usual Euler Gamma function (see, e.g. \cite%
{BCF,Caf3,Caf1,Caf2,NPV,War-DN1,War}). The {\bf fractional Laplacian} 
$(-\Delta )^{s}$ is defined for $u\in \mathbb{L}_s^{1}(\RR^N)$  by the formula

\begin{align}
(-\Delta )^{s}u(x)=C_{N,s}\mbox{P.V.}\int_{\RR^N}\frac{u(x)-u(y)}{|x-y|^{N+2s}}dy 
=\lim_{\varepsilon \downarrow 0}(-\Delta )_{\varepsilon
}^{s}u(x),\;\;x\in\RR^N,\label{eq11}
\end{align}%
provided that the limit exists. It has been shown in \cite[Proposition 2.2]{BPS} that for 
$u \in \mathcal{D}(\Om)$, we have that
 \[
    \lim_{s\uparrow 1}\int_{\RR^N} u (-\Delta)^su\;dx 
        = \int_{\RR^N} |\nabla u|^2 dx 
        = - \int_{\RR^N} u \Delta u\;dx 
        = - \int_{\Om} u \Delta u\;dx,
 \]
that is where the constant $C_{N,s}$ plays a crucial role.

Next, for $u \in W^{s,2}(\RR^N)$ we define the nonlocal normal derivative $\mathcal{N}_s$ as:
 \begin{align}\label{NLND}
    \mathcal{N}_s u(x) := C_{N,s} \int_\Om \frac{u(x)-u(y)}{|x-y|^{N+2s}}\;dy, 
            \quad x \in \RR^N \setminus \overline\Om . 
 \end{align}
 We shall call $\mathcal N_s$ the {\em interaction operator}.
Clearly $\mathcal{N}_s$ is a nonlocal operator and it is well defined on $W^{s,2}(\RR^N)$
as we discuss next.

 \begin{lemma}\label{lem:Nmap}
  The interaction operator $\mathcal{N}_s$ maps continuously $W^{s,2}(\RR^N)$ into
 $ W^{s,2}_{\rm loc}(\RR^N\setminus\Om)$. 
  As a result, if  $u \in W^{s,2}(\RR^N)$, then $\mathcal{N}_s u \in L^2(\RR^N\setminus\Om)$.
 \end{lemma}
 
 \begin{proof}
  We refer to \cite[Lemma~3.2]{ghosh2016calder} for the proof of the first part. 
  The second part is a direct consequence of  \eqref{eq:Ws2loc}.
 \end{proof}
 
Despite the fact that $\mathcal{N}_s$ is defined on $\RR^N \setminus \Om$, it
is still known as the ``normal" derivative. This is due to its similarity with the 
classical normal derivative as we shall discuss next.

 \begin{proposition}\label{prop:prop}
 The following assertions hold.
  \begin{enumerate}
    \item {\bf The divergence theorem for $(-\Delta)^s$.}  Let $u\in C_0^2(\RR^N)$, i.e., $C^2$ functions on $\RR^N$ that vanishes at $\pm \infty$. Then
      $$
       \int_\Om (-\Delta)^s u\;dx = -\int_{\RR^N\setminus\Om} \mathcal{N}_s u\;dx.
      $$ 
      
    \item {\bf The integration by parts formula for $(-\Delta)^s$.}   Let $u \in W^{s,2}(\RR^N)$ be such that $(-\Delta)^su \in L^2(\Omega)$. Then for every $v\in W^{s,2}(\RR^N)$ we have that
      \begin{align}\label{Int-Part}
       \frac{C_{N,s}}{2} 
        \int\int_{\RR^{2N}\setminus(\RR^N\setminus\Om)^2} 
         &\frac{(u(x)-u(y))(v(x)-v(y))}{|x-y|^{N+2s}} \;dxdy \notag\\
        &= \int_\Om v(-\Delta)^s u\;dx + \int_{\RR^N\setminus\Om} v\mathcal{N}_s u\;dx ,
      \end{align}
      where $\RR^{2N}\setminus(\RR^N\setminus\Om)^2
       := (\Om\times\Om)\cup(\Om\times(\RR^N\setminus\Om))\cup((\RR^N\setminus\Om)\times\Om)$. 
     
    \item {\bf The limit as $s\uparrow 1$.}    Let $u,v\in C_0^2(\RR^N)$. Then
      $$
       \lim_{s\uparrow 1}\int_{\RR^N\setminus\Om} v \mathcal{N}_s u\;dx 
        = \int_{\pOm} \frac{\partial u}{\partial\nu} v d\sigma . 
      $$  
  \end{enumerate}
 \end{proposition}
 
 \begin{remark} 
  \rm{
  Comparing (a)-(c) in Proposition~\ref{prop:prop} with the classical properties of the standard
  Laplacian $\Delta$ we can immediately infer that $\mathcal{N}_s$ plays the same role for 
  $(-\Delta)^s$ that the classical normal derivative does for $\Delta$. For this reason, we 
  call $\mathcal{N}_s$ the nonlocal normal derivative. 
  }
 \end{remark} 
 
 \begin{proof}[\bf Proof of Proposition~\ref{prop:prop}]
 The proofs of Parts (a) and (c) are contained in \cite[Lemma 3.2]{SDipierro_XRosOton_EValdinoci_2017a} and \cite[Proposition 5.1]{SDipierro_XRosOton_EValdinoci_2017a}, respectively. The proof of Part (b) for smooth functions can be found in \cite[Lemma 3.3]{SDipierro_XRosOton_EValdinoci_2017a}. The version given here is obtained by using a density argument (cf. \cite[Proposition~3.7]{warma2018approximate}).
 \end{proof}

\section{The state equations}\label{s:state}
Before analyzing the optimal control problems \eqref{eq:dcp} and \eqref{eq:ncp}, for a given
function $z$ we shall focus on the Dirichlet \eqref{eq:Sd} and Robin \eqref{eq:Sn} exterior
value problems. We shall assume that $\Om$ is a bounded domain with Lipschitz continuous boundary.

\subsection{The Dirichlet problem for the fractional Laplacian}\label{s:dbvp}

We begin by rewriting the system \eqref{eq:Sd} in a more general form
 \begin{equation}\label{eq:Sd_1}
 \begin{cases}
    (-\Delta)^s u = f \quad &\mbox{in } \Om, \\
                u = z \quad &\mbox{in } \RR^N\setminus \Om .
 \end{cases}                
 \end{equation} 

Here is our notion of weak solutions.

 \begin{definition}[\bf Weak solution to the Dirichlet problem] 
 \label{def:weak_d}
    Let $f \in W^{-s,2}(\overline\Om)$, $z \in W^{s,2}(\RR^N\setminus\Om)$ and 
    $\mathcal{Z} \in W^{s,2}(\RR^N)$ be such that $\mathcal{Z}|_{\RR^N\setminus\Om} = z$. A    
    $u \in W^{s,2}(\RR^N)$ is said to be a weak solution to \eqref{eq:Sd_1} if 
    $u-\mathcal{Z} \in W^{s,2}_0(\overline\Om)$ and 
    \[
     \frac{C_{N,s}}{2} 
        \int_{\RR^N}\int_{\RR^{N}} 
         \frac{(u(x)-u(y))(v(x)-v(y))}{|x-y|^{N+2s}} \;dxdy 
         = \langle f,v \rangle  , 
    \]
    for every $v \in W^{s,2}_0(\overline\Om)$. 
 \end{definition}
 
 Firstly, we notice that since $\Omega$ is assumed to have a Lipschitz continuous boundary, we have that, for $z\in W^{s,2}(\Omc)$, there exists $\mathcal{Z}\in W^{s,2}(\RR^N)$ such that $\mathcal{Z}|_{\Omc}=z$.  Secondly, the existence and uniqueness of a weak solution $u$ to \eqref{eq:Sd_1} and the continuous dependence
of $u$ on the data $f$ and $z$ have been considered in \cite{GGrubb_2015a}, see also \cite{ghosh2016calder,MIVisik_GIEskin_1965a}. More precisely we have the following result.

 \begin{proposition}\label{prop:weak_Dir}
  Let  $f \in W^{-s,2}(\overline\Om)$ and $z \in W^{s,2}(\RR^N \setminus\Om)$. Then 
  there exists a unique weak solution $u$ to 
  \eqref{eq:Sd_1} in the sense of Definition~\ref{def:weak_d}. In addition there is a  constant $C>0$ such that
  \begin{align}\label{Es-DS}
   \|u\|_{W^{s,2}(\RR^N)} 
    \le C \left(\|f\|_{W^{-s,2}(\overline\Om)} 
                + \|z\|_{W^{s,2}(\RR^N \setminus\Om)} \right).
  \end{align}
\end{proposition}

 
Even though such a result is typically sufficient in most situations, nevertheless it is not
directly useful in the current context of optimal control problem \eqref{eq:dcp} since we are 
interested in taking the space $Z_D = L^2(\RR^N\setminus\Om)$. 
Thus we need existence of solution (in some sense) to the fractional Dirichlet
problem \eqref{eq:Sd_1} when the datum $z \in L^2(\RR^N\setminus\Om)$. In order to tackle 
this situation we introduce the notion of very-weak solutions for \eqref{eq:Sd_1}.

 \begin{definition}[\bf Very-weak solution to the Dirichlet problem] 
 \label{def:vweak_d}
    Let $z \in L^2(\RR^N\setminus\Om)$ and $f \in W^{-s,2}(\overline\Om)$. 
    A  $u \in L^2(\RR^N)$ is said to be a very-weak solution to \eqref{eq:Sd_1} if the identity
    \begin{equation}\label{eq:vw_d}
      \int_{\Om} 
         u (-\Delta)^s v\;dx
         = \langle f,v \rangle - \int_{\RR^N\setminus\Om} z \mathcal{N}_s v\;dx,
    \end{equation}
   holds for every $v \in V := \{v \in W^{s,2}_0(\overline\Om) \;:\; (-\Delta)^s v
    \in L^2(\Om) \}$. 
 \end{definition}

Next we prove the existence and uniqueness of a very-weak solution to \eqref{eq:Sd_1} in the sense
of Definition~\ref{def:vweak_d}.

 \begin{theorem}\label{thm:vwdexist}
  Let $f \in W^{-s,2}(\overline\Om)$ and $z \in L^2(\RR^N\setminus\Om)$. Then there exists 
  a unique very-weak solution $u$ to 
  \eqref{eq:Sd_1} according to Definition~\ref{def:weak_d} that fulfills
  \begin{align}\label{VWS_EST}
   \|u\|_{L^2(\Om)} \le C\left(\|f\|_{W^{-s,2}(\overline\Om)} 
       + \|z\|_{L^2(\RR^N \setminus\Om)} \right),
  \end{align}
for a constant $C >0$. In addition, if $z\in W^{s,2}(\Omc)$, then the following assertions hold.
\begin{enumerate}
\item Every weak solution of \eqref{eq:Sd_1} is also a very-weak solution.
\item Every very-weak solution of \eqref{eq:Sd_1} that belongs to $W^{s,2}(\RR^N)$ is also a weak solution.
\end{enumerate}
 \end{theorem}
 
 \begin{proof}
  In order to show the existence of a very-weak solution we shall apply the 
  Babu\v{s}ka-Lax-Milgram theorem. 
  
 Firstly, let $(-\Delta)_D^s$ be the realization of $(-\Delta)^s$ in $L^2(\Omega)$ with the zero Dirichlet exterior condition $u=0$ in $\RR^N\setminus\Omega$. More precisely,
  \begin{align*}
  D((-\Delta)_D^s)=V\;\mbox{ and }\; (-\Delta)_D^su=(-\Delta)^su.
  \end{align*}
 Then a norm on $V$ is given by $\|v\|_V = \|(-\Delta)_D^s v\|_{L^2(\Om)}$ 
  which follows from the fact that the operator $(-\Delta)_D^s$ is invertible (since by \cite{SV2} $(-\Delta)_D^s$ has a compact resolvent and its first eigenvalue is strictly positive).
  Secondly, let $\mathcal F$ be the bilinear form defined on $L^2(\Omega)\times V$ by
  \begin{align*}
  \mathcal F(u,v):=\int_{\Om}  u (-\Delta)^s v\;dx.
  \end{align*}
  Then $\mathcal F$ is clearly bounded 
  on $L^2(\Om) \times V$. More precisely there is a constant $C>0$ such that
  \begin{align*}
 \left|\mathcal F(u,v)\right|\le \|u\|_{L^2(\Omega)}\|(-\Delta)^sv\|_{L^2(\Omega}\le C\|u\|_{L^2(\Omega)}\|v\|_{V}.
  \end{align*}
 Thirdly, we show the inf-sup conditions. From the definition
  of $V$, it immediately follows that 
  \begin{align*}
  v\in W_0^{s,2}(\bOm)\;\;\mbox{ and }\; (-\Delta)^s v \in L^2(\Om)  \quad 
   \Longleftrightarrow\quad v \in V . 
  \end{align*}
 By setting $u := \frac{(-\Delta)_D^s v}{\|(-\Delta)_D^sv\|_{L^2(\Om)}} \in L^2(\Om)$, we obtain that
  \begin{align*}
   \sup_{u \in L^2(\Om) , \|u\|_{L^2(\Om)}=1} |(u,(-\Delta)_D^sv)_{L^2(\Om)}|
    &\ge \frac{|((-\Delta)_D^sv,(-\Delta)_D^sv)_{L^2(\Om)}|}{\|(-\Delta)_D^sv\|_{L^2(\Om)}} \\
    &= \|(-\Delta)_D^sv\|_{L^2(\Om)} = \|v\|_V . 
  \end{align*}
  Next we choose $v \in V$ as the unique weak solution of $(-\Delta)_D^sv = u/\|u\|_{L^2(\Om)}$
  for some $0\ne u \in L^2(\Om)$. Then we readily obtain that 
  \[
   \sup_{v\in V, \|v\|_{V}=1} |(u,(-\Delta)^sv)_{L^2(\Om)}| 
    \ge \frac{|(u,u)_{L^2(\Om)}|}{\|u\|_{L^2(\Om)}} = \|u\|_{L^2(\Om)} >0, 
  \]
  for all $0\ne u\in L^2(\Omega)$.
Finally, we have to show that the right-hand-side in \eqref{eq:vw_d}
  defines a linear continuous functional on $V$. Indeed, applying the H\"older inequality in conjunction
  with Lemma~\ref{lem:Nmap} we obtain that there is a constant $C>0$ such that
  \begin{align}\label{NOR-EST}
   \left|\int_{\RR^N\setminus\Om} z \mathcal{N}_s v\;dx \right|
   \le \|z\|_{L^2(\RR^N\setminus\Om)} \|\mathcal{N}_s v\|_{L^2(\RR^N\setminus\Om)} 
   \le C\|z\|_{L^2(\RR^N\setminus\Om)} \|v\|_{W^{s,2}_0(\overline\Om)} , 
  \end{align}
  where in the last step we have used the fact that 
  $\|v\|_{W_0^{s,2}(\overline\Om)} = \|v\|_{W^{s,2}(\RR^N)}$ for 
  $v \in W_0^{s,2}(\overline\Om)$. Moreover 
  \[
   |\langle f,v\rangle| \le \|f\|_{W^{-s,2}(\overline\Om)} \|v\|_{W^{s,2}_0(\overline\Om)}. 
  \]  
  In view of the last two estimates, the right-hand-side in \eqref{eq:vw_d}
  defines a linear continuous functional on $V$. Therefore all the requirements of the Babu\v{s}ka-Lax-Milgram 
  theorem holds. Thus, there exists a unique $u \in L^2(\Om)$ satisfying \eqref{eq:vw_d}. Let $u=z$ in $\Omc$, then $u\in L^2(\RR^N)$ and satisfies \eqref{eq:vw_d}. We have shown the existence of the uniqueness of a very-weak solution.. 
  
Next we show the estimate \eqref{VWS_EST}. Let $u\in L^2(\RR^N)$ be a very-weak solution. Let $v\in V$ be a solution of $(-\Delta)_D^sv=u$. Taking this $v$ as a test function in \eqref{eq:vw_d} and using \eqref{NOR-EST}, we get that there is a constant $C>0$ such that
 \begin{align*}
 \|u\|_{L^2(\Omega)}^2\le& \|f\|_{W^{-s,2}(\overline\Om)}\|v\|_{W_0^{s,2}(\overline\Om)}+\|z\|_{L^2(\Omc)}\|\mathcal N_sv\|_{L^2(\Omc)}\\
 \le &C\left( \|f\|_{W^{-s,2}(\overline\Om)}+\|z\|_{L^2(\Omc)}\right)\|v\|_{W_0^{s,2}(\overline\Om)}\\
 \le &C\left( \|f\|_{W^{-s,2}(\overline\Om)}+\|z\|_{L^2(\Omc)}\right)\|(-\Delta)_D^sv\|_{L^2(\Om)}\\
 \le &C\left( \|f\|_{W^{-s,2}(\overline\Om)}+\|z\|_{L^2(\Omc)}\right)\|u\|_{L^2(\Om)},
 \end{align*}
and we have shown \eqref{VWS_EST}. This completes the proof of the first part.  

Next we prove the last two assertions of the theorem. Assume that  $z\in W^{s,2}(\Omc)$.
  
 (a) Let $u\in W^{s,2}(\RR^N)\hookrightarrow L^2(\RR^N)$ be a weak solution of \eqref{eq:Sd_1}. It follows from the definition that $u=z$ in $\RR^N\setminus\Omega$ and 
\begin{align}\label{e1}
\frac{C_{N,s}}{2}\int_{\RR^N}\int_{\RR^N}\frac{(u(x)-u(y))(v(x)-v(y))}{|x-y|^{N+2s}}\;dxdy=\langle f,v\rangle,
\end{align}
for every $v\in V$.
Since $v=0$ in $\RR^N\setminus\Omega$, we have that
\begin{align}\label{e2}
\int_{\RR^N}\int_{\RR^N}&\frac{(u(x)-u(y))(v(x)-v(y))}{|x-y|^{N+2s}}\;dxdy\notag\\
&=\int\int_{\RR^{2N}\setminus(\RR^N\setminus\Omega)^2}\frac{(u(x)-u(y))(v(x)-v(y))}{|x-y|^{N+2s}}\;dxdy.
\end{align}
Using \eqref{e1}, \eqref{e2}, the integration by parts formula \eqref{Int-Part} together with the fact that $u=z$ in $\RR^N\setminus\Omega$, we get that
\begin{align*}
\frac{C_{N,s}}{2}\int_{\RR^N}\int_{\RR^N}&\frac{(u(x)-u(y))(v(x)-v(y))}{|x-y|^{N+2s}}\;dxdy\\
&=\langle f,v\rangle\\
&=\int_{\Omega}u(-\Delta)^sv\;dx+\int_{\RR^N\setminus\Omega}u\mathcal N_sv\;dx\\
&=\int_{\Omega}u(-\Delta)^sv\;dx+\int_{\RR^N\setminus\Omega}z\mathcal N_sv\;dx.
\end{align*}
Thus $u$ is a very-weak solution of \eqref{eq:Sd_1}.

(b) Finally let $u$ be a very-weak solution of \eqref{eq:Sd_1} and assume that $u\in W^{s,2}(\RR^N)$. Since $u=z$ in $\RR^N\setminus\Omega$, we have that $z\in W^{s,2}(\RR^N\setminus\Omega)$ and if $\mathcal Z\in W^{s,2}(\RR^N)$ satisfies $\mathcal Z|_{\RR^N\setminus\Omega}=z$, then clearly $(u-\mathcal Z)\in W_0^{s,2}(\bOm)$. Since $u$ is a very-weak solution of \eqref{eq:Sd_1}, then by definition, for every $v\in V=D((-\Delta)_D^s)$, we have that
\begin{align}\label{e3}
\int_{\Omega}u(-\Delta)^sv\;dx=\langle f,v\rangle-\int_{\RR^N\setminus\Omega}z\mathcal N_sv\;dx.
\end{align}
 Since $u\in W^{s,2}(\RR^N)$ and $v=0$ in $\RR^N\setminus\Omega$, then using  \eqref{Int-Part} again we get that
\begin{align}\label{e4}
\int_{\RR^N}\int_{\RR^N}&\frac{(u(x)-u(y))(v(x)-v(y))}{|x-y|^{N+2s}}\;dxdy\notag\\
&=\int\int_{\RR^{2N}\setminus(\RR^N\setminus\Omega)^2}\frac{(u(x)-u(y))(v(x)-v(y))}{|x-y|^{N+2s}}\;dxdy\notag\\
&=\int_{\Omega}u(-\Delta)^sv\;dx+\int_{\RR^N\setminus\Omega}u\mathcal N_sv\;dx\notag\\
&=\int_{\Omega}u(-\Delta)^sv\;dx+\int_{\RR^N\setminus\Omega}z\mathcal N_sv\;dx. 
\end{align}
It follows from \eqref{e3} and \eqref{e4} that for every $v\in V$, we have that
\begin{align}\label{e5}
\int_{\RR^N}\int_{\RR^N}&\frac{(u(x)-u(y))(v(x)-v(y))}{|x-y|^{N+2s}}\;dxdy=\langle f,v\rangle.
\end{align}  
Since $V$ is dense in $W_0^{s,2}(\bOm)$, we have that \eqref{e5} remains true for every $v\in W_0^{s,2}(\bOm)$. We have shown that $u$ is a weak solution of \eqref{eq:Sd_1} and the proof is finished.     
 \end{proof}

\subsection{The Robin problem for the fractional Laplacian}\label{s:rbvp}

In order to study the Robin problem \eqref{eq:Sn} we consider the Sobolev space
 introduced in \cite{SDipierro_XRosOton_EValdinoci_2017a}. For $g\in L^1(\RR^N\setminus\Omega)$ fixed, we let

\begin{align*}
W_{\Omega,g}^{s,2}:=\Big\{u:\RR^N\to\RR\;\mbox{ measurable, }\,\|u\|_{W_{\Omega,g}^{s,2}}<\infty\Big\}, 
\end{align*}
where
\begin{align}\label{norm}
\|u\|_{W_{\Omega,g}^{s,2}}:=\left(\|u\|_{L^2(\Omega)}^2+\||g|^{\frac 12}u\|_{L^2(\RR^N\setminus\Omega)}^2+\int\int_{\RR^{2N}\setminus(\RR^N\setminus\Omega)^2}\frac{|u(x)-u(y)|^2}{|x-y|^{N+2s}}dxdy\right)^{\frac 12}.
\end{align}
 Let $\mu$ be the measure on $\RR^N\setminus\Omega$ given by $d\mu=|g|dx$.
With this setting, the norm in \eqref{norm} can be rewritten as 
\begin{align}\label{norm-e}
\|u\|_{W_{\Omega,g}^{s,2}}:=\left(\|u\|_{L^2(\Omega)}^2+\|u\|_{L^2(\RR^N\setminus\Omega,\mu)}^2+\int\int_{\RR^{2N}\setminus(\RR^N\setminus\Omega)^2}\frac{|u(x)-u(y)|^2}{|x-y|^{N+2s}}dxdy\right)^{\frac 12}.
\end{align}

If $g=0$, we shall let $W_{\Omega,0}^{s,2}=W_{\Omega}^{s,2}$.
The following result has been proved in \cite[Proposition 3.1]{SDipierro_XRosOton_EValdinoci_2017a}.

\begin{proposition}
Let $g\in L^1(\RR^N\setminus\Omega)$. Then $W_{\Omega,g}^{s,2}$ is a Hilbert space.
\end{proposition}

Throughout the remainder of the article, for $g\in L^1(\Omc)$, we shall denote by $(W^{s,2}_{\Om,g})^\star$ the dual of $ W^{s,2}_{\Om,g}$. 

\begin{remark}\label{rem-35}
{\em We mention the following facts.
\begin{enumerate}
\item Recall that 
\begin{align*}
\RR^{2N}\setminus(\RR^N\setminus\Omega)^2=(\Omega\times\Omega)\cup (\Omega\times(\RR^N\setminus\Omega))\cup((\RR^N\setminus\Omega)\times\Omega),
\end{align*}
so that
\begin{align}\label{norm-just}
\int\int_{\RR^{2N}\setminus(\RR^N\setminus\Omega)^2}&\frac{|u(x)-u(y)|^2}{|x-y|^{N+2s}}dxdy=
\int_{\Omega}\int_{\Omega}\frac{|u(x)-u(y)|^2}{|x-y|^{N+2s}}dxdy\notag\\
&+\int_{\Omega}\int_{\RR^N\setminus\Omega}\frac{|u(x)-u(y)|^2}{|x-y|^{N+2s}}dxdy
+\int_{\RR^N\setminus\Omega}\int_{\Omega}\frac{|u(x)-u(y)|^2}{|x-y|^{N+2s}}dxdy.
\end{align}

\item If $g\in L^1(\Omc)$ and $u\in W_{\Omega,g}^{s,2}$, then using the H\"older inequality we get that
\begin{align}\label{E-g}
\left|\int_{\Omc}gu\;dx\right|\le& \int_{\Omc}|g|^{\frac 12}| |g^{\frac 12}| u|\;dx
\le \left(\int_{\Omc}|g|\;dx\right)^{\frac 12}\left(\int_{\Omc}|gu^2|\;dx\right)^{\frac 12}\notag\\
\le&\|g\|_{L^1(\Omc)}^{\frac 12}\|u\|_{L^2(\Omc,\mu)}\le \|g\|_{L^1(\Omc)}^{\frac 12}\|u\|_{W_{\Omega,g}^{s,2}}.
\end{align}
It follows from \eqref{E-g} that in particular, $L^1(\Omc,\mu)\hookrightarrow (W_{\Omega,g}^{s,2})^\star$.

\item By definition (using also \eqref{norm-just}),  $W_{\Om,g}^{s,2}\hookrightarrow W_{\Om}^{s,2} \hookrightarrow W^{s,2}(\Omega)$, so that we have the following continuous embedding
\begin{align}\label{sobo-g}
W_{\Om,g}^{s,2}\hookrightarrow W_{\Om}^{s,2} \hookrightarrow L^{\frac{2N}{N-2s}}(\Omega).
\end{align}
It follows from \eqref{sobo-g} that the embedding $W_{\Om,g}^{s,2}\hookrightarrow L^{2}(\Omega)$ and $W_{\Om}^{s,2}\hookrightarrow L^{2}(\Omega)$ are  compact.
\end{enumerate}
}
\end{remark}

We consider a generalized version of the  system \eqref{eq:Sn} with nonzero right-hand-side $f$. 
Throughout the following section, the measure $\mu$ is defined with $g$ replaced by $\kappa$. That is, $d\mu=\kappa dx$ (recall that $\kappa$ is assumed to be non-negative).
Here is our notion of weak solutions. 

 \begin{definition}\label{def:weak_n}
  Let  $z \in L^2(\RR^N\setminus\Om,\mu)$ and $f\in ( W_{\Om,\kappa}^{s,2})^\star$.
 A $u\in W_{\Om,\kappa}^{s,2}$ is said to be a weak solution of \eqref{eq:Sn} if the identity
\begin{align}\label{we-so}
\int\int_{\RR^{2N}\setminus(\Omc)^2}&\frac{(u(x)-u(y))(v(x)-v(y))}{|x-y|^{N+2s}}\;dxdy+\int_{\Omc}\kappa uv\;dx\notag\\
&=\langle f,v\rangle_{ ( W_{\Om,\kappa}^{s,2})^\star, W_{\Om,\kappa}^{s,2}}+\int_{\Omc}\kappa zv\;dx,
\end{align}
holds for every $v\in  W_{\Om,\kappa}^{s,2}$.
 \end{definition}

We have the following existence result.

\begin{proposition}\label{pro-sol-Ro}
Let $\kappa\in L^1(\Omc)\cap L^\infty(\Omc)$.
Then for every $z\in L^2(\RR^N\setminus\Omega,\mu)$ and $f\in ( W_{\Om,\kappa}^{s,2})^\star$, there exists a  weak solution $u\in W_{\Om,\kappa}^{s,2}$ of \eqref{eq:Sn}. 
\end{proposition}

\begin{proof}
Let  $\mathcal E$ with domain $D(\mathcal E)=W_{\Om,\kappa}^{s,2}$ be given by
\begin{align}\label{form-F}
\mathcal E(u,v):=\int\int_{\RR^{2N}\setminus(\Omc)^2}\frac{(u(x)-u(y))(v(x)-v(y))}{|x-y|^{N+2s}}\;dxdy
+\int_{\Omc}\kappa uv\;dx.
\end{align}
Then $\mathcal E$ is a bilinear, symmetric, continuous and closed form on $L^2(\Omega)$.  Hence, for every $z\in L^2(\Omc,\mu)\subset (W_{\Om,\kappa}^{s,2})^\star$ and $f\in (W_{\Om,\kappa}^{s,2})^\star$, there is a function $u\in W_{\Om,\kappa}^{s,2}$ such that 
\begin{align*}
\mathcal E(u,v)=&\langle f,v\rangle_{(W_{\Om,\kappa}^{s,2})^\star,W_{\Om,\kappa}^{s,2}}+\langle z,v\rangle_{(W_{\Om,\kappa}^{s,2})^\star,W_{\Om,\kappa}^{s,2}}\\
=&\langle f,v\rangle_{(W_{\Om,\kappa}^{s,2})^\star,W_{\Om,\kappa}^{s,2}}+\int_{\Omc}\kappa zv\;dx, 
\end{align*}
for every $v\in W_{\Om,\kappa}^{s,2}$. That is, $u$ satisfies \eqref{we-so}. Thus $u$ is a weak solution of \eqref{eq:Sn}. 
The proof is finished.
\end{proof}
\begin{remark}\label{rem:nonunique}
{\rm Notice that similarly to the classical Neumann problem when $\kappa \equiv 0$, Proposition~\ref{pro-sol-Ro} only guarantees uniqueness of solutions to \eqref{eq:Sn} up to a constant. In case
we assume $\kappa$ to be strictly positive, uniqueness can be guaranteed under Assumption~\ref{asum} below. In that case we can also show that there is a constant $C>0$ such that
\begin{align}\label{Est-sol-RB}
\|u\|_{W_{\Om,\kappa}^{s,2}}\le C\left(\|f\|_{ ( W_{\Om,\kappa}^{s,2})^\star}+\|z\|_{L^2(\Omc,\mu)}\right).
\end{align}
}
\end{remark}

\section{Fractional Dirichlet boundary control problem}\label{s:dcp}

We begin by introducing the appropriate function spaces needed to study \eqref{eq:dcp}. 
We let 
 \[
    Z_D := L^2(\RR^N\setminus\Om) , \quad U_D := L^2(\Om)  .
 \]
In view of Theorem~\ref{thm:vwdexist} the following (solution-map) control-to-state
map 
 \begin{align*}
    S : Z_D \rightarrow U_D,\;\;z \mapsto S z = u ,
 \end{align*}
is well-defined, linear and continuous. We also notice that for $z\in Z_D$, we have that $u:=Sz\in L^2(\RR^N)$.
As a result we can write the \emph{reduced fractional
Dirichlet exterior control problem} as follows:

 \begin{equation}\label{eq:rpd}
    \min_{z\in Z_{ad,D}} \mathcal{J}(z) 
       := J(Sz) + \frac{\xi}{2} \|z\|_{Z_D}^2 . 
 \end{equation}
We then have the following well-posedness result for \eqref{eq:rpd} and equivalently
\eqref{eq:dcp}. 

 \begin{theorem}\label{thm:exit_dcp}
  Let $Z_{ad,D}$ be a closed and convex subset of $Z_D$. 
  Let either $\xi > 0$ or $Z_{ad,D}$ be bounded and 
  let $J : U_D \rightarrow \mathbb{R}$ be weakly lower-semicontinuous. Then  
  there exists a solution $\bar{z}$ to \eqref{eq:rpd} and equivalently to 
  \eqref{eq:dcp}. If either $J$ is convex and $\xi > 0$ or $J$ is strictly
  convex and $\xi \ge 0$, then $\bar{z}$ is unique. 
 \end{theorem}
 
 \begin{proof}
  The proof uses the so-called direct-method or the Weierstrass theorem   
  \cite[Theorem~3.2.1]{HAttouch_GButtazzo_GMichaille_2014a}. We notice for  
  $\mathcal{J} : Z_{ad,D} \rightarrow \RR$, it is always possible to construct a 
  minimizing sequence $\{z_n\}_{n\in\NN}$ (cf. \cite[Theorem~3.2.1] {HAttouch_GButtazzo_GMichaille_2014a} for a construction) such that 
  \[
    \inf_{z\in Z_{ad,D}} \mathcal{J}(z) = \lim_{n\rightarrow\infty} \mathcal{J}(z_n) .
  \]
  If $\xi>0$ or $Z_{ad,D} \subset Z_D$ is bounded then $\{z_n\}_{n\in\NN}$ is a bounded 
  sequence in $Z_D$ which is a Hilbert space. Due to the reflexivity of $Z_D$ 
  we have that (up to a subsequence if necessary) $z_n\rightharpoonup\bar{z}$ (weak convergence) in 
  $Z_D$ as $n\to\infty$. Since $Z_{ad,D}$ is closed and convex, hence is weakly closed, 
  we have that $\bar{z} \in Z_{ad,D}$. 
  
  Since $S : Z_{ad,D} \rightarrow U_D$ is linear and continuous, we have that 
  it is weakly continuous. This implies that  $Sz_n \rightharpoonup S\bar{z}$ in $U_D$ as $n\to\infty$. 
  We have to show that $(S\bar{z},\bar{z})$ 
  fulfills the state equation according to Definition~\ref{def:vweak_d}. In particular 
  we need to study the identity 
  \begin{equation}\label{eq:exit_d_1}
      \int_{\Om} 
         u_n (-\Delta)^s v\;dx
         =- \int_{\RR^N\setminus\Om} z_n \mathcal{N}_s v\;dx , \quad 
                 \forall v \in V ,
  \end{equation}   
  as $n\rightarrow\infty$, where $u_n := Sz_n$. Since $u_n \rightharpoonup S\bar{z} =: \bar{u}$
  in $U_D$ as $n\rightarrow\infty$  and $z_n\rightharpoonup \bar{z}$ in $Z_D$
  as $n\rightarrow\infty$, we can immediately take the limit in \eqref{eq:exit_d_1} and 
  obtain that $(\bar{u},\bar{z}) \in U_D \times Z_{ad,D}$ fulfills the state equation in 
  the sense of Definition~\ref{def:vweak_d}. 
  
  It then remains to show that $\bar{z}$ is the minimizer of \eqref{eq:rpd}. This is a 
  consequence of the fact that $\mathcal{J}$ is weakly lower semicontinuous. Indeed, 
  $\mathcal{J}$ is the sum of two weakly lower semicontinuous functions ($\|\cdot\|^2_{Z_D}$ is 
  continuous and convex therefore weakly lower semicontinuous). 
  
  Finally, if $\xi > 0$ and $J$ is convex then $\mathcal{J}$ is strictly convex 
  (sum of a strictly convex and convex functions). On the other hand, if 
  $J$ is strictly convex then $\mathcal{J}$ is strictly convex. 
  In either case we have that $\mathcal{J}$ is strictly convex and thus the uniqueness
  of $\bar{z}$ follows. 
 \end{proof}

We will next derive the first order necessary optimality conditions for \eqref{eq:rpd}. 
We begin by identifying the structure of the adjoint operator $S^*$. 

 \begin{lemma}\label{lem:adjS_d}
  For the state equation \eqref{eq:Sd} the adjoint operator 
  $S^* : U_D \rightarrow Z_D$ is given by 
  \[
   S^* w =- \mathcal{N}_s p \in Z_D,
  \]
  where $w \in U_D$ and $p \in W_0^{s,2}(\overline\Om)$ is the weak solution 
  to the problem
  \begin{equation}\label{eq:adj_state_d}
  \begin{cases} 
   (-\Delta)^s p &= w \quad \mbox{ in } \Om \\
               p &= 0 \quad \mbox{ in } \RR^N \setminus \Om .
  \end{cases} 
  \end{equation}
 \end{lemma}
 
 \begin{proof}
  According to the definition of $S^*$ we have that for every $w \in U_D$ and 
  $z \in Z_D$,
  \[
   (w,Sz)_{L^2(\Om)} = (S^*w,z)_{L^2(\RR^N\setminus\Om)} . 
  \] 
  Next, testing the adjoint equation \eqref{eq:adj_state_d} with $Sz$ and using the fact that $Sz$ is a very-weak solution of \eqref{eq:Sd_1} with $f=0$, we arrive at 
  \begin{align*}
   (w,Sz)_{L^2(\Om)} = ( Sz, (-\Delta)^s p )_{L^2(\Om)} 
     =  -(z,\mathcal{N}_s p)_{L^2(\RR^N\setminus\Om)} = (z,S^*w)_{L^2(\RR^N\setminus\Om)} .
  \end{align*}
  This yields the asserted result. 
 \end{proof}

For the remainder of this section we will assume that $\xi > 0$.

 \begin{theorem}\label{thm:focD}
  Let the assumptions of Theorem~\ref{thm:exit_dcp} hold. Let $\mathcal{Z}$ be an open 
  set in $Z_D$ such that $Z_{ad,D} \subset \mathcal{Z}$. Let 
  $u \mapsto J(u) : U_D \rightarrow \RR$ be continuously Fr\'echet differentiable with 
  $J'(u) \in U_D$. If $\bar{z}$ is a minimizer of \eqref{eq:rpd}
  over $Z_{ad,D}$, then the first order necessary optimality conditions are given by 
  \begin{equation}\label{eq:foc_dcp}
   \left( -\mathcal{N}_s\bar{p}+\xi\bar{z}, z-\bar{z}\right)_{L^2(\RR^N\setminus\Om)} \ge 0 , \quad z \in Z_{ad,D} 
  \end{equation}
  where $\bar{p} \in W_0^{s,2}(\overline\Om)$ solves the adjoint equation 
  \begin{equation}\label{eq:adj_state_dp}
  \begin{cases} 
   (-\Delta)^s \bar{p} &= J'(\bar{u}) \quad \mbox{ in } \Om \\
               \bar{p} &= 0 \quad \mbox{ in } \RR^N \setminus \Om .
  \end{cases} 
  \end{equation}
  Equivalently we can write \eqref{eq:foc_dcp} as 
  \begin{equation}\label{eq:projf_d}
   \bar{z} = \mathcal{P}_{Z_{ad}}\left(\frac{1}{\xi}\mathcal{N}_s\bar{p}\right) , 
  \end{equation}
  where $\mathcal{P}_{Z_{ad}}$ is the projection onto 
  the set $Z_{ad}$. If $J$ is convex then \eqref{eq:foc_dcp} is a sufficient condition. 
 \end{theorem}
 
 \begin{proof}
  The proof is a straightforward application of differentiability properties of $J$  
  and chain rule in conjunction with Lemma~\ref{lem:adjS_d}. Indeed for a given direction
  $h \in Z_{ad,D}$ we have that the directional derivative of $\mathcal{J}$ is given by 
  
  \begin{align*}
   \mathcal{J}'(\bar{z}) h =& (J'(S\bar{z}),Sh)_{L^2(\Om)}  + \xi (\bar{z},h)_{L^2(\RR^N\setminus\Om)} \\
  = &(S^*J'(S\bar{z}),h)_{L^2(\Om)} + \xi (\bar{z},h)_{L^2(\RR^N\setminus\Om)}, 
 \end{align*}
  where in the first step we have used that $J'(S\bar{z}) \in \mathcal{L}(L^2(\Om),\RR) 
  = L^2(\Om)$ and in the second step we have used that $S$ is linear and bounded therefore 
  $S^*$ is well-defined. Then using Lemma~\ref{lem:adjS_d} we arrive at the asserted result. 
   
  From Lemma~\ref{lem:Nmap} we recall that $\mathcal{N}_s\bar{p} \in L^2(\RR^N\setminus\Om)$.
  Therefore the equivalence between \eqref{eq:foc_dcp} and \eqref{eq:projf_d} follows by
  using \cite[Theorem~3.3.5]{HAttouch_GButtazzo_GMichaille_2014a}.
 \end{proof}

 \begin{remark}[\bf Regularity for optimization variables]\label{rem:ocdreg}
  \rm{
   We recall a rather surprising result for the adjoint equation 
   \eqref{eq:adj_state_d}. The standard shift argument that is known to hold for the classical 
   Laplacian on smooth open sets does not hold in the case of the fractional Laplacian i.e.,
   $p$ does not always belong to $W^{2s,2}(\Om)$. More precisely assume that $\Omega$ is a smooth bounded open set.  If $0<s<\frac 12$, then by \cite[Formula (7.4)]{Grub} we have that $D((-\Delta)_D^s)= W_0^{2s,2}(\bOm)$ and hence, $p\in W^{2s,2}(\Om)$ in that case. But if $\frac 12\le s<1$, an example has been given in \cite[Remark 7.2]{RS-DP} where $D((-\Delta)_D^s)\not\subset W^{2s,2}(\Omega)$, thus in that case $p$ does not always belong to $W^{2s,2}(\Omega)$.
   As a result, the best known result 
   for  $\mathcal{N}_s p$ is as given in Lemma~\ref{lem:Nmap}. Since 
   $\mathcal{P}_{Z_{ad}}$ is a contraction (Lipschitz) we can conclude that $\bar{z}$ has
   the same regularity as $\mathcal{N}_s \bar{p}$, i.e., they are in $L^2(\RR^N\setminus\Om)$ globally 
   and in $W^{s,2}_{\rm loc}(\RR^N\setminus\Om)$ locally. As it is well-known, in case of 
   the classical Laplacian, one can use a boot-strap argument to improve the regularity of 
   $S\bar z=\bar{u}$. However this is not the case for fractional exterior value problems.
  }
 \end{remark}

\section{Fractional Robin exterior control problem}\label{s:ncp}
In this section we shall study the fractional Robin exterior control problem 
\eqref{eq:Sn}. We begin by setting the functional analytic framework. We let  
 \[
    Z _R:= L^2(\RR^N\setminus\Om,\mu), 
            \quad U_R := W^{s,2}_{\Om,\kappa} .
 \]
Notice that $d\mu = \kappa dx$.  In addition we shall assume that $\kappa\in L^1(\Omc)\cap L^\infty(\Omc)$ and $\kappa > 0$ a.e. in $\Omc$.
In view of Proposition~\ref{pro-sol-Ro} the following (solution-map) 
control-to-state map 

 \[
  S : Z_R \rightarrow U_R,\;\;\;z\mapsto u,
 \]
is well-defined. Moreover $S$ is linear and continuous (by \eqref{Est-sol-RB}).  Since $U_R \hookrightarrow L^2(\Om)$ with the embedding being continuous we can instead define 
 \[
  S : Z_R \rightarrow L^2(\Om). 
 \]
We can then write the so-called 
\emph{reduced fractional Robin exterior control problem}
 \begin{equation}\label{eq:rpn}
    \min_{z\in Z_{ad,R}} \mathcal{J}(z) 
       := J(Sz) + \frac{\xi}{2} \|z\|_{L^2(\RR^N\setminus\Om,\mu)}^2 . 
 \end{equation}
 We have the following well-posedness result.
 
 \begin{theorem}\label{thm:exit_ncp}
  Let $Z_{ad,R}$ be a closed and convex subset of $Z_R$. 
  Let either $\xi > 0$ or $Z_{ad,R} \subset Z_R$ be bounded.
  Moreover, let 
  $J : L^2(\Om) \rightarrow \mathbb{R}$ be weakly lower-semicontinuous. Then there 
  exists a solution $\bar{z}$ to \eqref{eq:rpn} and equivalently to \eqref{eq:ncp}. 
  If either $J$ is convex and $\xi > 0$ or $J$ is strictly
  convex and $\xi \ge 0$ then $\bar{z}$ is unique. 
 \end{theorem}
 
 \begin{proof}
We proceed as the proof of Theorem~\ref{thm:exit_dcp}. 
Let  $\{z_n\}_{n\in\NN}\subset Z_{ad}$ be a minimizing sequence such that
\begin{align*}
 \inf_{z\in Z_{ad,R}} \mathcal{J}(z) = \lim_{n\rightarrow\infty} \mathcal{J}(z_n) .
\end{align*}
If $\xi > 0$ or 
$Z_{ad,R} \subset Z_R$ is bounded then
after a subsequence, if necessary,  we have $z_n \rightharpoonup \bar{z}$ in $L^2(\RR^N\setminus\Om,\mu)$ as $n\to\infty$. 
Now since $Z_{ad,R}$ is a convex and closed subset of  $Z_R$, it follows that  $\bar{z} \in Z_{ad,R}$.
  
 Next we show that the pair $(S\bar z,\bar z)$ satisfies the state equation. Notice that $u_n:=Sz_n$ is the weak solution of \eqref{eq:Sn} with boundary value $z_n$. Thus, by definition,  $u_n\in  W_{\Om,\kappa}^{s,2}$ and the identity
  \begin{align}\label{We-so-n}
\mathcal E(u_n,v)=\int_{\Omc} z_nv\;d\mu,
\end{align}
holds for every $v\in  W_{\Om,\kappa}^{s,2}$ and where we recall that $\mathcal E$ is given in \eqref{form-F}. We also notice that the mapping $S$ is also bounded  from $Z_R$ into $W_{\Om,\kappa}^{s,2}$ (by \eqref{Est-sol-RB}). This shows that the sequence $\{u_n\}_{n\in\NN}$ is bounded in $W_{\Om,\kappa}^{s,2}$. Thus, after a subsquence, if necessary, we have that $Sz_n=u_n \rightharpoonup S\bar{z}=\bar u$ in $W_{\Om,\kappa}^{s,2}$ as $n\to\infty$. This implies that 
\begin{align*}
&\lim_{n\to\infty}\left(\int\int_{\RR^{2N}\setminus(\Omc)^2}\frac{(u_n(x)-u_n(y))(v(x)-v(y))}{|x-y|^{N+2s}}\;dxdy+\int_{\Omc} u_nv\;d\mu\right)\\
=& \int\int_{\RR^{2N}\setminus(\Omc)^2}\frac{(\bar u(x)-\bar u(y))(v(x)-v(y))}{|x-y|^{N+2s}}\;dxdy+\int_{\Omc}\bar uv\;d\mu,
\end{align*}
for every $v\in W_{\Om,\kappa}^{s,2}$. Since $z_n \rightharpoonup \bar z$ in $L^2(\Omc,\mu)$ as $n\to\infty$, it follows that
 \begin{align*}
\lim_{n\to\infty}\int_{\Omc}z_nv\;d\mu=\int_{\Omc}\bar zv\;d\mu,
 \end{align*}
for every $v\in W_{\Om,\kappa}^{s,2}$. Therefore we can pass to the limit in \eqref{We-so-n} as $n\to\infty$ to obtain that $(S\bar z,\bar z)=(\bar u,\bar z)$ satisfies the state equation \eqref{eq:Sn}.
The rest of the steps are similar to the proof of Theorem~\ref{thm:exit_dcp} and we omit them for brevity. 
 \end{proof}
 
As in the case of the fractional Dirichlet exterior control problem \eqref{eq:rpd} we will 
next identify the adjoint of the control-to-state map $S$.

 \begin{lemma}\label{lem:adjS_n}
  For the state equation \eqref{eq:Sn} the adjoint operator 
  $S^* : L^2(\Om) \rightarrow Z_R$
  is given by 
  
  \[
   ( S^* w , z )_{Z_R} = \int_{\RR^N\setminus\Om} p z\;d\mu \quad 
    \forall z \in Z_R , 
  \]
  where $w \in L^2(\Om)$ and $p \in W^{s,2}_{\Om,\kappa}$ is the weak solution to
  \begin{equation}\label{eq:adj_state_n}
  \begin{cases} 
   (-\Delta)^s p &= w \quad \mbox{in } \Om \\
     \mathcal{N}_s p +\kappa p&= 0 \quad \mbox{ in } \RR^N \setminus \Om .
  \end{cases} 
  \end{equation}
 \end{lemma}
 
 \begin{proof}
  Let $w \in L^2(\Om)$ and $z \in Z_R$. Then  $Sz \in W^{s,2}_{\Om,\kappa}
  \hookrightarrow L^2(\Om)$ with the embedding being continuous. Then we can write
  \[
   (w,Sz)_{L^2(\Om)} = ( S^*w,z )_{Z_R} . 
  \] 
  Next we test \eqref{eq:adj_state_n} with $Sz$ to arrive at 
  \begin{align*}
   ( w,Sz)_{L^2(\Om)} 
         =& \frac{C_{N,s}}{2} 
         \int\int_{\RR^{2N}\setminus(\RR^N\setminus\Om)^2} 
         \frac{(u(x)-u(y))(p(x)-p(y))}{|x-y|^{N+2s}} \;dxdy \\
         &+ \int_{\Omc}
         u p\;d\mu 
         = \int_{\RR^N\setminus\Om} zp\;d\mu 
          = ( S^*w, z )_{Z_R},
  \end{align*}
  where we have used the fact that $u$ solves the state equation according to 
  Definition~\ref{def:weak_n}. This yields the asserted result. 
 \end{proof}

For the remainder of this section we will assume that $\xi > 0$. The proof
of next result is similar to Theorem~\ref{thm:focD} and is omitted for brevity. 

 \begin{theorem}
  Let the assumptions of Theorem~\ref{thm:exit_ncp} hold. Let $\mathcal{Z}$ be an open 
  set in $Z_R$ such that $Z_{ad,R} \subset \mathcal{Z}$. Let 
  $u \mapsto J(u) : L^2(\Om) \rightarrow \RR$ be continuously Fr\'echet differentiable with 
  $J'(u) \in L^2(\Om)$. If $\bar{z}$ is a minimizer of \eqref{eq:rpn}
  over $Z_{ad,R}$ then the first necessary optimality conditions are given by 
  \begin{equation}\label{eq:foc_ncp}
   \int_{\RR^N\setminus\Om} (\bar{p}+\xi\bar{z}) (z-\bar{z})\;d\mu \ge 0 , \quad z \in Z_{ad,R} 
  \end{equation}
  where $\bar{p} \in W^{s,2}_{\Om,\kappa}$ solves the adjoint equation 
  \begin{equation}\label{eq:adj_state_np}
  \begin{cases} 
   (-\Delta)^s \bar{p} &= J'(\bar{u}) \quad \mbox{in } \Om \\
     \mathcal{N}_s \bar{p} +\kappa \bar{p}&= 0 \quad \mbox{ in } \RR^N \setminus \Om .
  \end{cases} 
  \end{equation}
  Equivalently we can write \eqref{eq:foc_ncp} as 
  \begin{equation}\label{eq:projf_n}
   \bar{z} = \mathcal{P}_{Z_{ad,R}}\left(-\frac{\bar{p}}{\xi}\right) , 
  \end{equation}
  where $\mathcal{P}_{Z_{ad,R}}$ is the projection onto 
  the set $Z_{ad,R}$. 
  If $J$ is convex then \eqref{eq:foc_ncp} is 
  a sufficient condition. 
 \end{theorem}
 
    \begin{remark}[\bf Regularity of optimization variables]
        {\rm 
            As pointed out in Remark~\ref{rem:ocdreg} (Dirichlet case) the regularity for the
            integral fractional Laplacian is a delicate issue. In fact for the Robin problem, 
            in $\Omc$ we can only guarantee that $\bar{p}$ is in $L^2(\Omc,\mu)$ . 
            Therefore we cannot use the classical boot-strap argument to further improve the regularity 
            of the control $\bar{z}$. 
        }         
    \end{remark}


\section{Approximation of Dirichlet exterior value and control problems}\label{s:DirRob}

We recall that the Dirichlet exterior value problem 
\eqref{eq:fracPoisson} in our case is only understood in the very-weak sense (cf.~Theorem~\ref{thm:vwdexist}). Moreover a numerical approximation of solutions
to this problem will require a direct approximation of the interaction operator $\mathcal{N}_s$ which is challenging. 


The purpose of this section is to not only introduce a new approach to approximate 
weak and very-weak solutions to the nonhomogeneous Dirichlet exterior value problem 
(recall that if $z$ is regular enough then a very-weak solution is a weak solution, 
and every weak solution is a very-weak solution cf.~Theorem~\ref{thm:vwdexist}) 
but also to consider a regularized fractional Dirichlet exterior control problem. 
We begin by stating the regularized Dirichlet exterior value problem: 
Let $n\in\NN$. Find $u_n \in W^{s,2}_{\Om,\kappa}$ solving 
 \begin{equation}\label{eq:Sn-G-Reg}
 \begin{cases}
    (-\Delta)^s u_n &= 0 \quad \mbox{in } \Om \\
    \mathcal{N}_s u_n  +n\kappa u_n &= n\kappa z \quad \mbox{in } \RR^N\setminus \Om .
 \end{cases}                
 \end{equation} 
Notice that the fractional regularized Dirichlet exterior problem \eqref{eq:Sn-G-Reg} is 
nothing but the fractional Robin exterior value problem \eqref{eq:Sn}. 
We proceed by showing that as $n \rightarrow \infty$ the solution $u_n$ 
to \eqref{eq:Sn-G-Reg} converges to $u$ solving the state equation \eqref{eq:fracPoisson} 
in the weak sense \eqref{eq:vw_d}. This is our new method to solve the non-homogeneous
Dirichlet exterior value problem. Recall that the weak formulation of \eqref{eq:Sn-G-Reg} does not require access to $\mathcal{N}_s$ (cf. Definition~\eqref{def:weak_n}) and it 
is straightforward to implement.

In this section we are interested in solutions $u_n$ to the system \eqref{eq:Sn-G-Reg} that belong to the space $W_{\Om,\kappa}^{s,2}\cap L^2(\Omc)$ which is endowed with the norm
\begin{align}
\|u\|_{W_{\Om,\kappa}^{s,2}\cap L^2(\Omc)}:=\left(\|u\|_{W_{\Om,\kappa}^{s,2}}^2+\|u\|_{L^2(\Omc)}^2\right)^{\frac 12},\;\;u\in W_{\Om,\kappa}^{s,2}\cap L^2(\Omc).
\end{align}
 In addition, in our application we shall take $\kappa$ such that its support $\mbox{supp}[\kappa]$ is a compact set in $\Omc$. For this reason we shall assume the following.

\begin{assumption}\label{asum}
We assume that $\kappa\in L^1(\Omc)\cap L^\infty(\Omc)$ and satisfies $\kappa>0$ almost everywhere in $K:=\mbox{supp}[\kappa]\subset\Omc$, where $K$ is a compact set.
\end{assumption}

It follows from Assumption \ref{asum} that $\displaystyle\int_{\Omc}\kappa\;dx>0$. 

To show the existence of weak solutions to the system in \eqref{eq:Sn-G-Reg} that belong to $W_{\Om,\kappa}^{s,2}\cap L^2(\Omc)$, we need some preparation.

\begin{lemma}\label{le62}
Assume that Assumption \ref{asum} holds. Then
\begin{equation}\label{norm-equiv}
\|u\|_W:=\left(\int\int_{\RR^{2N}\setminus(\Omc)^2}\frac{|u(x)-u(y)|^2}{|x-y|^{N+2s}}\;dxdy+\int_{\Omc}|u|^2\;dx\right)^{\frac 12},
\end{equation}
defines an equivalent norm on $W_{\Om,\kappa}^{s,2}\cap L^2(\Omc)$.
\end{lemma}

\begin{proof}
Firstly, it is readily seen that there is a constant $C>0$ such that 
\begin{equation}\label{Ca-2}
\|u\|_W\le C\|u\|_{W_{\Om,\kappa}^{s,2}\cap L^2(\Omc)} \;\mbox{ for all }\; u\in W_{\Om,\kappa}^{s,2}\cap L^2(\Omc).
\end{equation}

Secondly we claim that there is a constant $C>0$ such that 
\begin{equation}\label{Ca-1}
\|u\|_{W_{\Om,\kappa}^{s,2}\cap L^2(\Omc)}\le C\|u\|_W \;\mbox{ for all }\; u\in W_{\Om,\kappa}^{s,2}\cap L^2(\Omc).
\end{equation}
It is clear that 
\begin{align}\label{A2}
\int_{\Omc}|u|^2\;d\mu\le \|\kappa\|_{L^\infty(\Omc)}\int_{\Omc}|u|^2\;dx.
\end{align}
It suffices to show that
 there is a constant $C>0$ such that for every $u\in W_{\Om,\kappa}^{s,2}\cap L^2(\Omc)$,
\begin{align}\label{A1}
\int_{\Omega}|u|^2\;dx\le C\left(\int\int_{\RR^{2N}\setminus(\Omc)^2}\frac{|u(x)-u(y)|^2}{|x-y|^{N+2s}}\;dxdy+\int_{\Omc}|u|^2\;dx\right),
\end{align}
We prove \eqref{A1} by contradiction.
Assume to the contrary that for every $n\in\NN$, there exists $(u_n)_{n\in\NN}\subset W_{\Om,\kappa}^{s,2}\cap L^2(\Omc)$ such that
\begin{align}\label{eq-con}
\int_{\Om}|u_n|^2\;dx> n\left(\int\int_{\RR^{2N}\setminus(\Omc)^2}\frac{|u_n(x)-u_n(y)|^2}{|x-y|^{N+2s}}\;dxdy+\int_{\Omc}|u_n|^2\;dx\right).
\end{align}
By possibly dividing \eqref{eq-con} by $\|u_n\|_{L^2(\Omega)}^2$ we may assume that $\|u_n\|_{L^2(\Omega)}^2=1$ for every $n\in\NN$. Hence, by \eqref{eq-con}, there is a constant $C>0$ (independent of $n$) such that for every $n\in\NN$,
\begin{align}\label{A1-1}
\int\int_{\RR^{2N}\setminus(\Omc)^2}\frac{|u_n(x)-u_n(y)|^2}{|x-y|^{N+2s}}\;dxdy+\int_{\Omc}|u_n|^2\;dx\le C.
\end{align}
Since $\kappa\in L^\infty(\Omc)$,  \eqref{A1-1} and \eqref{A2}  imply that for every $n\in\NN$,
\begin{align}\label{A1-2}
\int_{\Omc}|u_n|^2\;d\mu\le C.
\end{align}
Now \eqref{A1-1}, \eqref{A1-2} together with $\|u_n\|_{L^2(\Omega)}^2=1$ implies that $(u_n)_{n\in\NN}$ is a bounded sequence in $W_{\Om,\kappa}^{s,2}\cap L^2(\Omc)$. Therefore, after passing to a subsequence, if necessary, we may assume that $u_n$ converges weakly to some $u\in W_{\Om,\kappa}^{s,2}\cap L^2(\Omc)$ and strongly to $u$ in $L^2(\Omega)$, as $n\to\infty$ (as the embedding $W_{\Om,\kappa}^{s,2}\hookrightarrow L^2(\Omega)$ is compact by Remark \ref{rem-35}(c)). 
It follows from \eqref{eq-con} and the fact that $\|u_n\|_{L^2(\Omega)}^2=1$ that
\begin{align*}
\lim_{n\to\infty}\int\int_{\RR^{2N}\setminus(\Omc)^2}\frac{|u_n(x)-u_n(y)|^2}{|x-y|^{N+2s}}\;dxdy=0\;\mbox{ and }\;\lim_{n\to\infty} \int_{\Omc}|u_n|^2\;dx=0.
\end{align*}
This implies that $u_n|_{\Omc}$ converges strongly to zero in $L^2(\Omc)$ as $n\to\infty$,  and after passing to a subsequence, if necessary, we have that
\begin{equation}\label{lim}
\lim_{n\to\infty}|u_n(x)-u_n(y)|=0\;\;\mbox{ for a.e. }\, (x,y)\in\RR^{2N}\setminus(\Omc)^2,
\end{equation}
and
\begin{align}\label{lim2}
u_n\to 0\;\mbox{ a.e. in }\;\Omc\;\mbox{ as }\; n\to\infty.
\end{align}
More precisely, \eqref{lim} implies that
\begin{equation}\label{MaWa}
\begin{cases}
\lim_{n\to\infty}|u_n(x)-u_n(y)|=0\;\;&\mbox{ for a.e. }\, (x,y)\in\Omega\times\Omega,\\
\lim_{n\to\infty}|u_n(x)-u_n(y)|=0\;\;&\mbox{ for a.e. }\, (x,y)\in\Omega\times(\Omc),\\
\lim_{n\to\infty}|u_n(x)-u_n(y)|=0\;\;&\mbox{ for a.e. }\, (x,y)\in(\Omc)\times\Omega.
\end{cases}
\end{equation}
Using \eqref{MaWa}, we get that $u_n$ converges a.e. to some constant function $c$ in $\RR^N$ as $n\to\infty$.
From \eqref{lim2} and the uniqueness of the limit, we have that $c=0$ a.e. in $\RR^N$.
Since (after passing to a subsequence, if necessary) $u_n$ converges a.e. to $u$ in $\Omega$ as $n\to\infty$, the uniqueness of the limit also implies that $c=u=0$ a.e. on $\Omega$. On the other hand, we have  $\|u\|_{L^2(\Omega)}^2=\lim_{n\to\infty}\|u_n\|_{L^2(\Omega)}^2=1$, and this is a contradiction. Hence, \eqref{eq-con} is not possible and we have shown \eqref{A1}.

Finally the lemma follows from \eqref{Ca-2} and \eqref{Ca-1}. The proof is finished.
\end{proof}

%

%
%
%

The following theorem is the main result of this section.
 
 \begin{theorem}[\bf Approximation of weak solutions to Dirichlet problem]
 \label{thm:approx_dbcp}
 Assume that Assumption \ref{asum} holds. Then the following assertions hold.
 
\begin{enumerate}
\item  Let $z \in W^{s,2}(\Omc)$ and  $u_{n} \in W^{s,2}_{\Om,\kappa}\cap L^2(\Omc)$ be the weak solution of  \eqref{eq:Sn-G-Reg}. Let $u\in W^{s,2}(\RR^N)$ be the weak solution to the state equation 
  \eqref{eq:Sd}. Then there is a constant $C>0$ (independent of $n$) such that
  \begin{align}\label{es-diff}
  \|u-u_{n}\|_{L^2(\RR^N)}\le \frac{C}{n}\|u\|_{W^{s,2}(\RR^N)}.
  \end{align}
In particular $u_{n}$ converges strongly to $u$ in $L^2(\Omega)$ as $n\to\infty$.

\item Let $z \in L^2(\Omc)$ and $u_{n} \in W^{s,2}_{\Om,\kappa}\cap L^2(\Omc)$ be the weak solution of \eqref{eq:Sn-G-Reg}. Then there is a subsequence that we still denote by $\{u_n\}_{n\in\NN}$  and a $\tilde u\in L^2(\RR^N)$ such that $u_{n}\rightharpoonup \tilde u$ in $L^2(\RR^N)$ as $n\to\infty$, and $\tilde u$ satisfies

\begin{align}\label{eq63}
\int_{\Omega}\tilde u(-\Delta)^sv\;dx=-\int_{\Omc}\tilde u\mathcal N_sv\;dx,
\end{align}
for all $v\in V$.
\end{enumerate}
 \end{theorem}
 
    \begin{remark}[\bf Convergence to very-weak solution]
        {\rm 
            Notice that Part (a) of Theorem~\ref{thm:approx_dbcp} implies strong 
            convergence to a weak solution (with rate). On the other hand, Part (b)
            ``almost" implies weak convergence to a very-weak solution (we still do not 
            know if $\tilde{u}|_{\Omc} = z$). We emphasize that such an approximation 
            of very-weak solutions using Robin problem, to the best of our knowledge,
            is open even for the classical case $s = 1$ when the boundary function just belongs to $L^2(\pOm)$. 
        }    
    \end{remark}
 
 \begin{proof}[\bf Proof of Theorem \ref{thm:approx_dbcp}]
 (a) Let $z \in W^{s,2}(\Omc)$. Firstly,  recall that under our assumption
 $W^{s,2}(\Omc)\hookrightarrow L^2(\Omc)\hookrightarrow L^2(\Omc,\mu)$. Secondly, consider the system \eqref{eq:Sn-G-Reg}. A weak solution is $u_{n}\in W^{s,2}_{\Om,\kappa}\cap L^2(\Omc)$ such that the identity
  \begin{align}\label{RP-WS}
\frac{C_{N,s}}{2}\int\int_{\RR^{2N}\setminus(\Omc)^2}&\frac{(u_{n}(x)-u_{n}(y))(v(x)-v(y))}{|x-y|^{N+2s}}\;dxdy\notag\\
&+n\int_{\Omc}u_{n} v\;d\mu=n\int_{\Omc} zv\;d\mu,
\end{align}
holds for every $v\in  W_{\Om,\kappa}^{s,2}\cap L^2(\Omc)$. Proceeding as the proof of Proposition \ref{pro-sol-Ro} we can easily deduce that for every $n\in\NN$, there is a unique $u_{n}\in W^{s,2}_{\Om,\kappa}\cap L^2(\Omc)$ satisfying \eqref{RP-WS}.

 For $v,w\in W^{s,2}_{\Om,\kappa}\cap L^2(\Omc)$, we shall let
\begin{align*}
\mathcal E_n(v,w):=\frac{C_{N,s}}{2}\int\int_{\RR^{2N}\setminus(\Omc)^2}\frac{(v(x)-v(y))(w(x)-w(y))}{|x-y|^{N+2s}}\;dxdy+n\int_{\Omc}vw\;d\mu.
\end{align*}
We notice that proceeding as the proof of Lemma \ref{le62} we can deduce that there is a constant $C>0$ such that
\begin{align}\label{Ine-eq-norm}
\frac{C_{N,s}}{2}\int\int_{\RR^{2N}\setminus(\Omc)^2}\frac{|u_{n}(x)-u_{n}(y)|^2}{|x-y|^{N+2s}}\;dxdy+n\int_{\Omc}|u_{n}|^2\;dx
\le C\mathcal E_n(u_n,u_n).
\end{align}

Next, let $u\in W^{s,2}(\RR^N)$ be the weak solution of \eqref{eq:Sd_1} and $v\in W^{s,2}_{\Om,\kappa}\cap L^2(\Omc)$. Using the integration by parts formula \eqref{Int-Part} we get that
\begin{align}\label{IN-I}
\mathcal E_n(u-u_{n},v)=&\int_{\Omega}(-\Delta)^s(u-u_{n})v\;dx+\int_{\Omc}\mathcal N_s(u-u_{n})v\;dx\notag\\
&+n\int_{\Omc}\left(u-u_{n}\right) v\;d\mu\notag\\
=&\int_{\Omega}(-\Delta)^s(u-u_{n})v\;dx +\int_{\Omc}v\mathcal N_su\;dx\notag\\
&-\int_{\Omc}\left(\mathcal N_su_{n}+n\kappa(u_{n}-z)\right)v\;dx\notag\\
=&\int_{\Omc}v\mathcal N_su\;dx.
\end{align}
Taking $v=u-u_{n}$ in \eqref{IN-I} and using \eqref{Ine-eq-norm}, we get that there is a constant $C>0$ (independent of $n$) such that
\begin{align*}
n\|u-u_{n}\|_{L^2(\Omc)}^2&\le \mathcal E_n(u-u_{n},u-u_{n})=\int_{\Omc}(u-u_{n})\mathcal N_su\;dx\\
&\le \|u-u_{n}\|_{L^2(\Omc)}\|\mathcal N_su\|_{L^2(\Omc)}\\
&\le C\|u-u_{n}\|_{L^2(\Omc)} \|u\|_{W^{s,2}(\RR^N)}.
\end{align*}
We have shown that there is a constant $C>0$ (independent of $n$) such that
\begin{align}\label{B1}
\|u-u_{n}\|_{L^2(\Omc)}\le \frac{C}{n}\|u\|_{W^{s,2}(\RR^N)}.
\end{align}
Next, observe that
\begin{align}\label{B1-1}
\|u-u_{n}\|_{L^2(\Omega)}=\sup_{\eta\in L^2(\Omega)}\frac{\left|\int_{\Omega} (u-u_{n})\eta\;dx\right|}{\|\eta\|_{L^2(\Omega)}}.
\end{align}
For any $\eta\in L^2(\Omega)$, let $w\in W_0^{s,2}(\bOm)$ be the weak solution of the Dirichlet problem
\begin{align}\label{edp}
(-\Delta)^sw=\eta\;\;\mbox{ in }\;\Omega,\;\;\;w=0\;\;\mbox{ in }\;\Omc.
\end{align}
It follows from Proposition \ref{prop:weak_Dir} that there is a constant $C>0$ such that
\begin{align}\label{B2}
\|w\|_{W^{s,2}(\RR^N)}\le C\|\eta\|_{L^2(\Omega)}.
\end{align}
Since $w\in W_0^{s,2}(\bOm)$, then using \eqref{IN-I} we have that
\begin{align*}
&\int_{\Omega}(u-u_{n})(-\Delta)^sw\;dx\\
=&\frac{C_{N,s}}{2}\int\int_{\RR^{2N}\setminus(\Omc)^2}\frac{((u-u_{n})(x)-(u-u_{n})(y))(w(x)-w(y))}{|x-y|^{N+2s}}\;dxdy\\
&-\int_{\Omc}(u-u_{n})\mathcal N_sw\;dx\\
=&\mathcal E_n(u-u_{n},w)-\int_{\Omc}(u-u_{n})\mathcal N_sw\;dx\\
=&\int_{\Omc}w\mathcal N_su\;dx-\int_{\Omc}(u-u_{n})\mathcal N_sw\;dx\\
=&-\int_{\Omc}(u-u_{n})\mathcal N_sw\;dx.
\end{align*}
It follows from the preceding identity, \eqref{B1} and \eqref{B2} that
\begin{align}\label{B3}
\left|\int_{\Omega}(u-u_{n})(-\Delta)^sw\;dx\right|=&\left|\int_{\Omc}(u-u_{n})\mathcal N_sw\;dx\right|\notag\\
\le& \|u-u_{n}\|_{L^2(\Omc)}\|\mathcal N_sw\|_{L^2(\Omc)}\notag\\
\le  &\frac{C}{n}\|u\|_{W^{s,2}(\RR^N)}\|w\|_{W^{s,2}(\RR^N)}\notag\\
\le&  \frac{C}{n}\|u\|_{W^{s,2}(\RR^N)}\|\eta\|_{L^2(\Omega)}.
\end{align}
Using \eqref{B1-1} and \eqref{B3} we get that
\begin{align}\label{B4}
\|u-u_{n}\|_{L^2(\Omega)}\le \frac{C}{n}\|u\|_{W^{s,2}(\RR^N)}.
\end{align}
Now the estimate \eqref{es-diff} follows from \eqref{B1} and \eqref{B4}. Observe that it follows from \eqref{es-diff} that $u_{n}\to u$ in $L^2(\RR^N)$ as $n\to\infty$ and this completes the proof of Part (a).\\

(b) Now let $z \in L^2(\Omc)\hookrightarrow L^2(\Omc,\mu)$. Notice that $\{u_{n}\}_{n\in\NN}$ satisfies \eqref{RP-WS}. 
Proceeding as the proof of Lemma \ref{le62} we can deduce that there is a constant $C>0$ (independent of $n$) such that
\begin{align*}
n\|u_n\|_{L^2(\Omc)}^2\le C\mathcal E_n(u_n,u_n)\le nC\|\kappa\|_{L^\infty(\Omc)}\|z\|_{L^2(\Omc)} \|u_n\|_{L^2(\Omc)},
\end{align*}
and this  implies that
\begin{align}\label{weq-2es}
\|u_n\|_{L^2(\Omc)}\le C\|z\|_{L^2(\Omc)}.
\end{align}
Now we proceed as the proof of \eqref{B4}. As in \eqref{B1-1} we have that
\begin{align}\label{wB1-1}
\|u_{n}\|_{L^2(\Omega)}=\sup_{\eta\in L^2(\Omega)}\frac{\left|\int_{\Omega} u_{n}\eta\;dx\right|}{\|\eta\|_{L^2(\Omega)}}.
\end{align}
Let $\eta\in L^2(\Omega)$ and $w\in W_0^{s,2}(\bOm)$  the weak solution of \eqref{edp}.
Since $w\in W_0^{s,2}(\bOm)$, then using \eqref{IN-I} we have that
\begin{align*}
&\int_{\Omega}u_{n}(-\Delta)^sw\;dx\\
=&\frac{C_{N,s}}{2}\int\int_{\RR^{2N}\setminus(\Omc)^2}\frac{(u_{n}(x)-u_{n}(y))(w(x)-w(y))}{|x-y|^{N+2s}}\;dxdy-\int_{\Omc}u_{n}\mathcal N_sw\;dx\\
=&-\int_{\Omc}u_{n}\mathcal N_sw\;dx.
\end{align*}
It follows from the preceding identity, \eqref{weq-2es} and \eqref{B2} that
\begin{align}\label{wB3}
\left|\int_{\Omega}u_{n}(-\Delta)^sw\;dx\right|=&\left|\int_{\Omc}u_{n}\mathcal N_sw\;dx\right|
\le \|u_{n}\|_{L^2(\Omc)}\|\mathcal N_sw\|_{L^2(\Omc)}\notag\\
\le  &C\|z\|_{L^{2}(\Omc)}\|w\|_{W^{s,2}(\RR^N)}.
\end{align}
Using \eqref{weq-2es}, \eqref{wB3} and \eqref{B2} we get that there is a constant $C>0$ (independent of $n$) such that
\begin{align}\label{wB4}
\|u_{n}\|_{L^2(\Omega)}\le C\|z\|_{L^{2}(\Omc)}.
\end{align}
Combing \eqref{weq-2es} and \eqref{wB4} we get that
\begin{align}\label{wB5}
\|u_{n}\|_{L^2(\RR^N)}\le C\|z\|_{L^{2}(\Omc)}.
\end{align}
Hence, the sequence $\{u_{n}\}_{n\in\NN}$ is bounded in $L^2(\RR^N)$. Thus, after a subsequence, if necessary, we have that $u_{n}$ converges weakly to some $\tilde u$ in $L^2(\RR^N)$ as $n\to\infty$.

Using \eqref{RP-WS} we get that for every $v\in V:=\{v\in W_0^{s,2}(\bOm):\; (-\Delta)^sv\in L^2(\Omega)\}$, 
 \begin{align}\label{RP-WS-2}
\frac{C_{N,s}}{2}\int\int_{\RR^{2N}\setminus(\Omc)^2}\frac{(u_{n}(x)-u_{n}(y))(v(x)-v(y))}{|x-y|^{N+2s}}\;dxdy=0 . 
\end{align}
Using the integration by part formula \eqref{Int-Part} we can deduce that
\begin{align}\label{RP-WS-3}
\frac{C_{N,s}}{2}\int\int_{\RR^{2N}\setminus(\Omc)^2}&\frac{(u_{n}(x)-u_{n}(y))(v(x)-v(y))}{|x-y|^{N+2s}}\;dxdy\notag\\
=&\int_{\Omega}u_{n}(-\Delta)^sv\;dx+\int_{\Omc}u_{n}\mathcal N_sv\;dx,
\end{align}
for every $v\in V$. Combing \eqref{RP-WS-2} and \eqref{RP-WS-3} we get that the identity
\begin{align}\label{WS1}
\int_{\Omega}u_{n}(-\Delta)^sv\;dx+\int_{\Omc}u_{n}\mathcal N_sv\;dx=0,
\end{align}
holds for every $v\in V$. Passing to the limit in \eqref{WS1} as $n\to\infty$, we obtain that
\begin{align*}
\int_{\Omega}\tilde u(-\Delta)^sv\;dx+\int_{\Omc}\tilde u\mathcal N_sv\;dx=0,
\end{align*}
for every $v\in V$.  We have shown \eqref{eq63} and the proof is finished.
 \end{proof}

 Toward this end we introduce the regularized fractional Dirichlet control problem 
 \begin{subequations}\label{eq:ncp_reg}
 \begin{equation}\label{eq:Jn_reg}
    \min_{u\in U_R, z\in Z_R} J(u) + \frac{\xi}{2} \|z\|^2_{L^2(\Omc)} ,
 \end{equation}
 subject to the regularized boundary value problem (Robin problem): 
 Find $u_n \in U_R$ solving 
 \begin{equation}\label{eq:Sn_reg}
 \begin{cases}
    (-\Delta)^s u &= 0 \quad \mbox{in } \Om \\
    \mathcal{N}_s u  +n \kappa u &= n \kappa z \quad \mbox{in } \RR^N\setminus \Om ,
 \end{cases}                
 \end{equation}
 and the control constraints 
 \begin{equation}\label{eq:Zn_reg}
    z \in Z_{ad,R} .
 \end{equation}
 \end{subequations}
 Here $Z_R:=L^2(\Omc)$, $Z_{ad,R}$ is a closed and convex subset and $U_R := W^{s,2}_{\Om,\kappa}\cap L^2(\Omc)$.
  We again remark that 
 \eqref{eq:ncp_reg} is nothing but the fractional Robin exterior control problem. 

 \begin{theorem}[\bf Approximation of the Dirichlet control problem]
  The regularized control problem \eqref{eq:ncp_reg} admits a minimizer $(z_n,u(z_n))
  \in Z_{ad,R} \times (W^{s,2}_{\Om,\kappa}\cap L^2(\Omc))$. 
  Let $Z_{R} =W^{s,2}(\Omc)$ and $Z_{ad,R} \subset Z_{R}$ be bounded. 
  Then for any sequence $\{n_\ell\}_{\ell=1}^\infty$ with $n_\ell \rightarrow \infty$,
  there exists a subsequence still denoted by $\{n_\ell\}_{\ell=1}^\infty$ 
  such that     
  $z_{n_\ell} \rightharpoonup \tilde{z}$ in $W^{s,2}(\Omc)$ and 
  $u(z_{n_\ell}) \rightarrow u(\tilde{z})$ in $L^2(\RR^N)$ as $n\rightarrow \infty$ with  
  $(\tilde{z},u(\tilde{z}))$ solving the Dirichlet control problem 
  \eqref{eq:dcp} with $Z_{ad,D}$ replaced by 
  $Z_{ad,R}$. 
 \end{theorem}
 
 \begin{proof}
  Since the regularized control problem \eqref{eq:ncp_reg} is nothing but 
  the Robin control problem therefore existence of minimizers
  follows by directly using  Theorem~\ref{thm:exit_ncp}.
  Following the 
  proof of Theorem~\ref{thm:exit_ncp} and using the fact that 
  $Z_{ad,R}$ is a bounded subset of the reflexive Banach space 
  $W^{s,2}(\Omc)$, after a subsequence, if necessary, 
  we have that $z_{n_\ell} \rightharpoonup \tilde{z}$ in 
  $W^{s,2}(\Omc)$ as $n_\ell\to\infty$.
  Now since $Z_{ad,R}$ is closed and convex, then it is weakly   
  closed. Thus $\tilde{z} \in Z_{ad,R}$. 
 
 Following the proof of Theorem~\ref{thm:approx_dbcp} (a)
 there exists a subsequence $\{u_{n_\ell}\}$ such that 
 $u_{n_\ell} \rightarrow \tilde{u}$  in $L^2(\RR^N)$ as $n_\ell \rightarrow \infty$.  
 Combining this convergence with the aforementioned convergence of $z_{n_\ell}$ 
 we conclude that $(\tilde{z},\tilde{u}) \in Z_{ad,R} \times W^{s,2}(\RR^N)$ 
 solves the Dirichlet exterior value problem \eqref{eq:Sd}. 
 
 It then remains to show that $(\tilde{z},\tilde{u})$ is a minimizer of \eqref{eq:dcp}.
 Let $(z',u')$ be any minimizer of \eqref{eq:dcp}. Let us consider the regularized 
 state equation \eqref{eq:Sn_reg} but with boundary datum $z'$. We denote the
 solution of the resulting state equation by $u'_{n_\ell}$. By using the same
 limiting argument as above we can select a subsequence such that 
 $u'_{n_\ell} \rightarrow u'$ in $L^2(\RR^N)$ as $n \rightarrow \infty$. 
 Letting $j(z,u) := J(u) + \frac{\xi}{2} \|z\|^2_{L^2(\Omc)}$, it then follows 
 that 
  \[
   j(z',u') 
    \le j(\tilde{z},\tilde{u}) 
    \le \liminf_{n\rightarrow\infty} j(z_{n_\ell},u_{n_\ell}) 
    \le \liminf_{n\rightarrow\infty} j(z',u_{n_\ell}') 
    = j(z',u')
  \] 
 where the second inequality is due to the weak-lower semicontinuity of $J$. 
 The third inequality is due to the fact that $\{(z_{n_\ell},u_{n_\ell})\}$ 
 is a sequence of minimizers for \eqref{eq:ncp_reg}. This is what we needed 
 to show. 
 \end{proof}

We conclude this section by writing the stationarity system 
corresponding to \eqref{eq:ncp_reg}: Find  
$(z,u,p) \in Z_{ad,R} \times (W^{s,2}_{\Om,\kappa} \cap L^2(\Omc))\times  (W^{s,2}_{\Om,\kappa}\cap L^2(\Omc))$ such that 

 \begin{equation}\label{eq:statptR}
 \begin{cases}
\displaystyle \mathcal E(u,v)  &= \int_{\Omc} n\kappa z v\;dx, \\
\displaystyle\mathcal E(w,p)    &= \int_{\Om}J'(u)w \;dx ,\\
\displaystyle   \int_{\RR^N\setminus\Om} (n\kappa p+\xi z) (\widetilde{z}-z)\;dx &\ge 0 , 
 \end{cases} 
 \end{equation}
for all $(\widetilde{z},v,w) \in Z_{ad} \times (W^{s,2}_{\Om,\kappa} \cap L^2(\Omc))\times  (W^{s,2}_{\Om,\kappa}\cap L^2(\Omc))$. 
%
%
%

\section{Numerics}

The purpose of this section is to introduce numerical approximation of the problems
we have considered so far. In Subsection \ref{s:RobinNum} we begin with a finite 
element approximation of the Robin problem \eqref{eq:Sn-G-Reg} which is the same 
as the regularized Dirichlet problem. We approximate the Dirichlet problem using
the Robin problem. Next in Subsection~\ref{s:source} we introduce an external
source identification problem where we clearly see the difference between
the nonlocal case and the classical case ($s\sim 1$).  Finally, Subection~\ref{s:dcpNum} 
is devoted to optimal control problems.

\subsection{Approximation of nonhomogeneous Dirichlet problem via Robin problem}
\label{s:RobinNum}

In view of Theorem~\ref{thm:approx_dbcp} we can approximate the Dirichlet 
problem with the help of the Robin (regularized Dirichlet) problem \eqref{eq:Sn-G-Reg}.
Therefore we begin by introducing a discrete scheme for the Robin problem. 
Let $\widetilde{\Om}$ be a large enough open bounded set containing $\Om$. We consider a 
conforming simplicial triangulation of $\Om$ and $\widetilde\Om \setminus \Om$ such
that the resulting partition remains admissible. We shall assume that the support of 
the datum $z$ and $\kappa$ is contained in $\widetilde\Om\setminus\Om$. We let our finite element space 
$\mathbb{V}_h$ (on $\widetilde\Om$) to be a set of standard continuous piecewise linear 
functions. Then the discrete (weak) version of \eqref{eq:Sn_reg} with nonzero right-hand-side is given by: Find $u_{h}
\in \mathbb{V}_h$ such that 
 \begin{align}\label{eq:discR}
 \begin{aligned}
  \int\int_{\RR^{2N}\setminus(\Omc)^2}&\frac{(u_{h}(x)-u_{h}(y))(v(x)-v(y))}{|x-y|^{N+2s}}\;dxdy +\int_{\widetilde\Om\setminus\Om} n\kappa u_{h} v\;dx \\
   &= \langle f,v\rangle_{(W^{s,2}_{\Om,\kappa}\cap L^2(\Omc))^\star,W^{s,2}_{\Om,\kappa}\cap L^2(\Omc)} + \int_{\widetilde\Om \setminus \Om} n\kappa z v\;dx \quad \forall v \in \mathbb{V}_h . 
 \end{aligned}   
 \end{align}
We approximate the double integral over $\RR^{2N}\setminus(\Omc)^2$ by using the approach
from \cite{GA:JP15,acosta2017short}. The remaining integrals are computed using quadrature
which is accurate for polynomials of degree less than and equal to 4. 

We next consider an example that has been taken from \cite{acosta2017finite}. 
Let $\Om = B_0(1/2) \subset \RR^2$ then our goal is to find 
$u$ solving
 \begin{align*}
  (-\Delta)^s u &= 2 \quad \mbox{in } \Om \\
      u(\cdot) &= \frac{2^{-2s}}{\Gamma(1+s)^{2}} \left(1-|\cdot|^2\right)_{+}^s \quad  \mbox{in } \Omc . 
 \end{align*}
The exact solution in this case is given by  
 \[
  u(x) = u_1(x) + u_2(x) = \frac{2^{-2s}}{\Gamma(1+s)^{2}} 
          \left( \left(1-|x|^2\right)_{+}^s + \left(\frac14-|x|^2\right)_{+}^s \right) ,
 \] 
where $u_1$ and $u_2$ solve
 \begin{equation}\label{eq:u1u2}
  \begin{cases}
   (-\Delta)^s u_1 &= 1 \quad \mbox{in } \Om \\
      u_1 &= \frac{2^{-2s}}{\Gamma(1+s)^{2}} \left(1-|\cdot|^2\right)_{+}^s \quad  \mbox{in } \Omc ,
  \end{cases}
  \qquad 
  \begin{cases}
   (-\Delta)^s u_2 &= 1 \quad \mbox{in } \Om \\
      u_2 &= 0 \quad  \mbox{in } \Omc . 
  \end{cases}
 \end{equation}
We let $\widetilde\Om = B_0(1.5)$. We next approximate \eqref{eq:u1u2} using \eqref{eq:discR} and we set $\kappa = 1$. Notice that we use a quasiuniform
mesh. At first we fix $s = 0.5$ and 
Degrees of Freedom (DoFs) to be $\mbox{DoFs} = 2920$. For this 
configuration, we study the $L^2(\Omega)$ error $\|u-u_h\|_{L^2(\Omega)}$
with respect to $n$ in Figure~\ref{eq:fig_rate} (left). As expected 
from Theorem~\ref{thm:approx_dbcp} (a) we observe an approximation rate
of $1/n$.

Next for a fixed $s = 0.5$, we check the stability of our scheme with 
respect to $n$ as we refine the mesh. We have plotted the 
$L^2$-error as we refine the mesh (equivalently increase DOFs) for $n = 1e2, 1e3, 1e4, 1e5$.
We notice that the error remains stable with respect to $n$ and 
we observe the expected rate of convergence with respect to
DoFs \cite{acosta2017finite} 
 \[
  \|u-u_h\|_{L^2(\Om)} \approx (\mbox{DoFs})^{-\frac12} . 
 \] 
In the right panel we have shown the $L^2$-error for a fixed $n = 1e5$ but for 
various $s = 0.2, 0.4, 0.6, 0.8$. In all cases we obtain the expected rate 
of convergence $(\mbox{DoFs})^{-\frac12}$. 

 \begin{figure}[h!]
  \centering
  \includegraphics[width=0.32\textwidth]{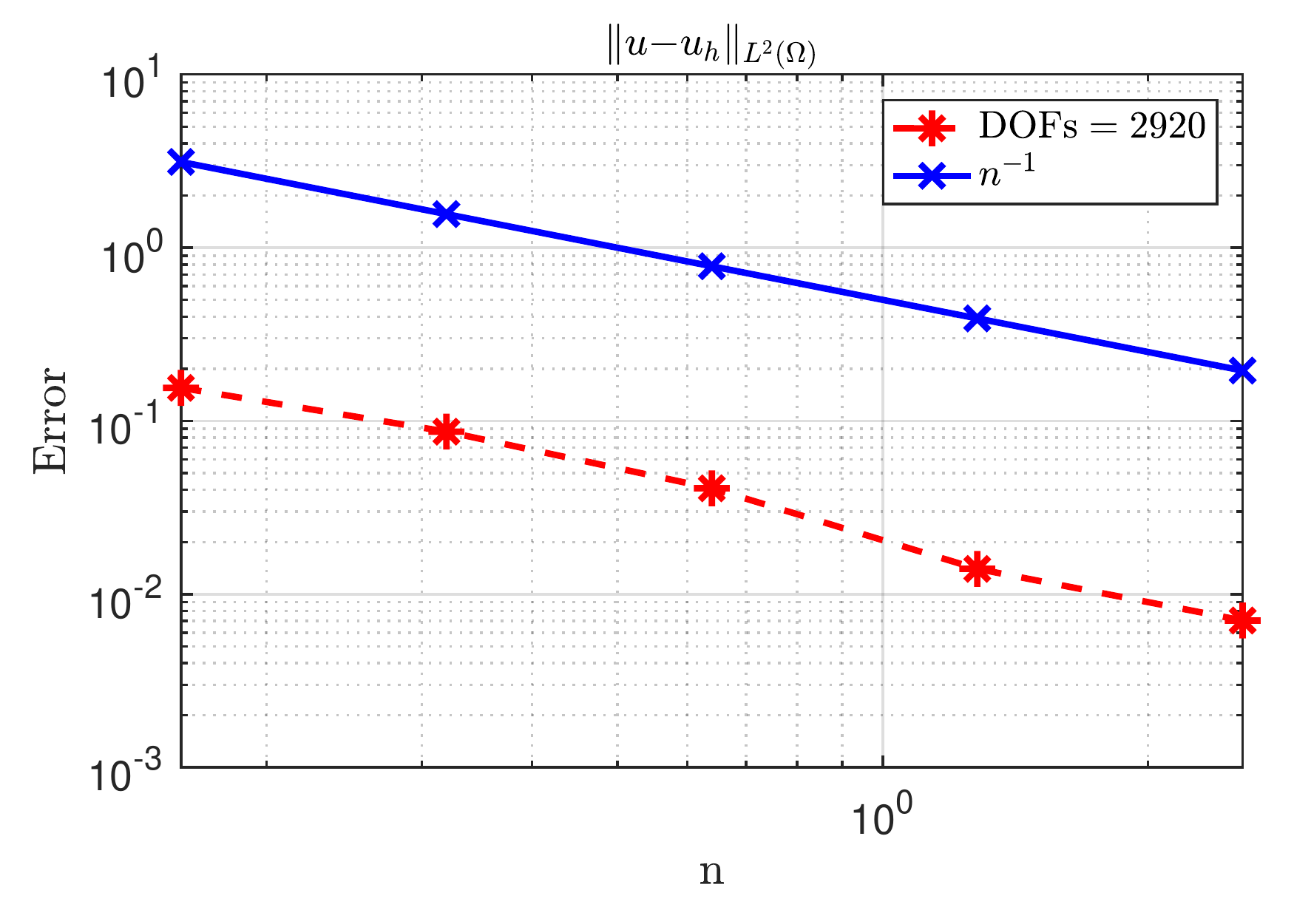}  
  \includegraphics[width=0.32\textwidth]{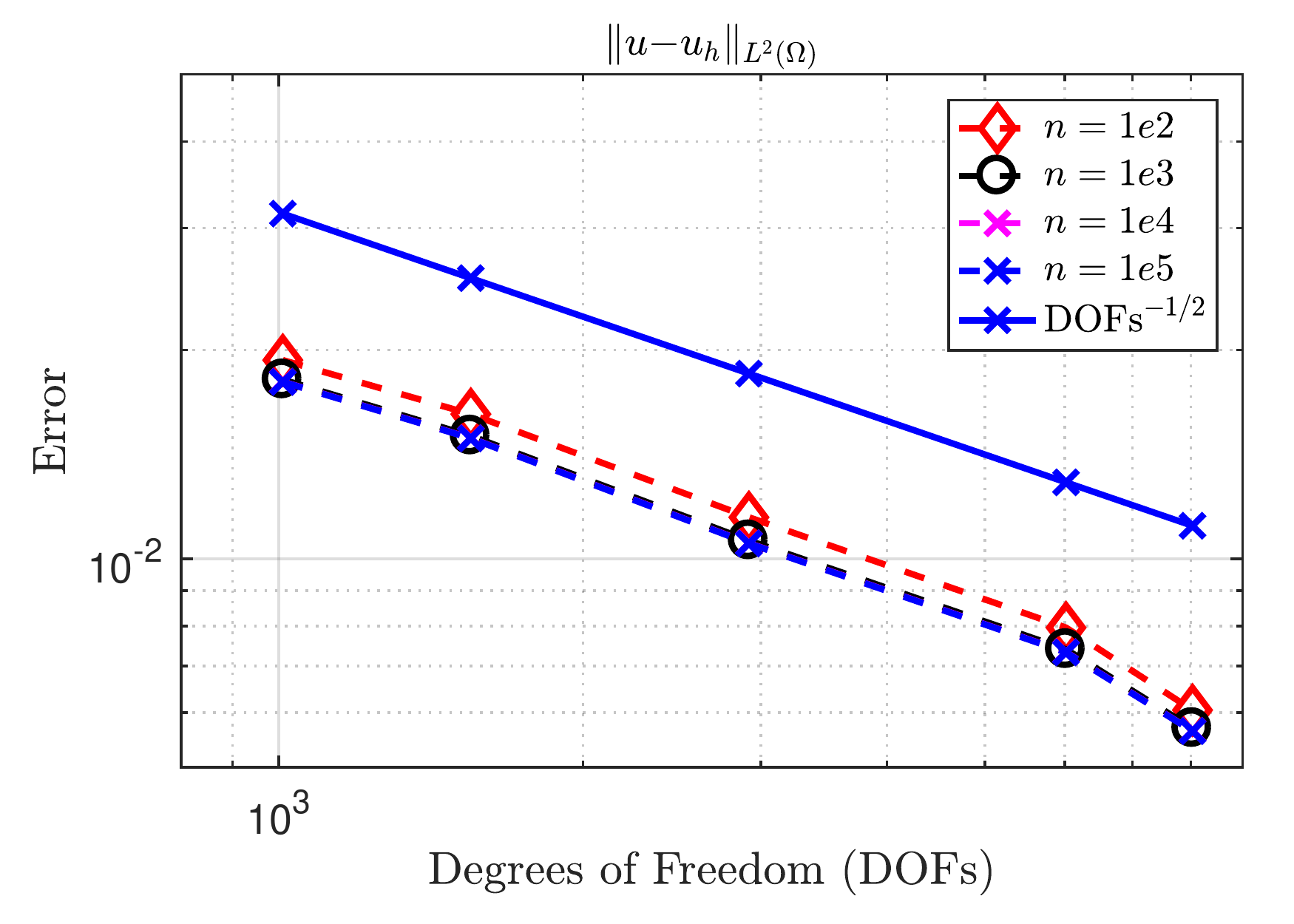}
  \includegraphics[width=0.32\textwidth]{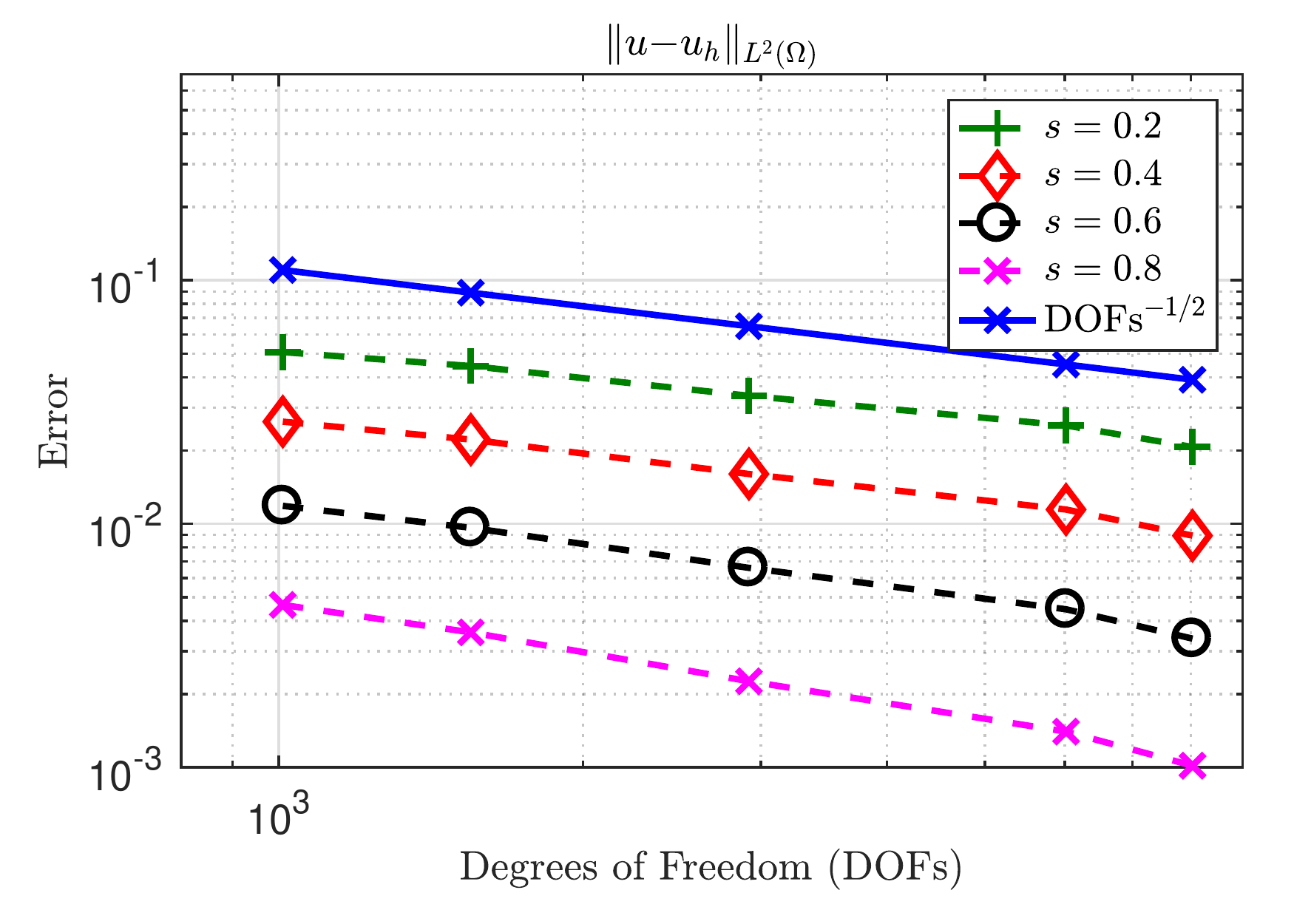}     
  \caption{\label{eq:fig_rate}
   Left panel: Let $s = 0.5$ and $\mbox{DoFs} = 2920$ be fixed. We let $\kappa = 1$ 
   and consider $L^2$-error between actual solution $u$ to the Dirichlet problem
   and its approximation $u_h$ which solves the Robin problem. We have plotted the 
   error with respect to $n$. We observe a rate of $1/n$ which confirms our 
   theoretical result \eqref{es-diff}. Middle panel: Let $s=0.5$ be fixed. For 
   each $n = 1e2, 1e3, 1e4, 1e5$ we have plotted the $L^2$-error with respect to 
   degrees of freedom (DOFs) as we refine the mesh. We notice the error is stable 
   with respect to $n$. In addition, the rate of convergence is 
   $(\mbox{DoFs})^{-\frac12}$ (as expected) and is independent of $n$. 
   Right panel: Let $n = 1e5$ be fixed. We again plot the $L^2$-error 
   with respct to DOFs for various values of $s$. The effective convergence rate 
   is again $(\mbox{DoFs})^{-\frac12}$ and is independent of $s$.}
 \end{figure}

\subsection{External source identification problem}
\label{s:source}

We next consider an inverse problem to identify a source that is located outside
the observation domain $\Om$. The optimality system is as given in \eqref{eq:statptR}
where we have approximated the Dirichlet problem by the Robin problem. 
We use the continuous piecewise linear finite element discretization for all the 
optimization variables: state $(u)$, control $(z)$, and adjoint $(p)$. We choose
our objective function as 
 \[
  J(u,z) = \frac12 \|u-u_d\|^2_{L^2(\Om)} 
   + \frac{\xi}{2} \|z\|^2_{Z_R} ,
 \]
and we let $Z_{ad,R} := \{ z\in Z_R \;:\; z \ge 0, \ \mbox{a.e. in } \widehat{\Om} \}$ 
where $\widehat{\Om}$ is the support set of control $z$ and $\kappa$ that is contained 
in $\widetilde\Om\setminus \Om$. Moreover $u_d : L^2(\Om) \rightarrow \mathbb{R}$ 
is the given data (observations). All the optimization problems below are solved using projected-BFGS
method with Armijo line search.

Our computational setup is shown in Figure~\ref{f:ex2_setup}. The centered square region 
is $\Om = [-0.4, 0.4]^2$ and the region inside the outermost ring is 
$\widetilde\Om = B_0(1.5)$. The smaller square inside 
$\widetilde\Om \setminus \Om$ is $\widehat\Om$ which is the support of source/control. 
The right panel in Figure~\ref{f:ex2_setup} shows a finite element mesh with DoFs = 6103. 
 \begin{figure}[htb] 
  \centering
  \includegraphics[width = 0.22\textwidth]{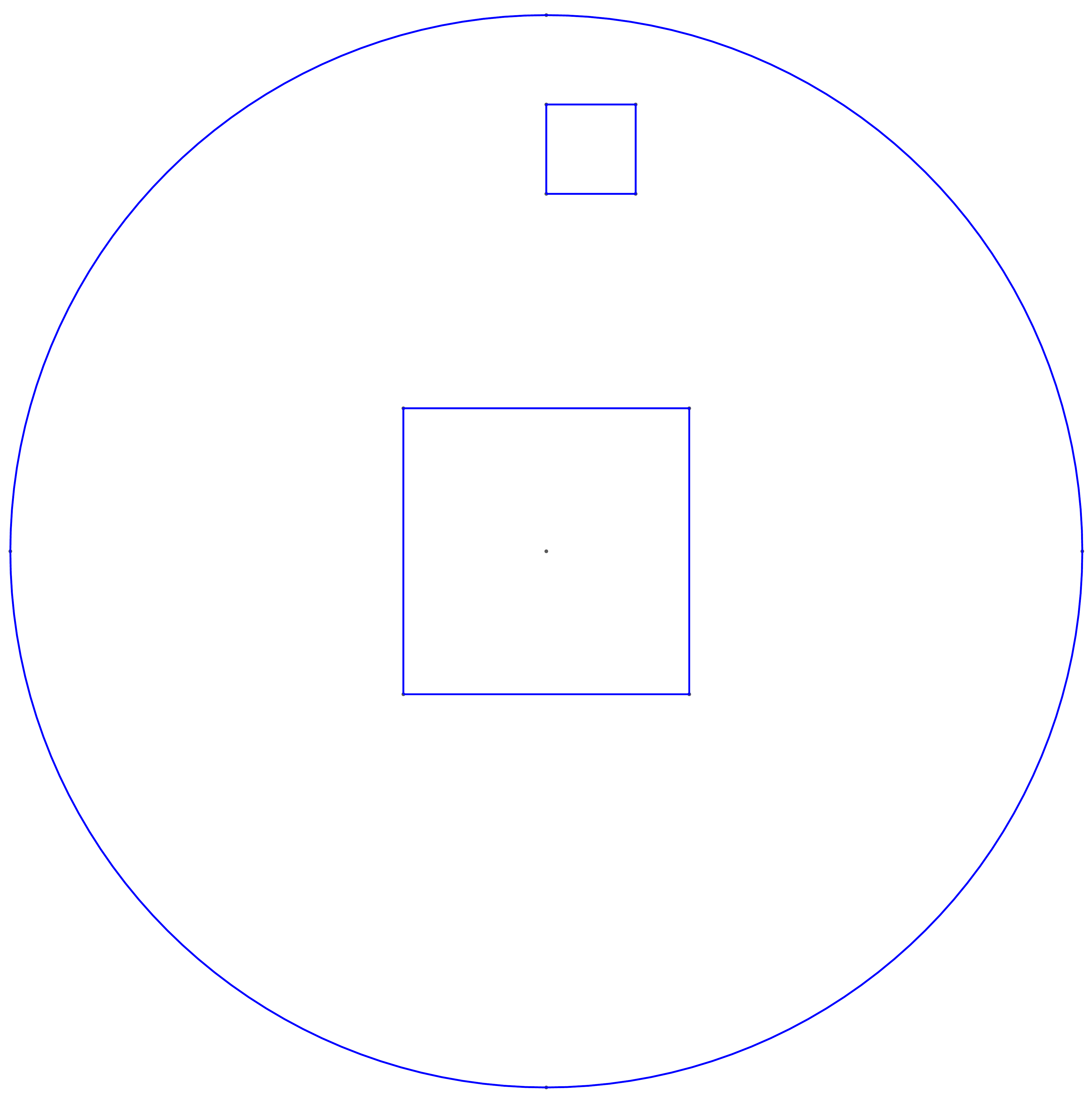}\qquad 
  \includegraphics[width = 0.25\textwidth]{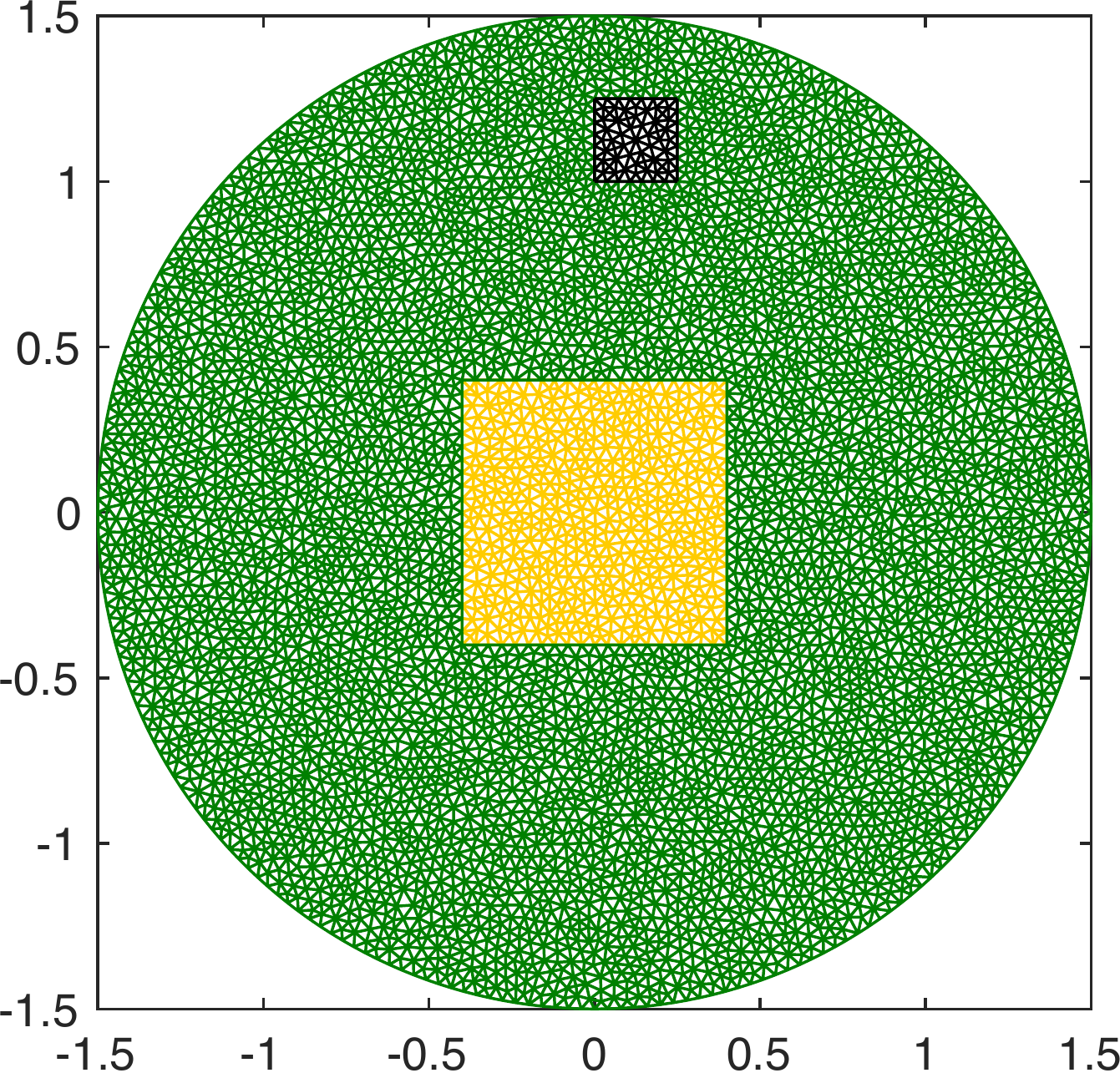}
  \caption{\label{f:ex2_setup}
  Left: computational domain where the inner square is $\Omega$, the region inside the
  outer circle is $\widetilde\Om$ and the outer square inside $\widetilde\Om \setminus
  \Om$ is $\widehat\Om$ which is the region where source/control is supported.
  Right: A finite element mesh.}
 \end{figure}

We define $u_d$ as follows. For $z = 1$, we first solve the state equation for 
$\tilde{u}$ (first equation in \eqref{eq:statptR}). We then add a normally 
distributed random noise with mean zero and standard deviation 0.02 to $\tilde{u}$. 
We call the resulting expression as $u_d$. Furthermore, we set 
$\kappa=1$, and $n = 1e5$. 

Our goal is then to identify the source $\bar{z}_h$. 
In Figure~\ref{f:ex4_1}, we first show the behavior of optimal $\bar{z}_h$ 
for different values of the regularization parameter 
$\xi = 1e-1, 1e-2, 1e-4, 1e-8, 1e-10$. As expected the larger the value of 
$\xi$, the smaller the magnitude of $\bar{z}_h$, and this behavior saturates 
at $\xi = 1e-8$. 
    \begin{figure}[htb]
        \includegraphics[width=0.3\textwidth]{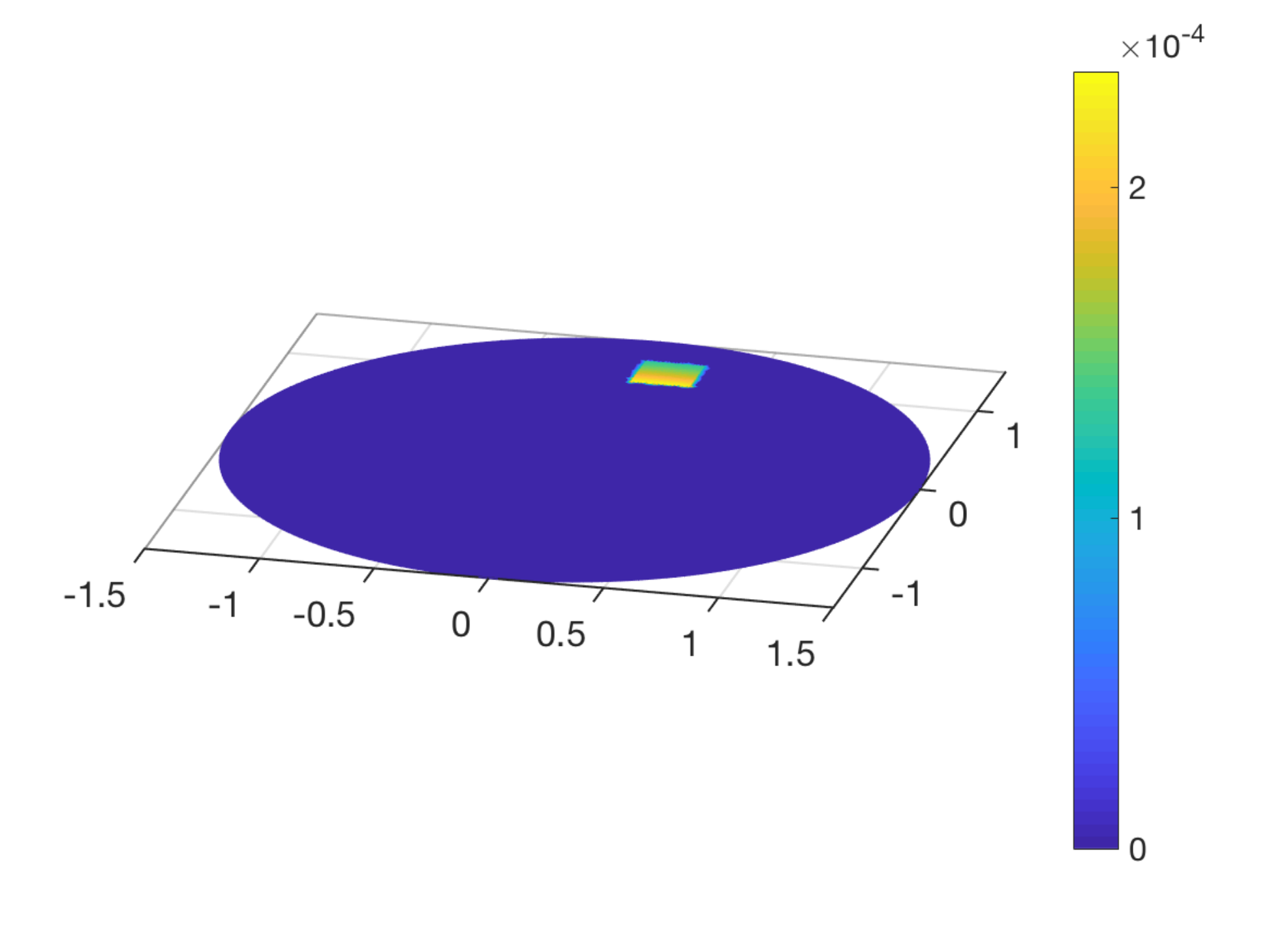}
        \includegraphics[width=0.3\textwidth]{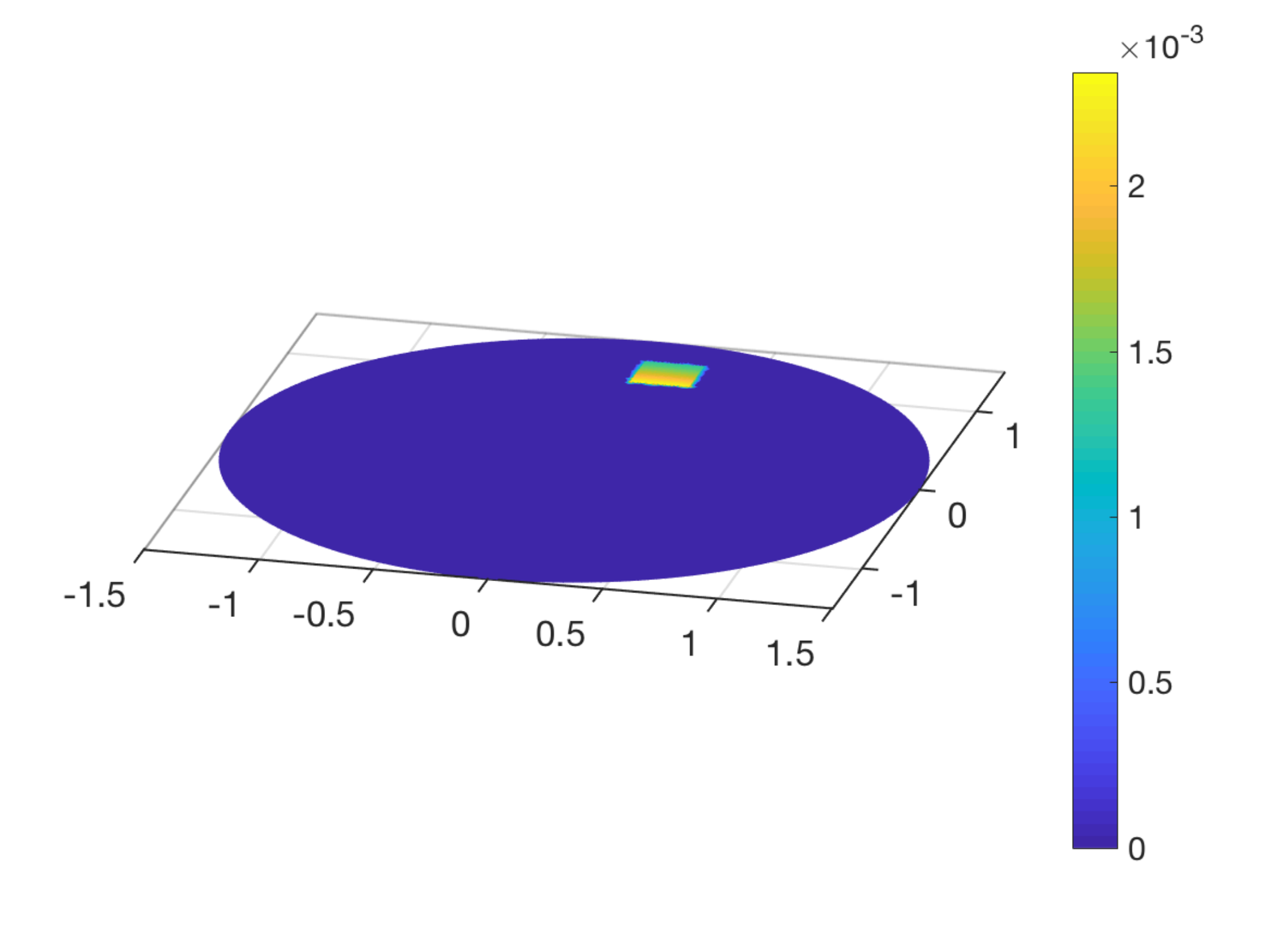}
        \includegraphics[width=0.3\textwidth]{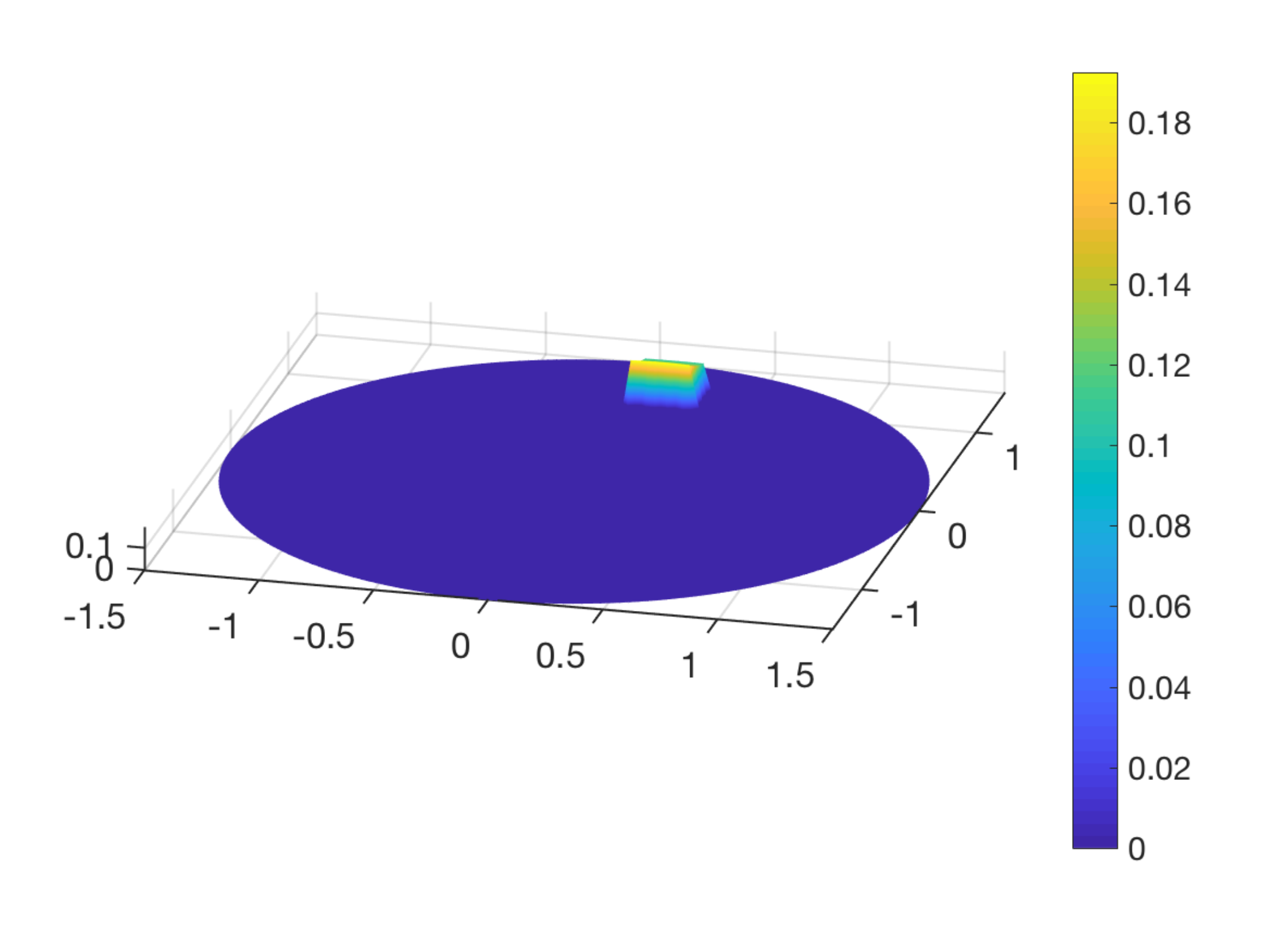}        
        \includegraphics[width=0.3\textwidth]{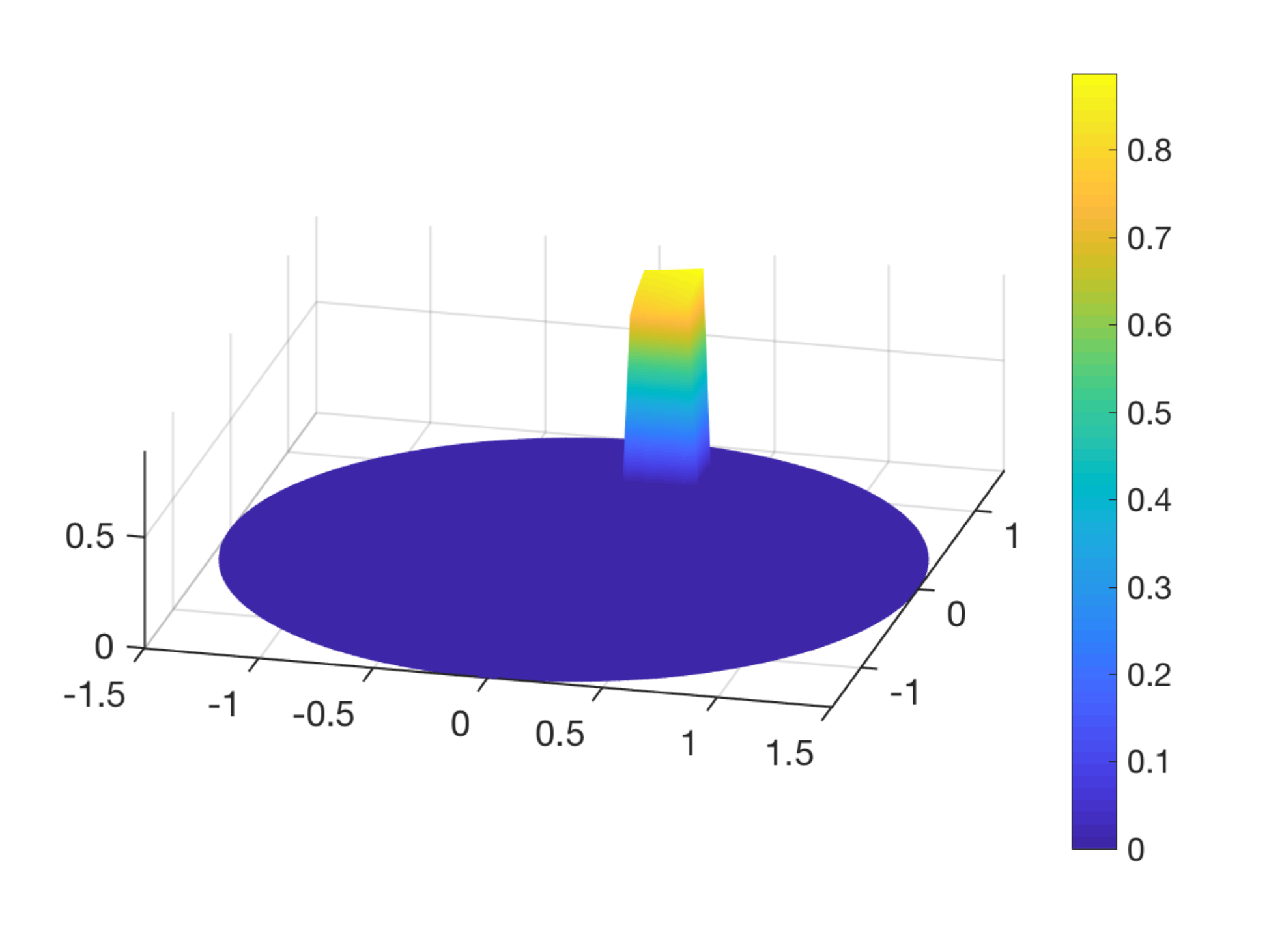}        
        \includegraphics[width=0.3\textwidth]{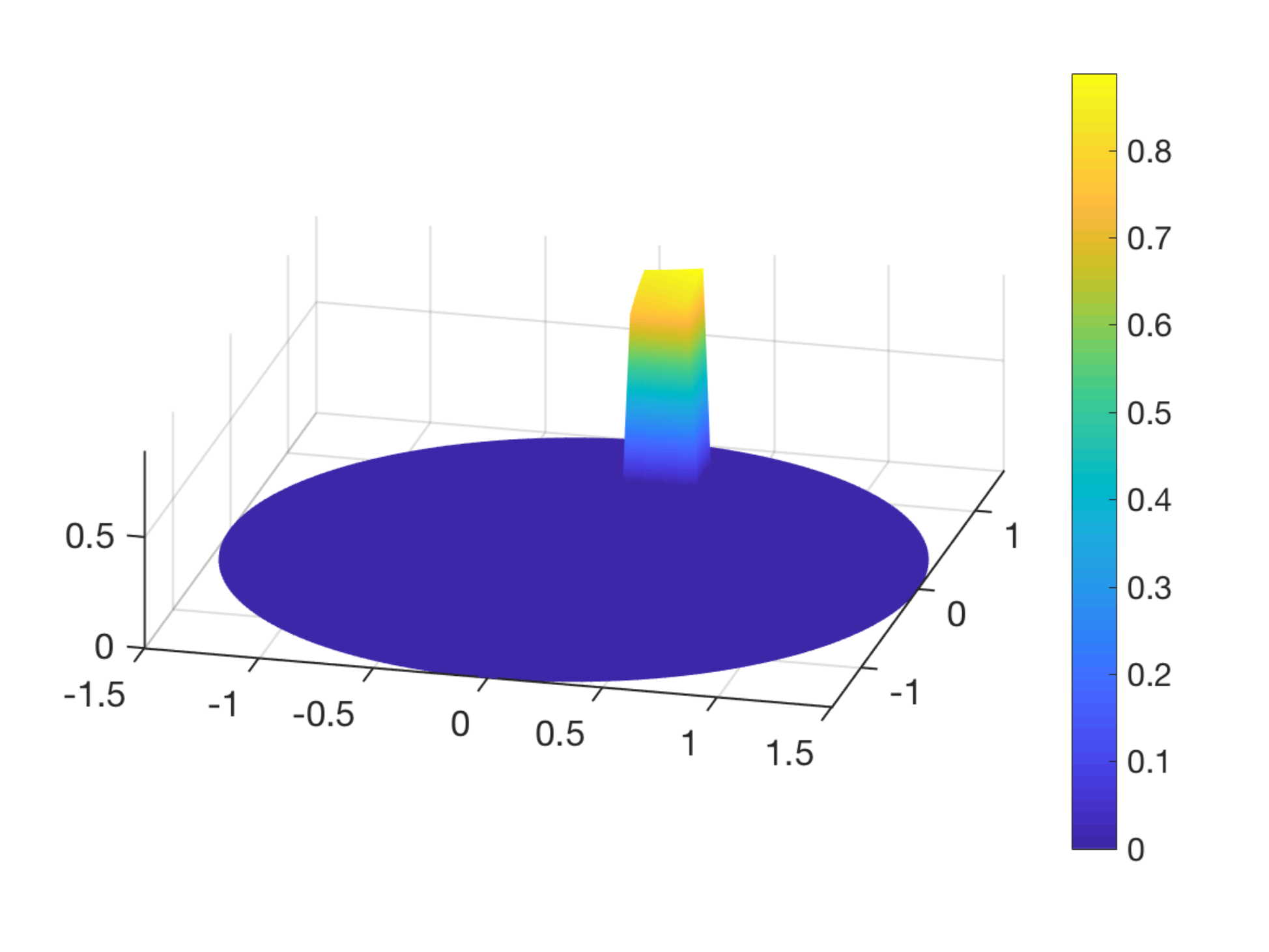}                
        \caption{\label{f:ex4_1}
        External source identificaiton problem. The panels show the
        behavior of $\bar{z}_h$ with repsect to the regularization parameter: top row from 
        left to right $\xi = 1e-1, 1e-2, 1e-4$; bottom row from left to right: $\xi = 1e-8,
        1e-10$. As expected the larger $\xi$, the smaller the magnitude of $\bar{z}_h$,
        but it saturates at $\xi = 1e-8$.}
    \end{figure}

Next, for a fixed $\xi = 1e-8$, Figure~\ref{f:ex4_2} shows the optimal $\bar{z}_h$ for $s = 0.1, 0.6, 0.7, 0.8, 0.9$. We notice that for large $s$, $\bar{z}_h \equiv 0$. This is 
expected as larger the $s$ is, the more close we are to the classical Poisson case and we know that we cannot impose external condition in that case. 

    \begin{figure}[htb]
        \includegraphics[width=0.3\textwidth]{figures/control/ex_4/noise_std02_lamb08/s01/ctrl}        
        \includegraphics[width=0.3\textwidth]{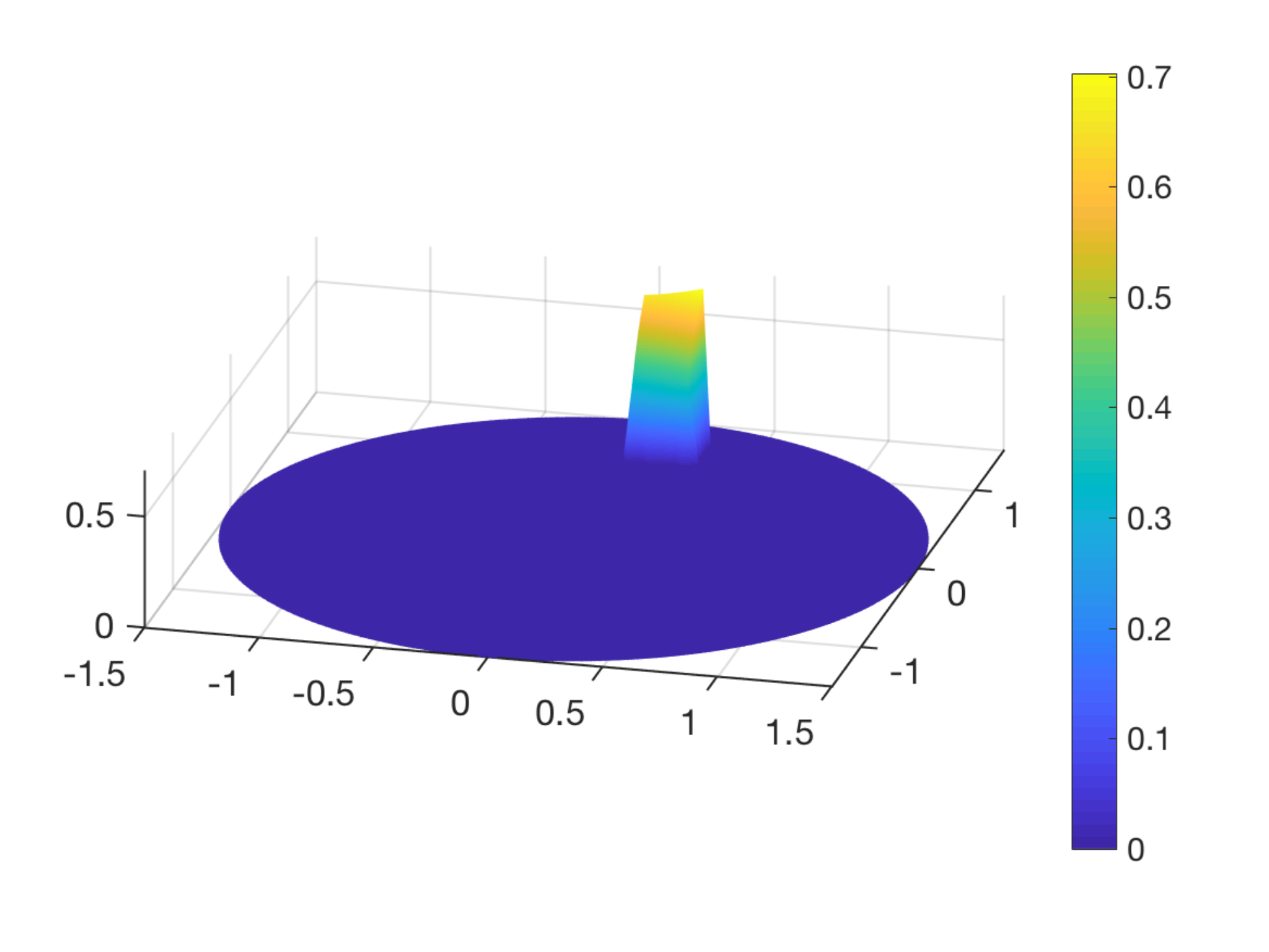}                        
        \includegraphics[width=0.3\textwidth]{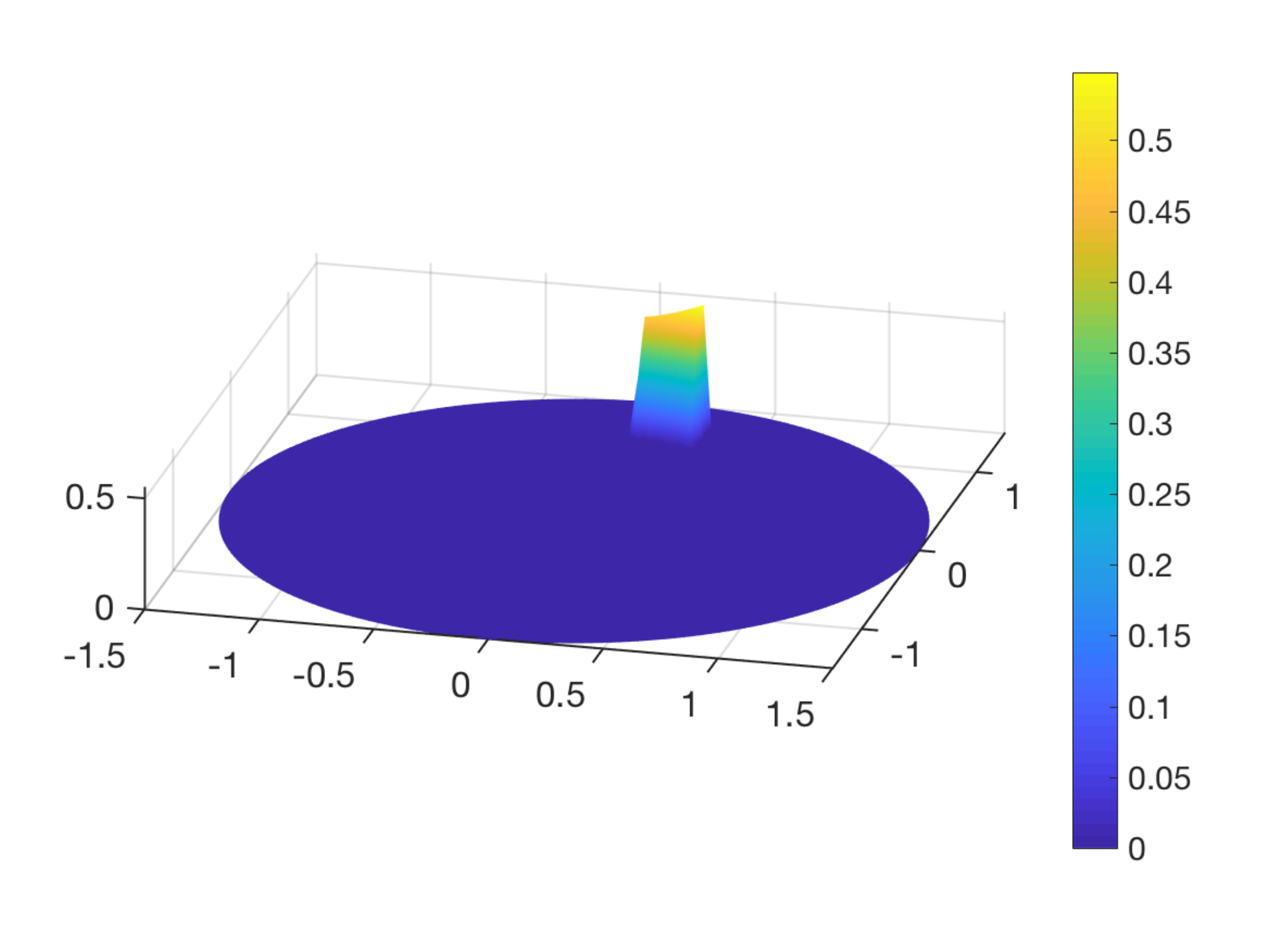}                
        \includegraphics[width=0.3\textwidth]{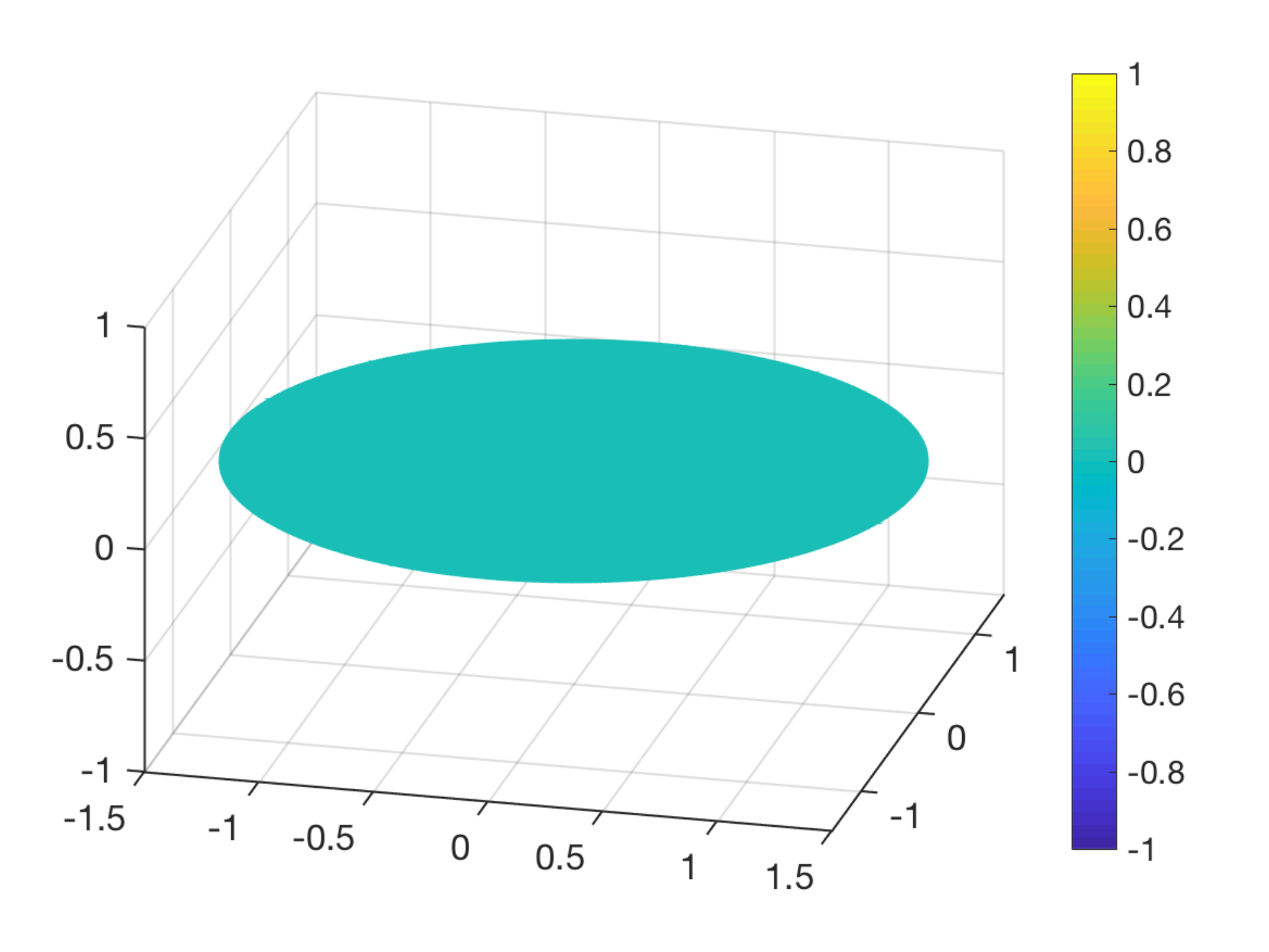}                
        \includegraphics[width=0.3\textwidth]{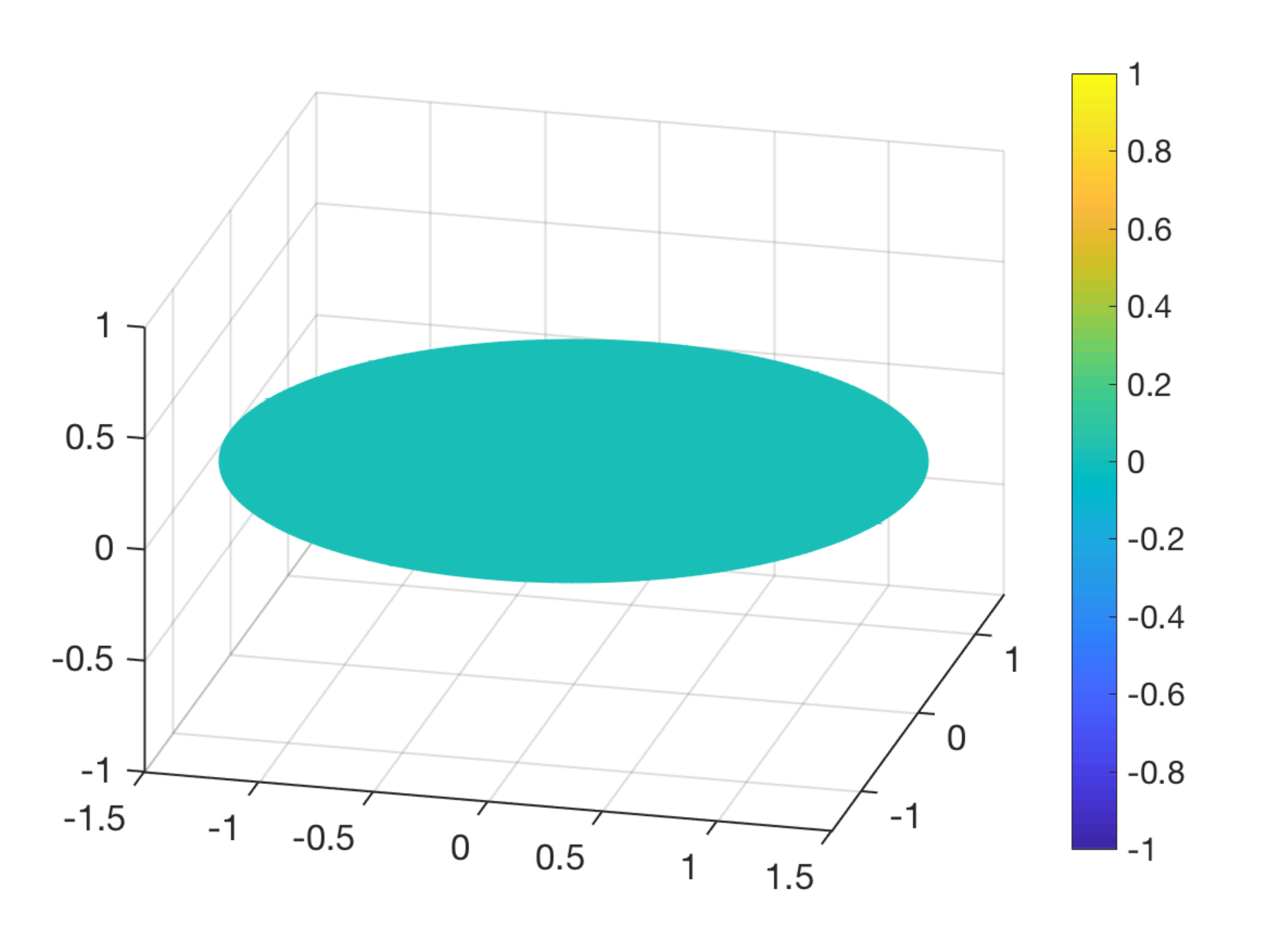}         
        \caption{\label{f:ex4_2}
        The panels show the behavior of $\bar{z}_h$ as we 
        vary the exponent $s$. Top row from left to right: $s = 0.1,0.6,0.7$. Bottom 
        row from left to right: $s = 0.8,0.9$. For smaller values of $s$, the recovery
        of $\bar{z}_h$ is quite remarkable. However, for larger values of $s$, $\bar{z}_h 
        \equiv 0$ as expected -- the behavior of $\bar{u}_h$ for large $s$ is close to 
        the classical Poisson problem which does not allow external sources.}
    \end{figure}

\subsection{Dirichlet control problem}
\label{s:dcpNum}


We next consider two Dirichlet control problems. The setup is similar to 
Subsection~\ref{s:source} except now we set $u_d \equiv 1$.

\begin{example} 
{\rm 
The computational setup for the first example is shown in Figure~\ref{f:ex1_setup}. 
Let $\Om = B_0(1/2)$ (the region insider the innermost ring) and the region 
inside the outermost ring is $\widetilde\Om = B_0(1.5)$. The annulus inside 
$\widetilde\Om \setminus \Om$ is 
$\widehat\Om$ which is the support of control. The right panel in 
Figure~\ref{f:ex1_setup} shows a finite element mesh with DoFs = 6069. 

In Figures~\ref{f:ex1s02} and \ref{f:ex1s08} we have shown the optimization results for 
$s = 0.2$ and $s = 0.8$, respectively. The top row shows the desired state $u_d$ (left) and 
the optimal state $\bar{u}_h$ (right). The bottom row shows the optimal control $\bar{z}_h$ 
(left) and the optimal adjoint variable $\bar{p}_h$ (right). We notice that in both cases
we can approximate the desired state to a high accuracy but the approximation is slightly
better for smaller $s$, especially close to the boundary. This is to be expected as for 
large values of $s$ the regularity of the adjoint variable deteriorates significantly 
(cf.~Remark~\ref{rem:ocdreg}). 

 \begin{figure}[htb] 
  \centering
  \includegraphics[width = 0.22\textwidth]{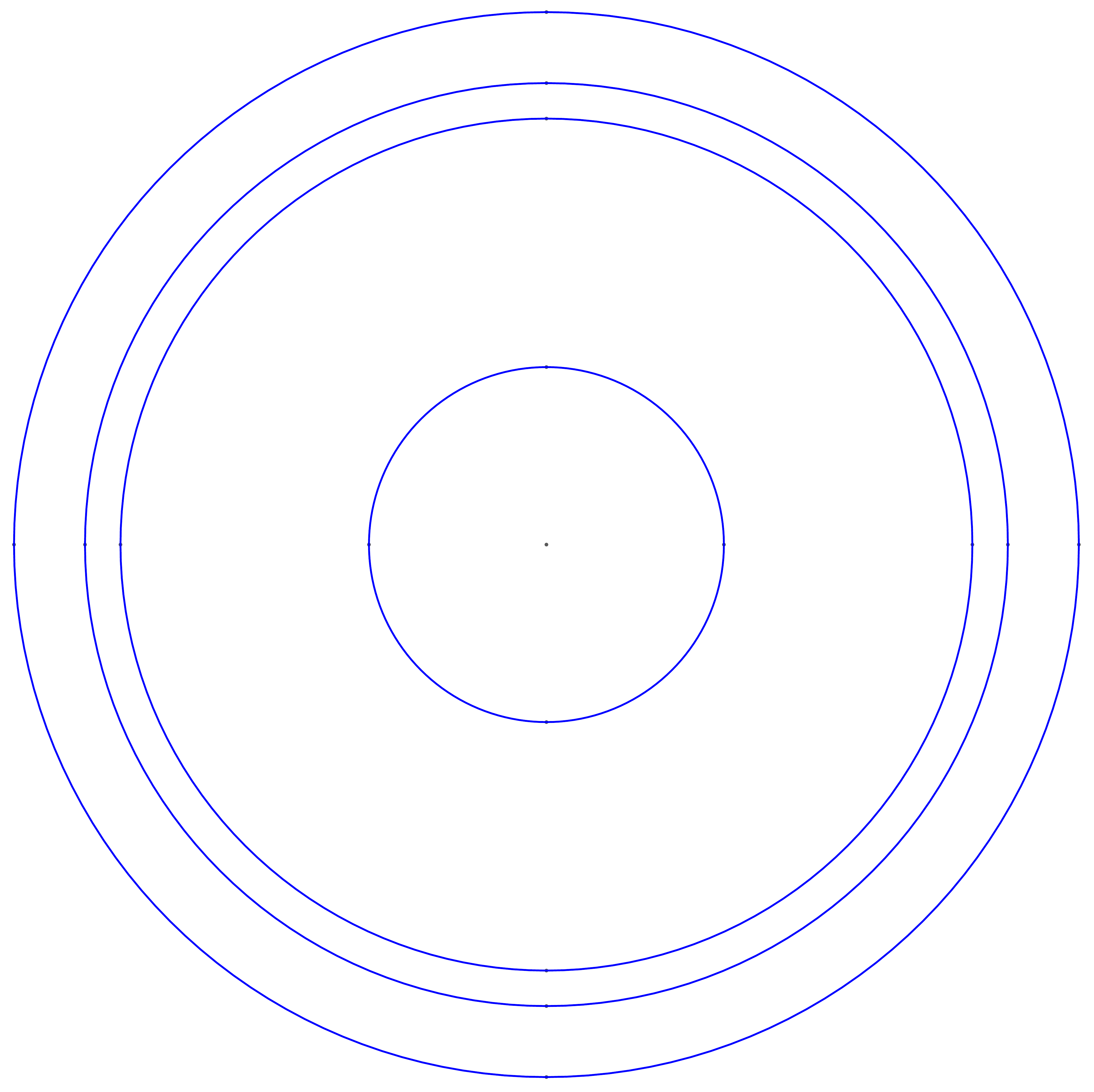}\qquad 
  \includegraphics[width = 0.24\textwidth]{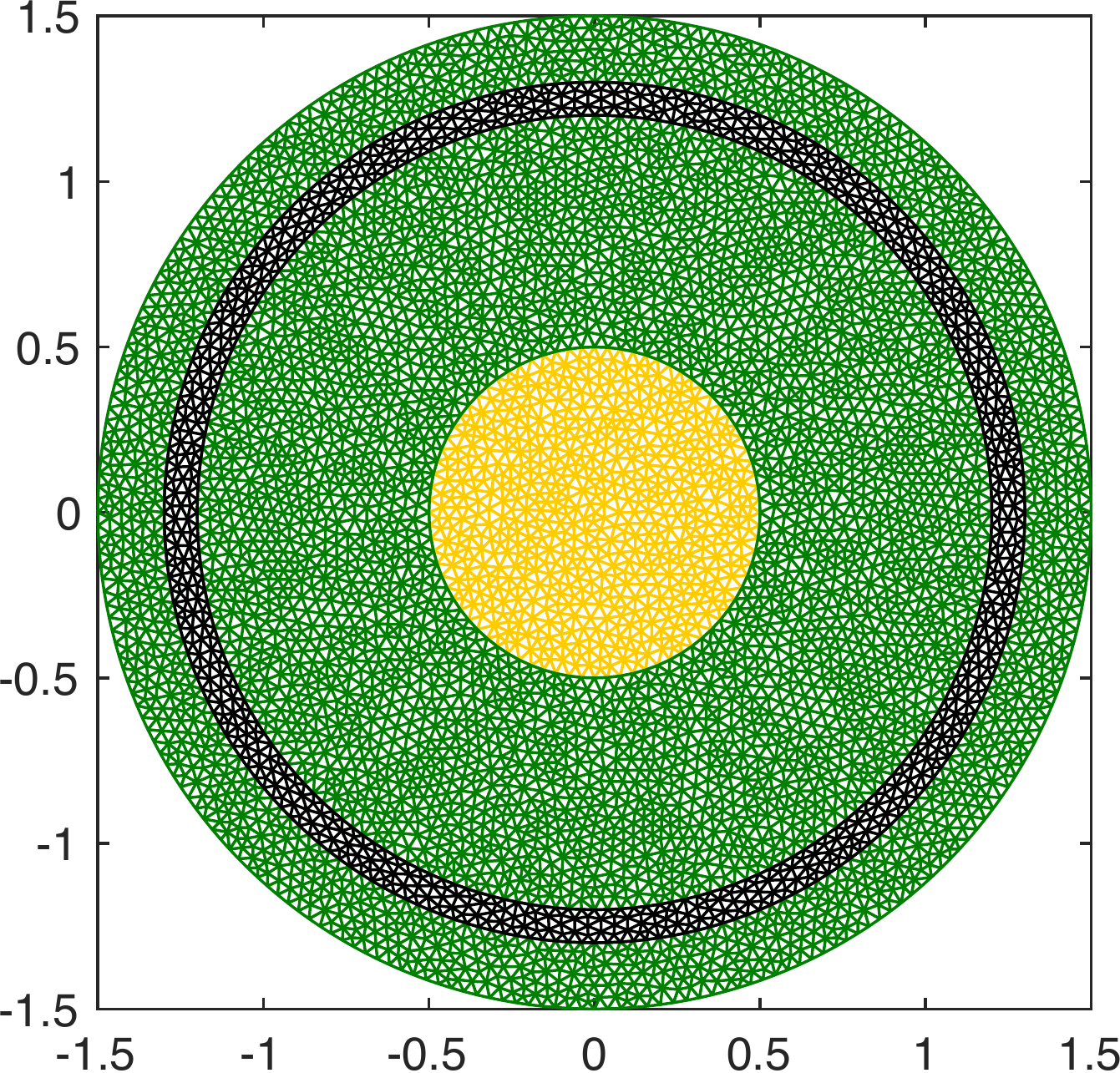}
  \caption{\label{f:ex1_setup}
  Left: computational domain where the inner circle is $\Omega$, the region inside the
  outer circle is $\widetilde\Om$, and the annulus inside $\widetilde\Om \setminus
  \Om$ is $\widehat\Om$ which is the region where the control is supported.
  Right: A finite element mesh.}
 \end{figure}
  
 \begin{figure}[htb]
  \centering
  \includegraphics[width=0.35\textwidth]{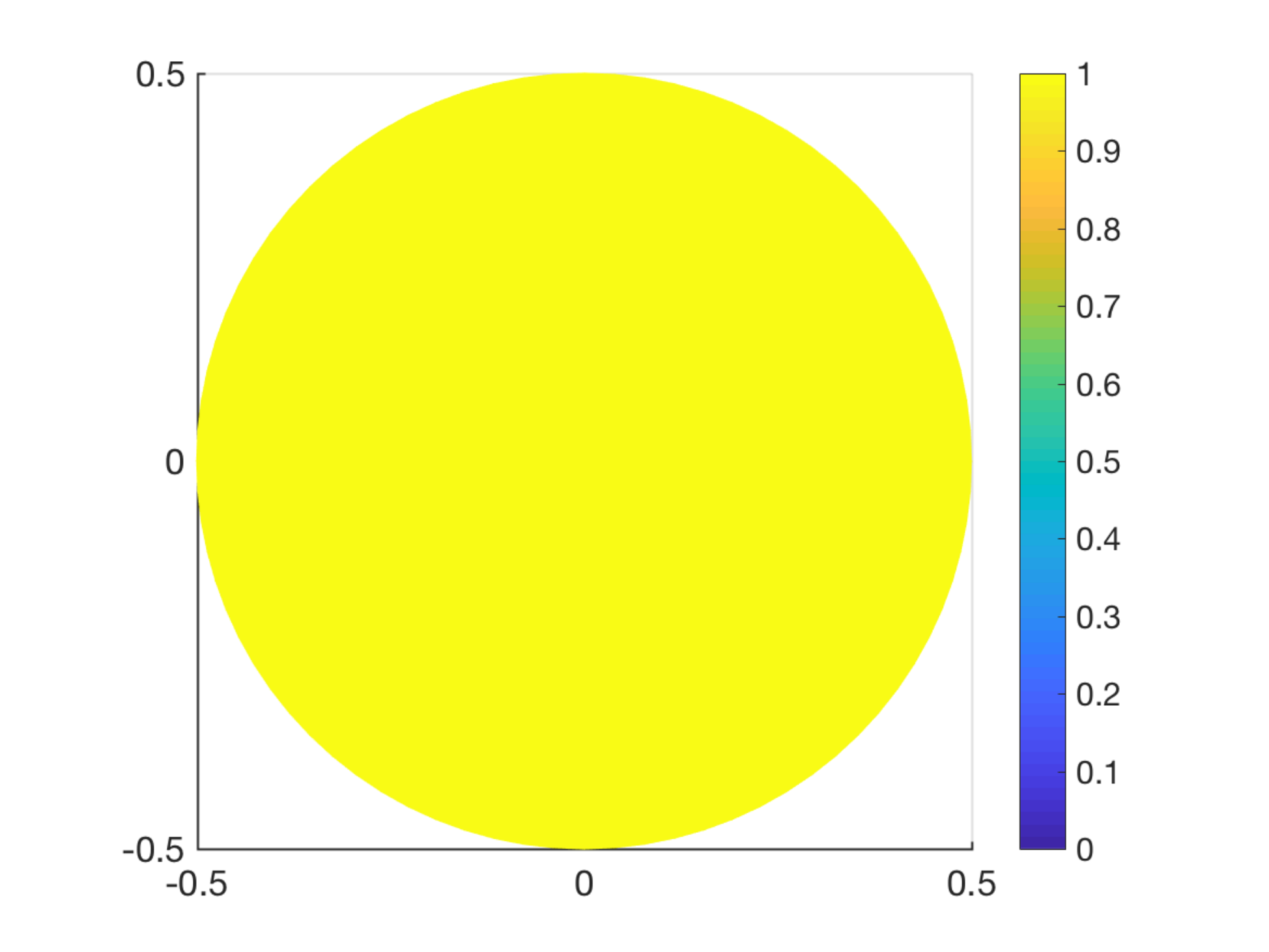}
  \includegraphics[width=0.35\textwidth]{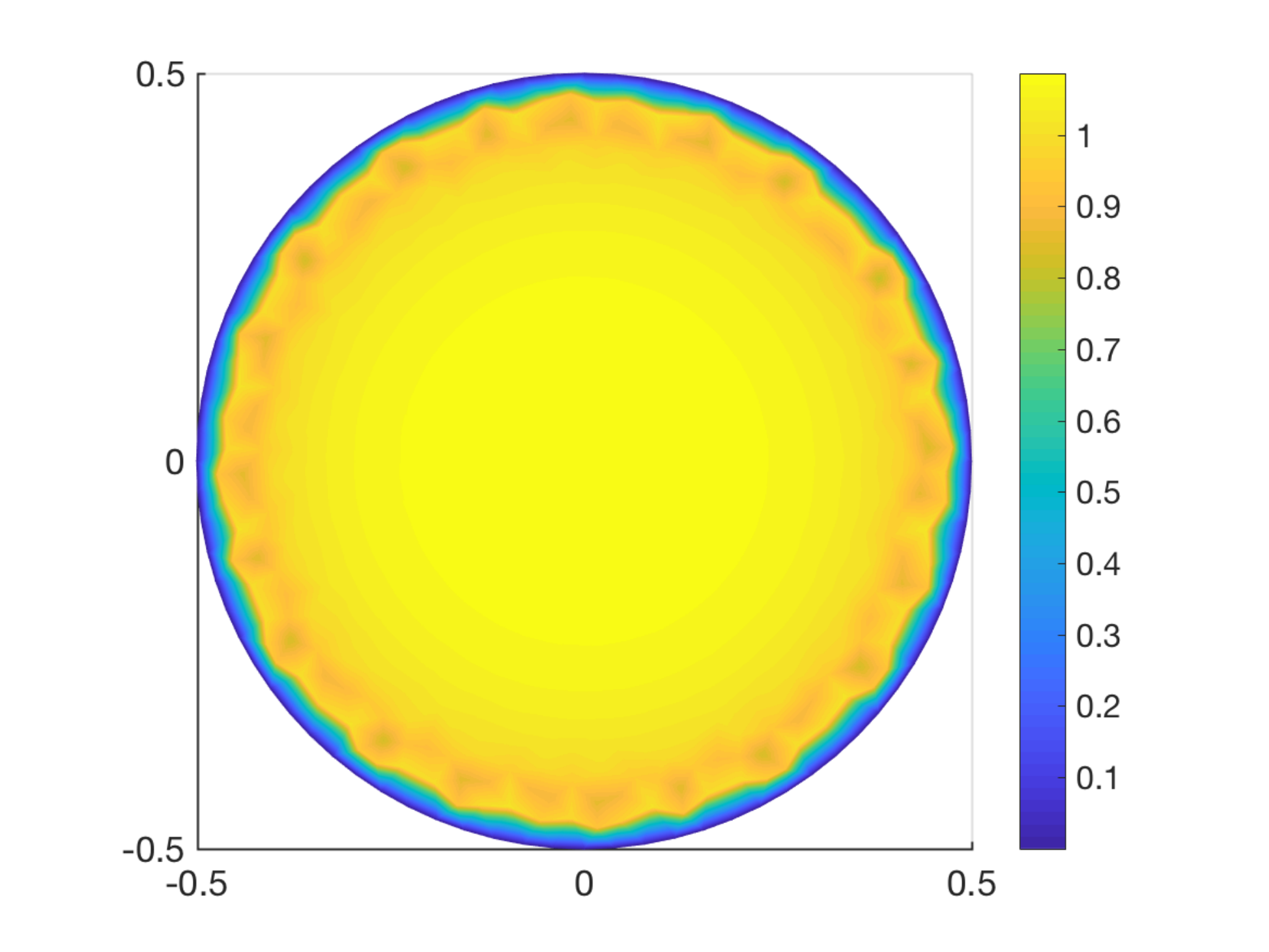}  
  \includegraphics[width=0.35\textwidth]{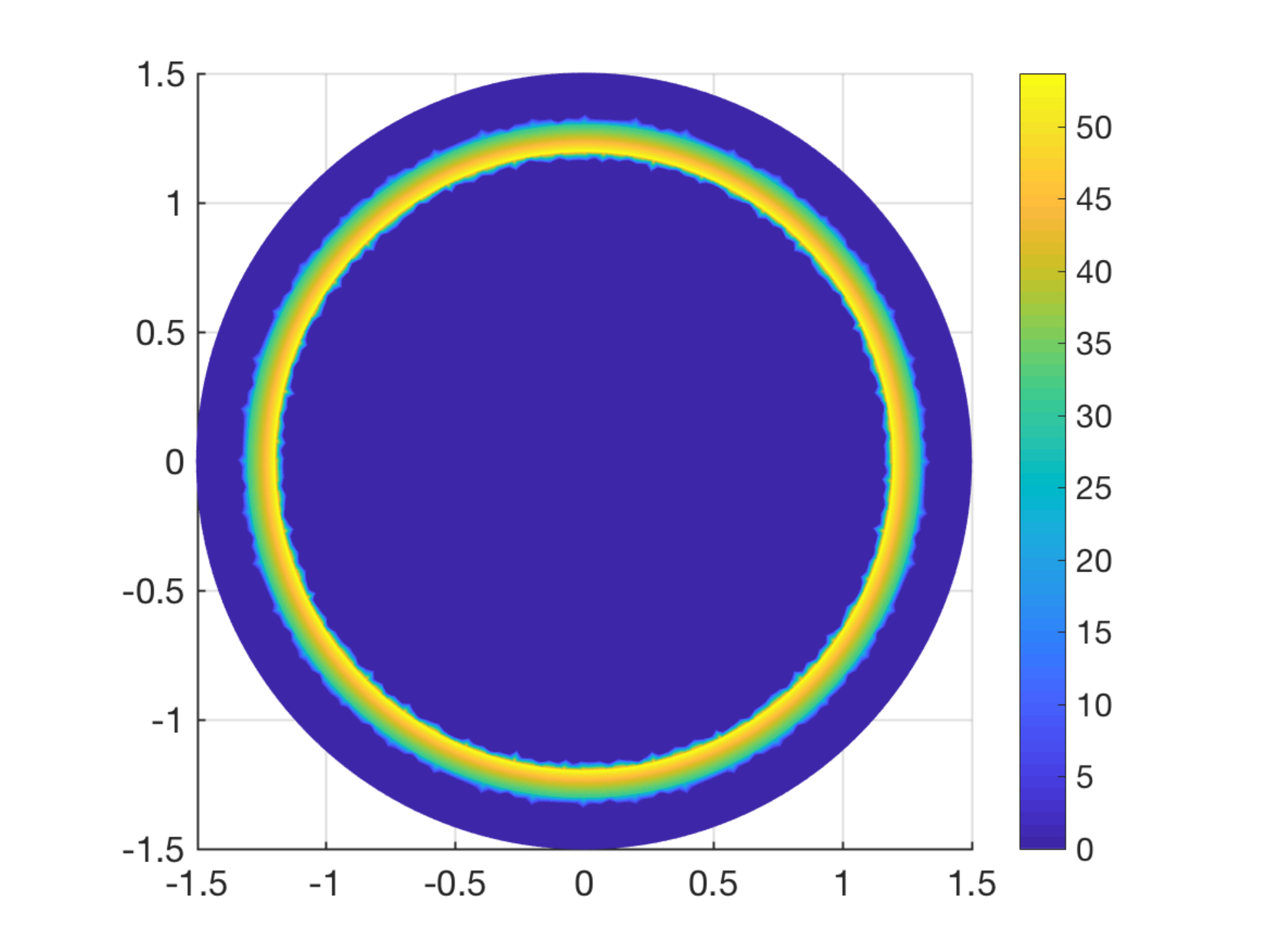}
  \includegraphics[width=0.35\textwidth]{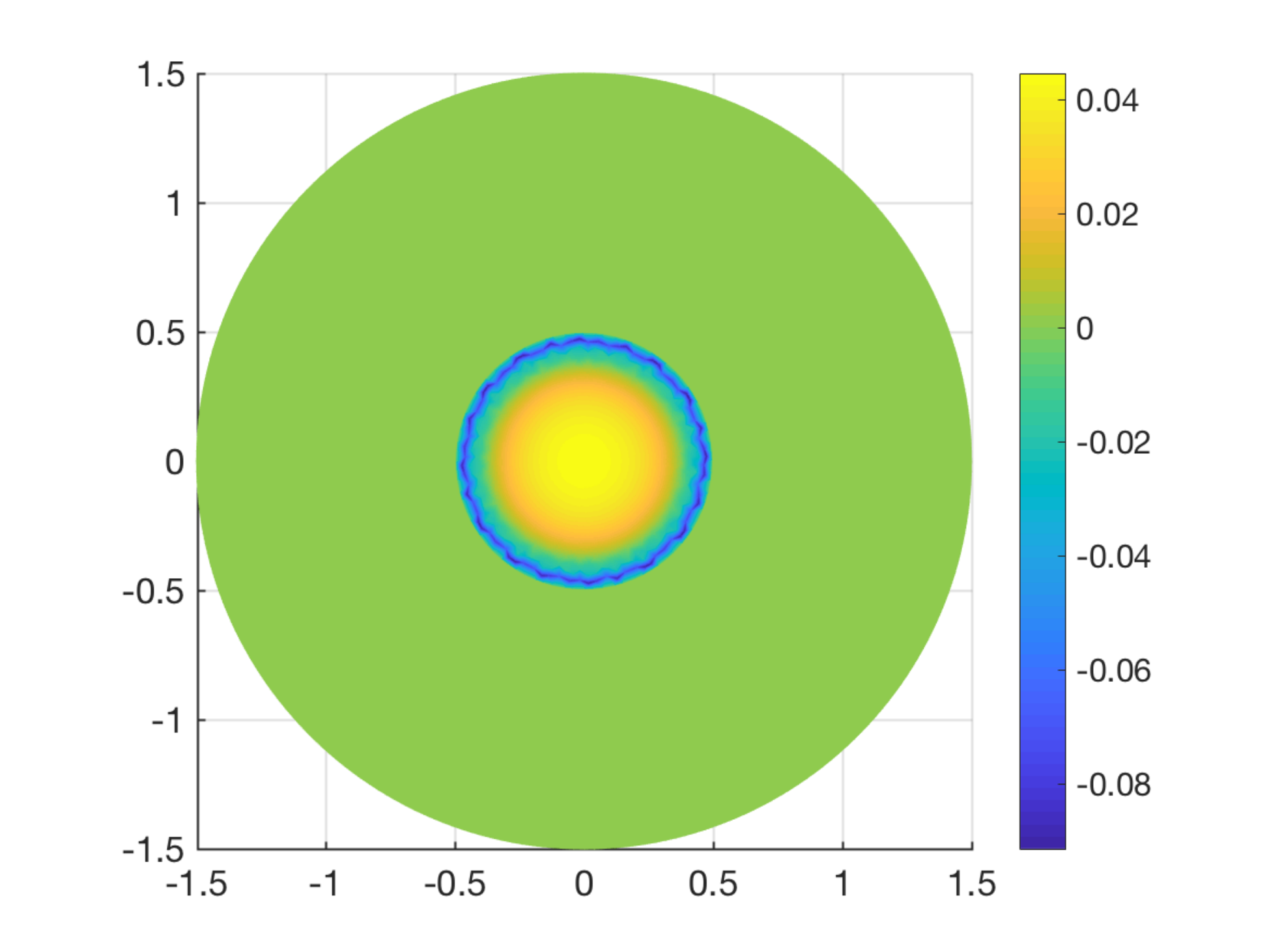}  
  \caption{\label{f:ex1s02}
  Example 1, $s = 0.2$: Top row: Left - Desired state $u_d$; Right - Optimal state 
  ${\bar u}_h$.
   Bottom row: Left - Optimal control ${\bar z}_h$, Right - Optimal adjoint ${\bar p}_h$.}
 \end{figure}
   
 \begin{figure}[htb]
  \centering
  \includegraphics[width=0.35\textwidth]{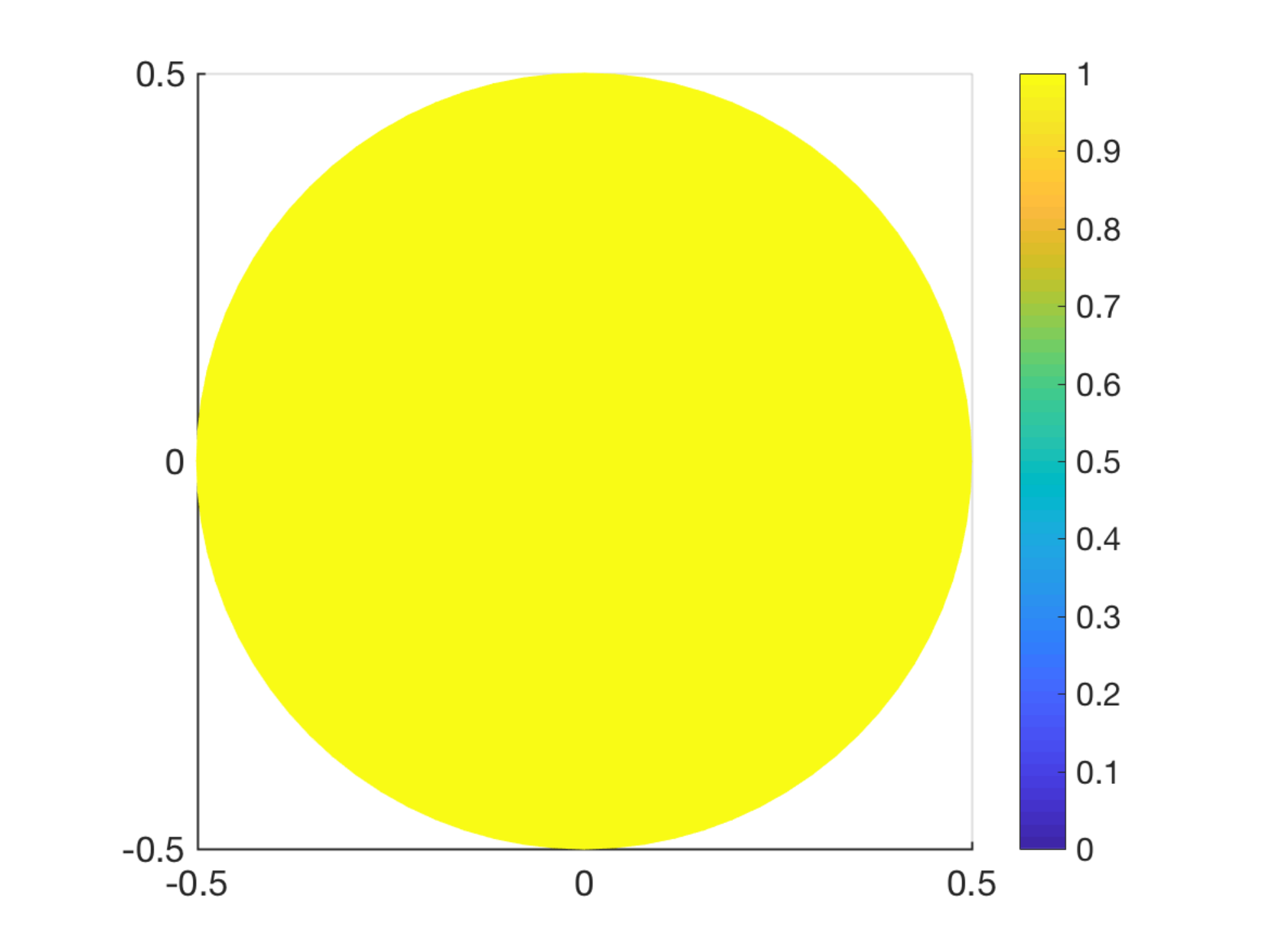}
  \includegraphics[width=0.35\textwidth]{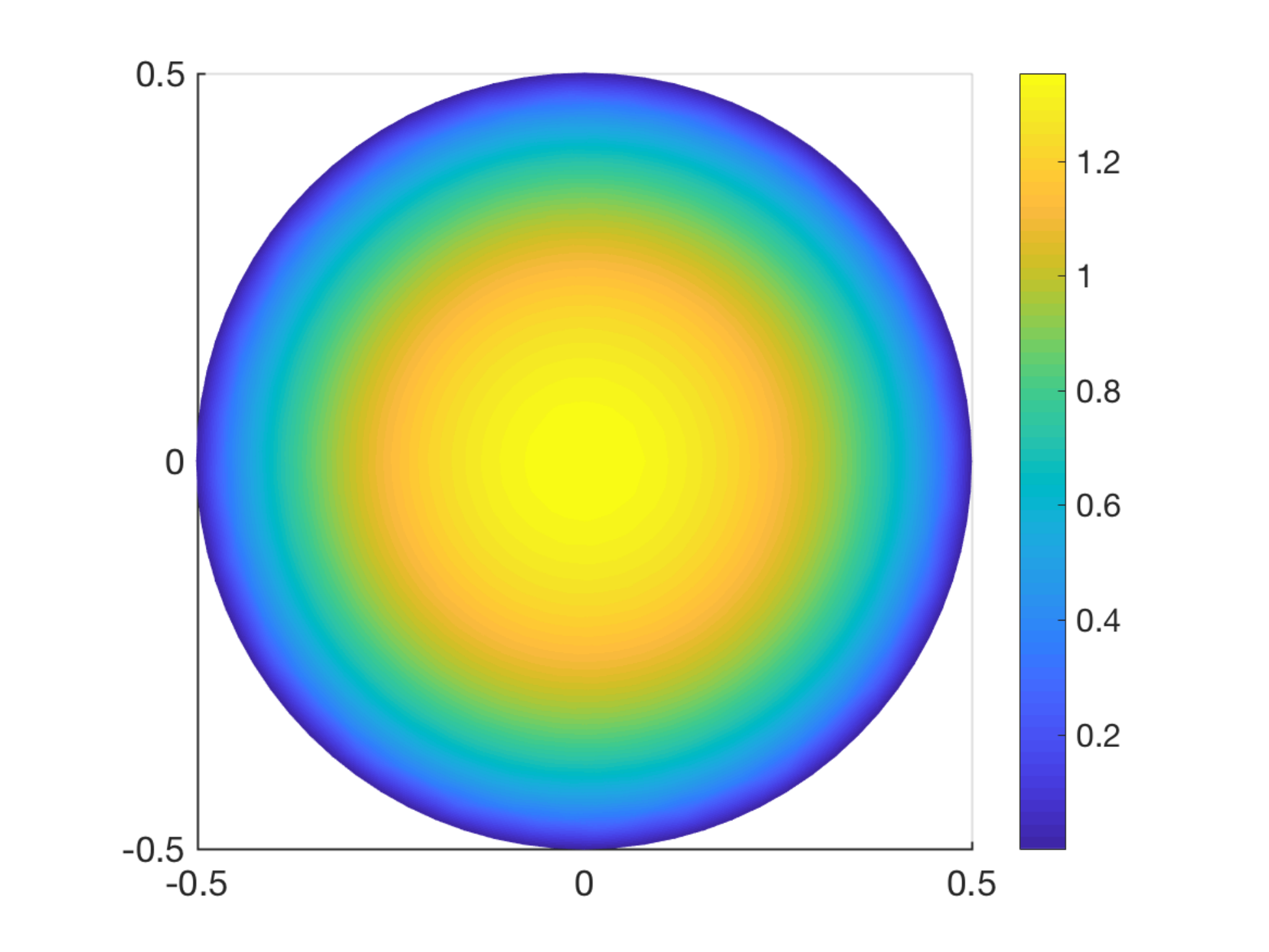}  
  \includegraphics[width=0.35\textwidth]{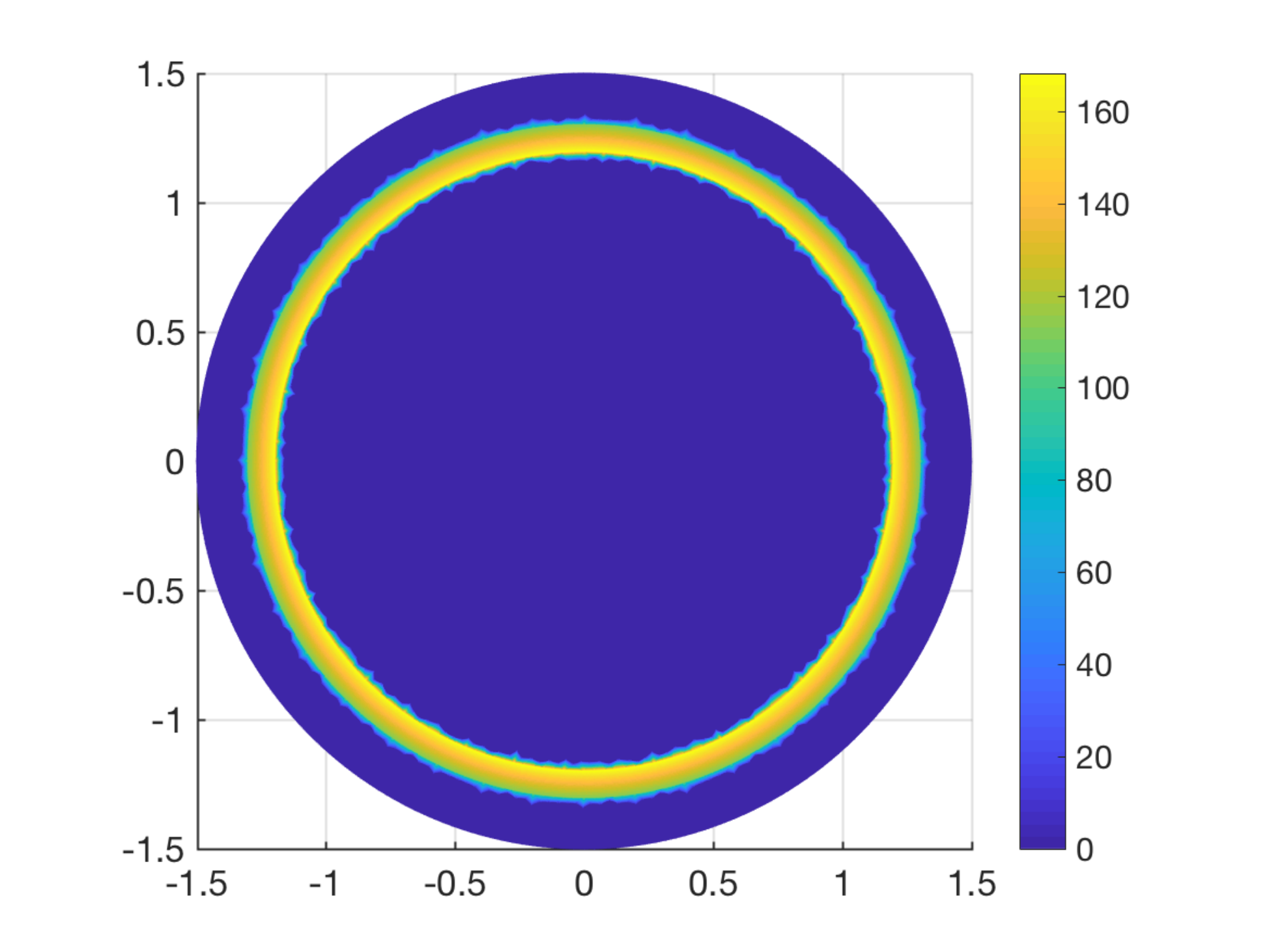}
  \includegraphics[width=0.35\textwidth]{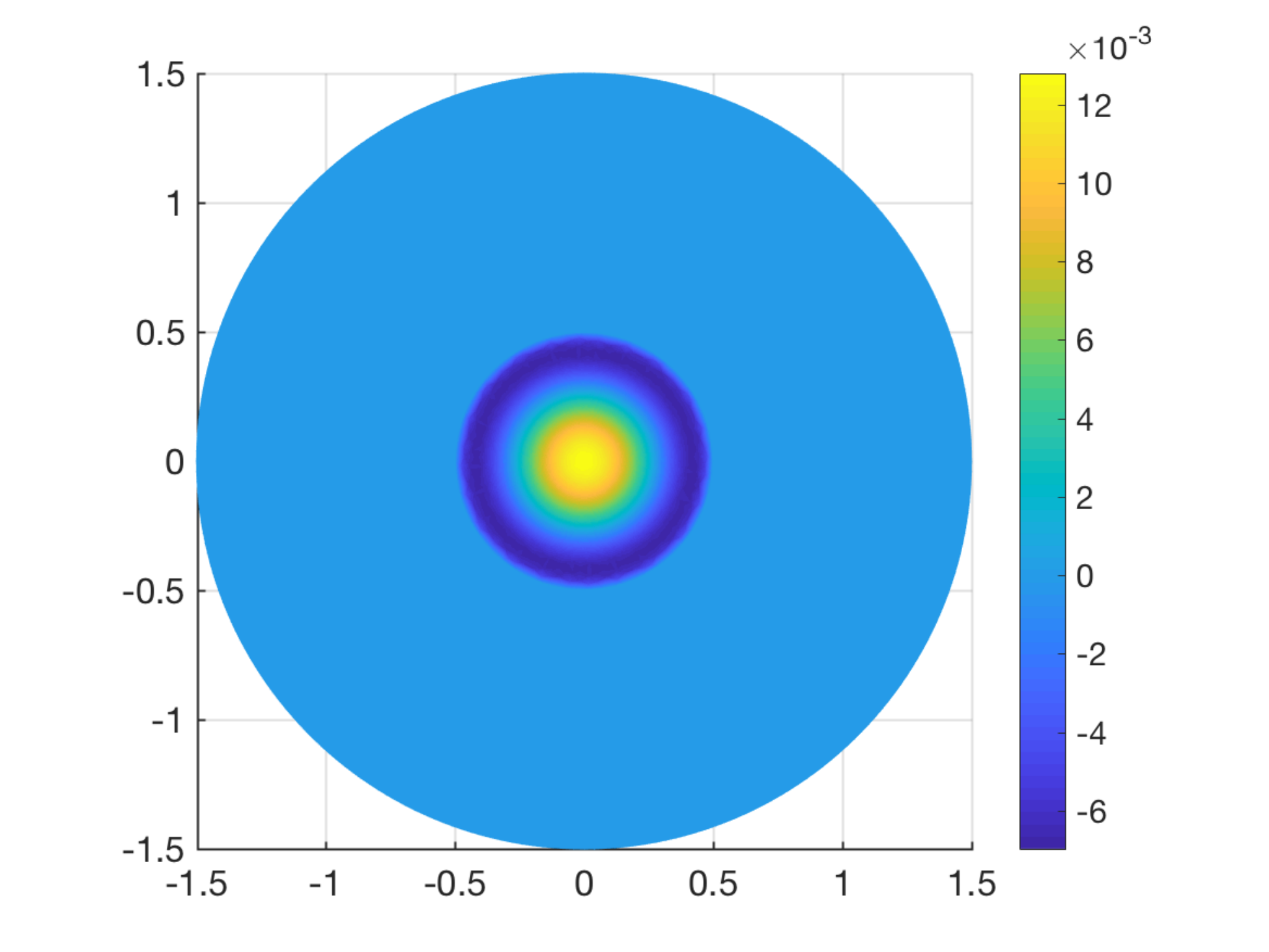}  
  \caption{\label{f:ex1s08}
  Example 1, $s = 0.8$: Top row: Left - Desired state $u_d$; Right - Optimal state 
  ${\bar u}_h$.
   Bottom row: Left - Optimal control ${\bar z}_h$, Right - Optimal adjoint ${\bar p}_h$.}
 \end{figure}
 }
\end{example}
\begin{example} 
{\rm 
The computational setup for our final example is shown in Figure~\ref{f:ex3_setup}. The
M-shape region is $\Om$ and the region inside the outermost ring is $\widetilde\Om = B_0(0.6)$. 
The smaller region inside $\widetilde\Om \setminus \Om$ is $\widehat\Om$ which is the support 
of control. The right panel in Figure~\ref{f:ex1_setup} shows a finite element mesh with 
DoFs = 4462. 

 \begin{figure}[htb] 
  \centering
  \includegraphics[width = 0.22\textwidth]{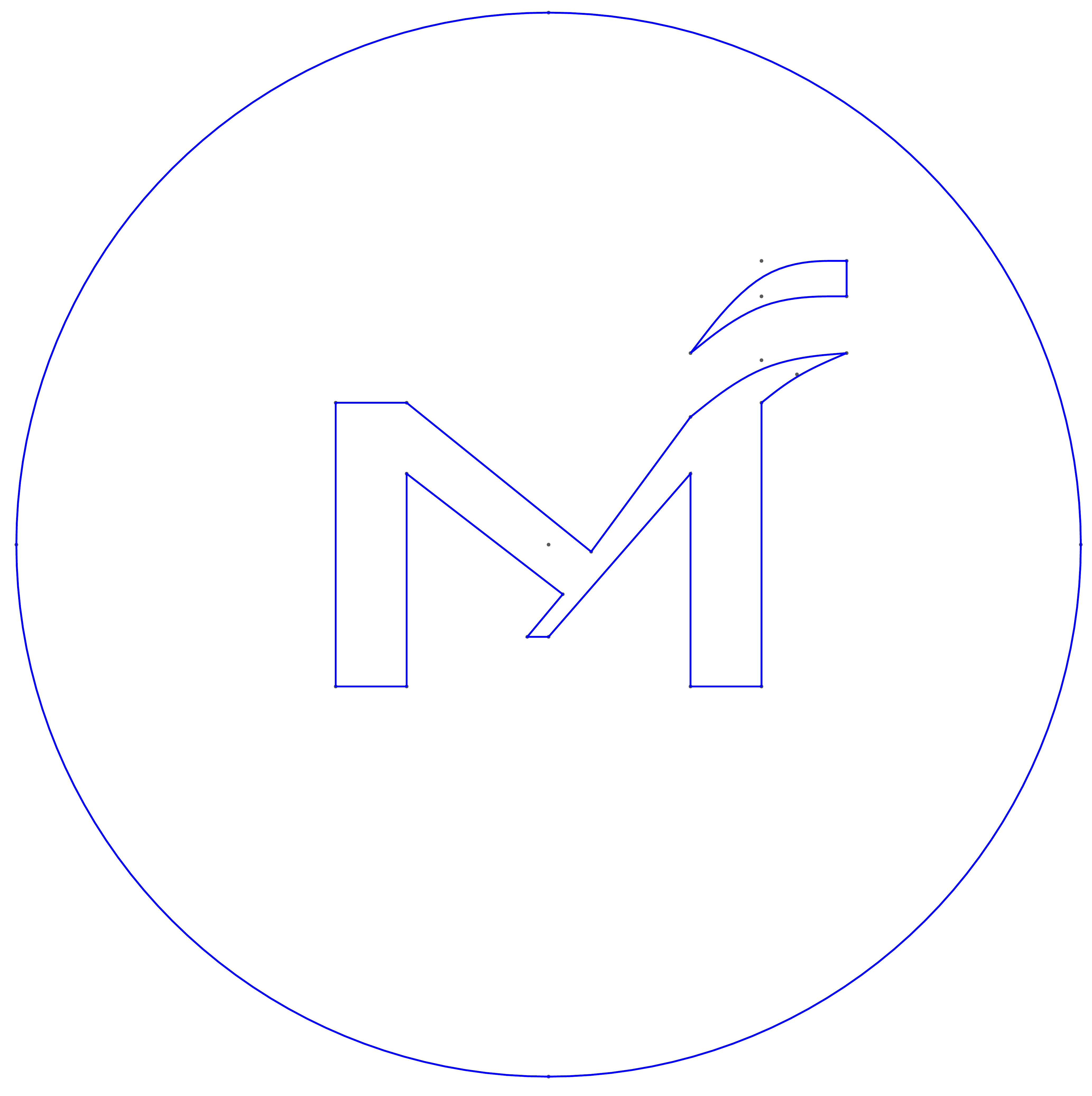}\qquad 
  \includegraphics[width = 0.24\textwidth]{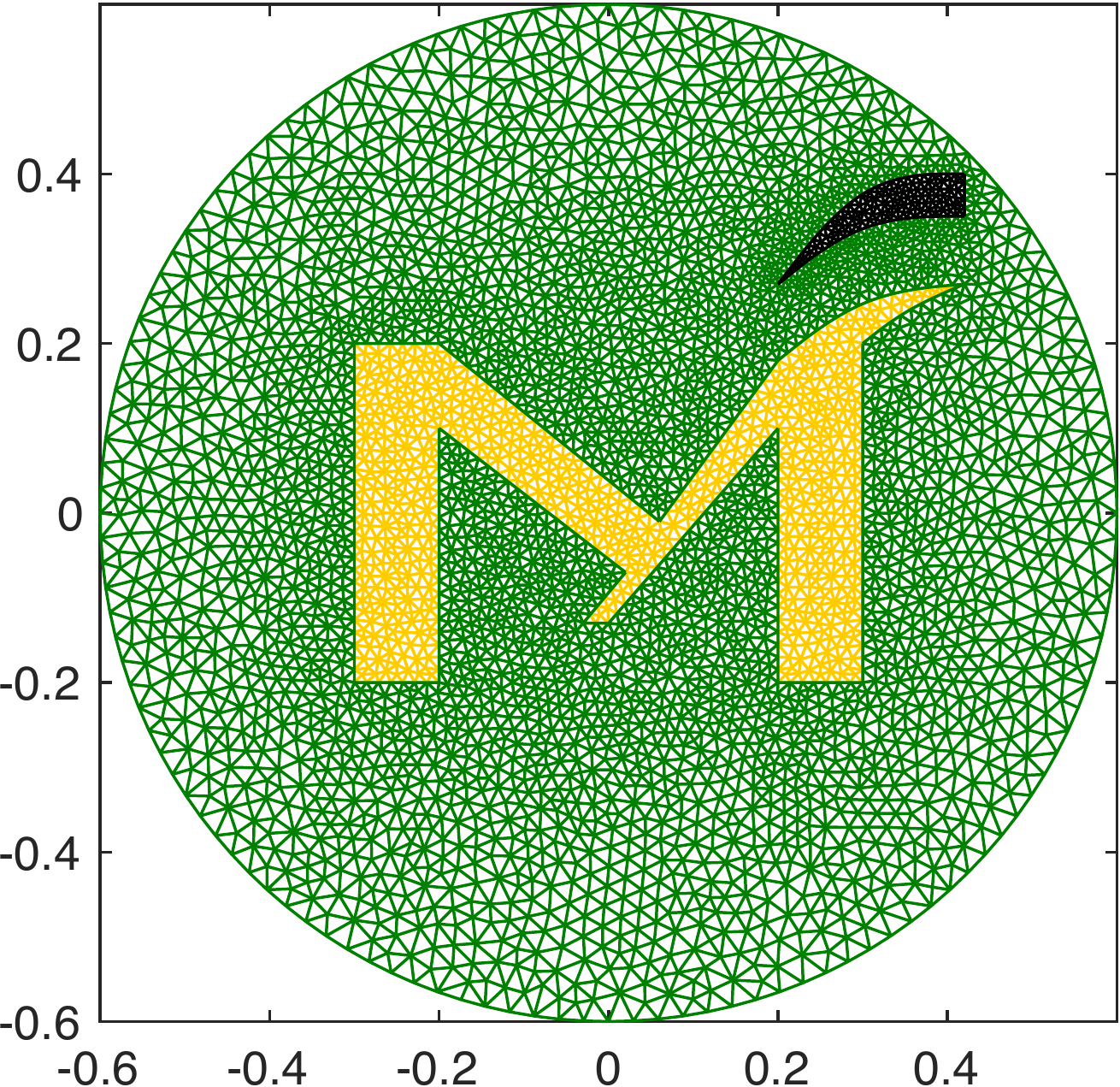}
  \caption{\label{f:ex3_setup} 
  Left: computational domain where the M-shaped region is $\Omega$, the region inside the
  outer circle is $\widetilde\Om$ and the region inside $\widetilde\Om \setminus
  \Om$ is $\widehat\Om$ which is the region where control is supported.
  Right: A finite element mesh.}
 \end{figure}

In Figure~\ref{f:ex3s08} we have shown the optimization results for 
$s = 0.8$. The top row shows desired state $u_d$ (left) and optimal state 
$\bar{u}_h$ (right). The bottom row shows the optimal control $\bar{z}_h$ 
(left) and the optimal adjoint variable $\bar{p}_h$ (right). Even though
the control is applied in an extremely small region we can still match
the desired state in certain parts of $\Om$.

 \begin{figure}[htb]
  \centering
  \includegraphics[width=0.31\textwidth]{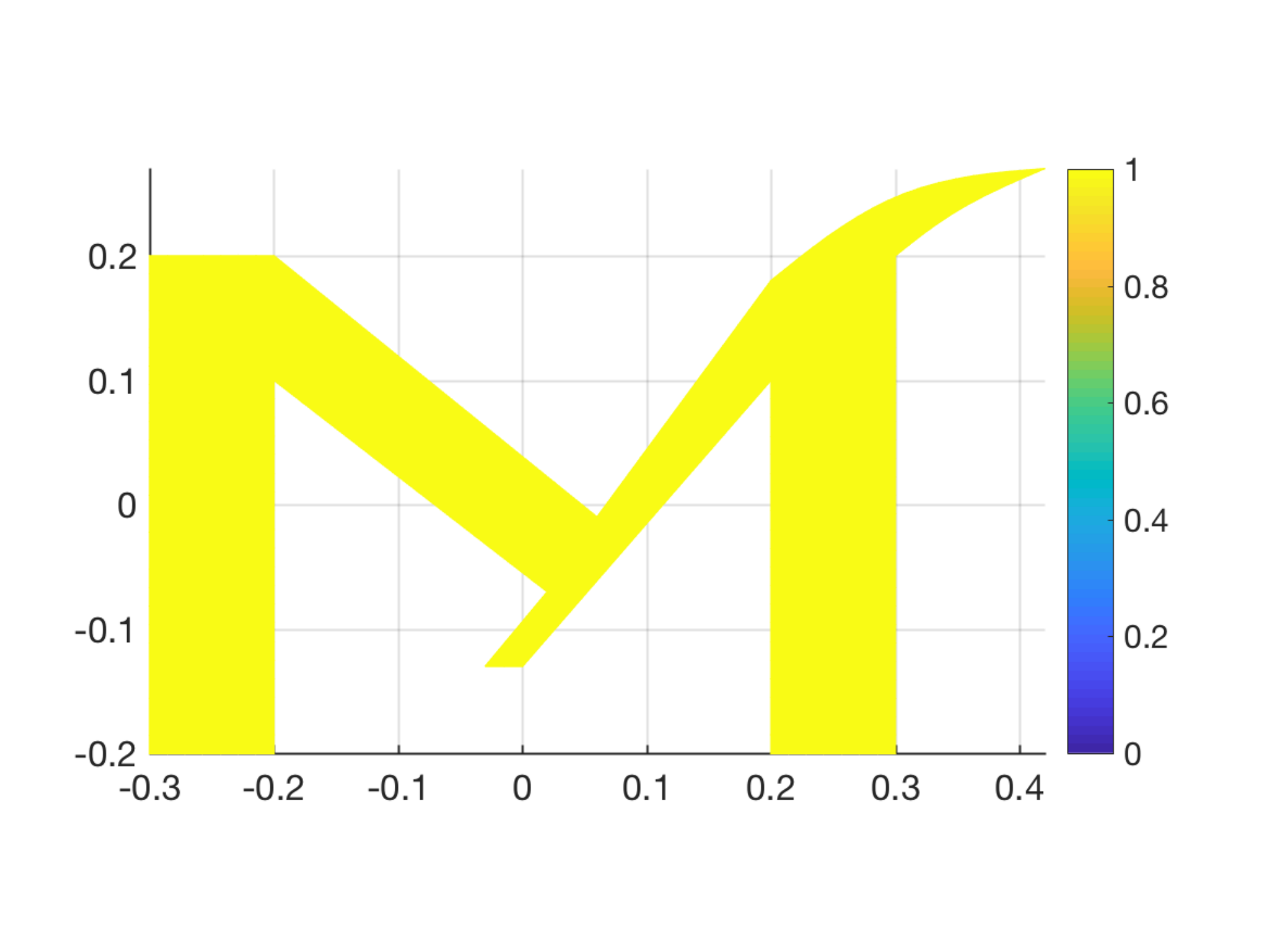} \quad
  \includegraphics[width=0.31\textwidth]{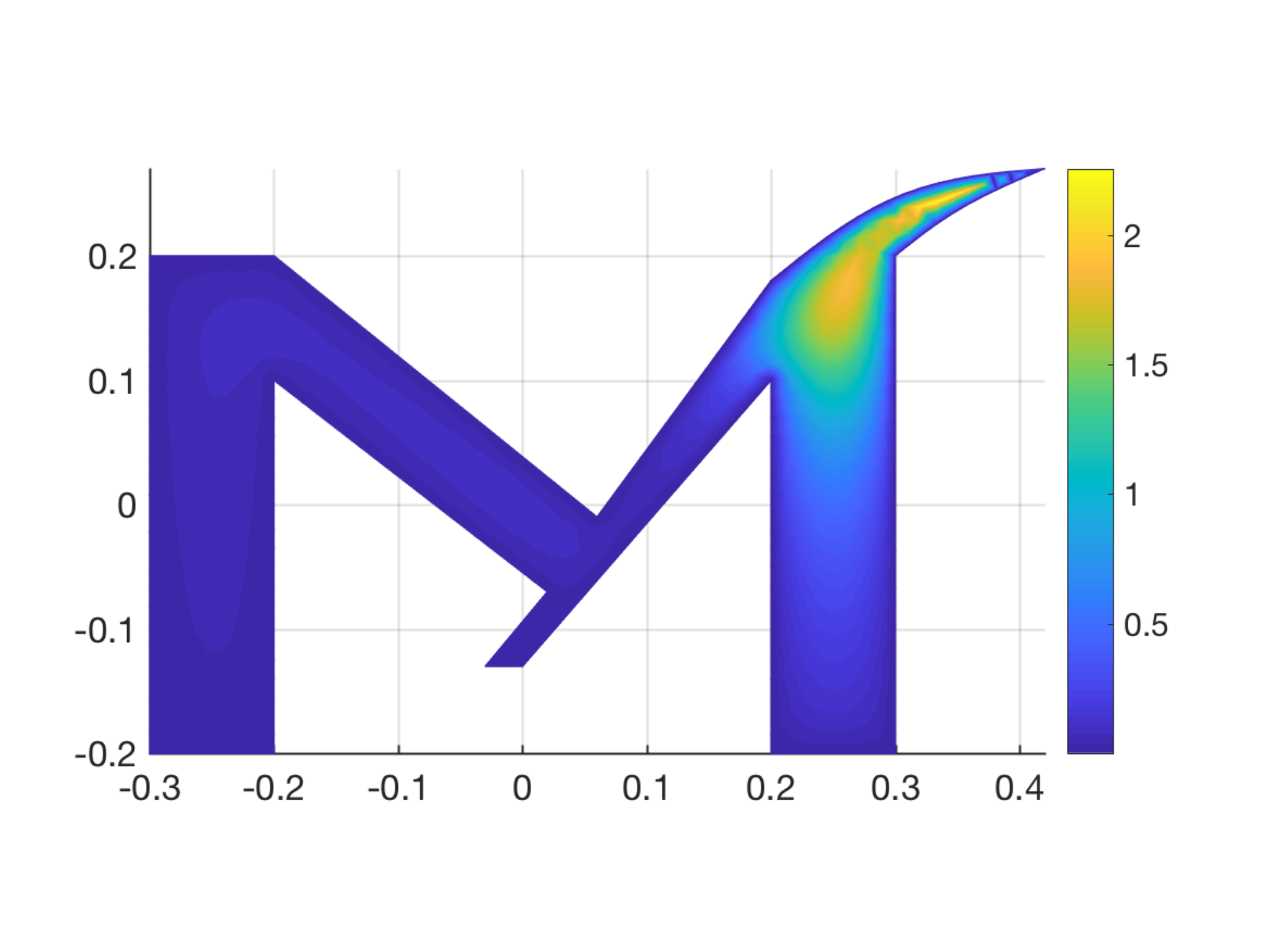}  
  \includegraphics[width=0.35\textwidth]{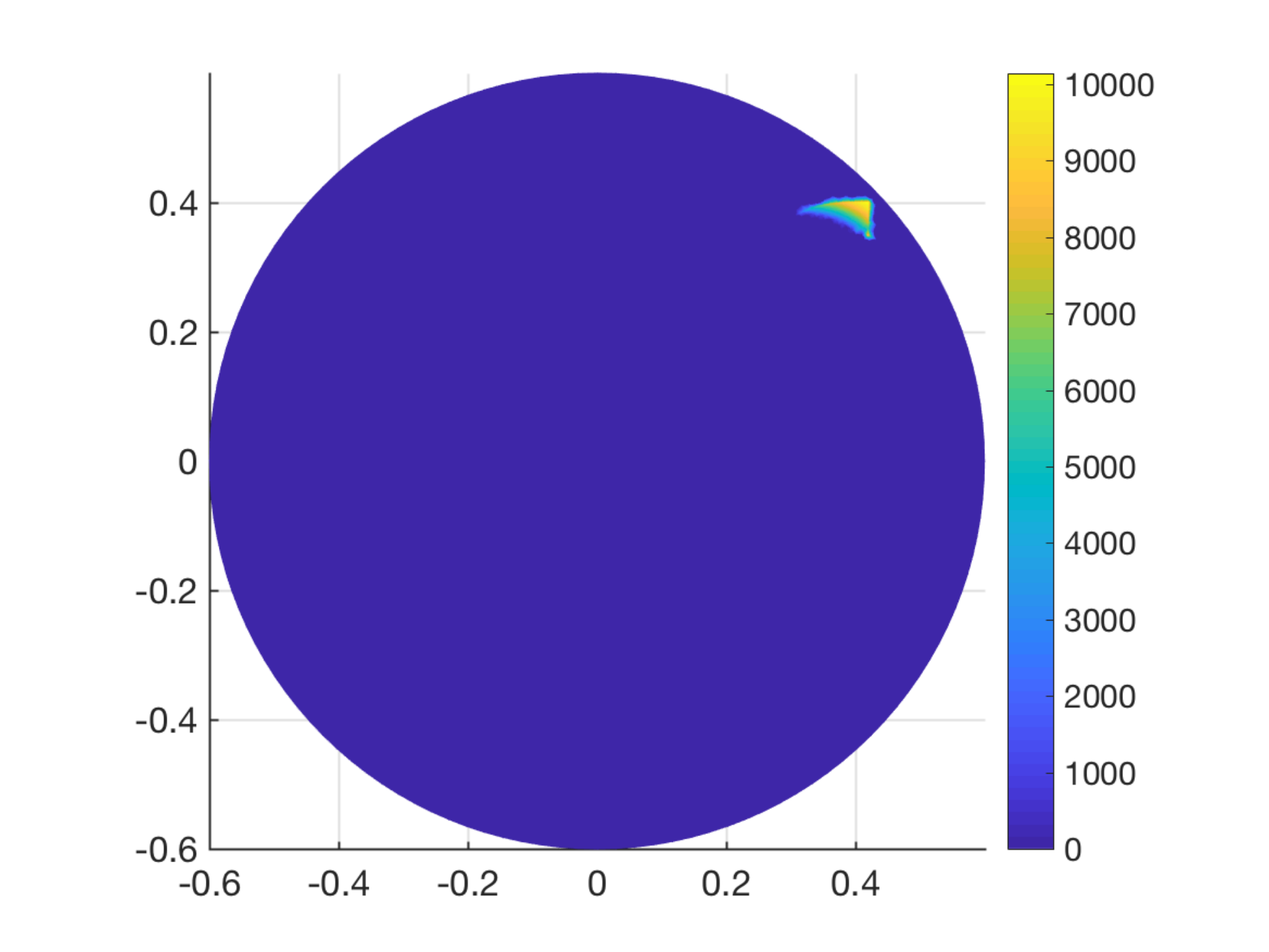}
  \includegraphics[width=0.35\textwidth]{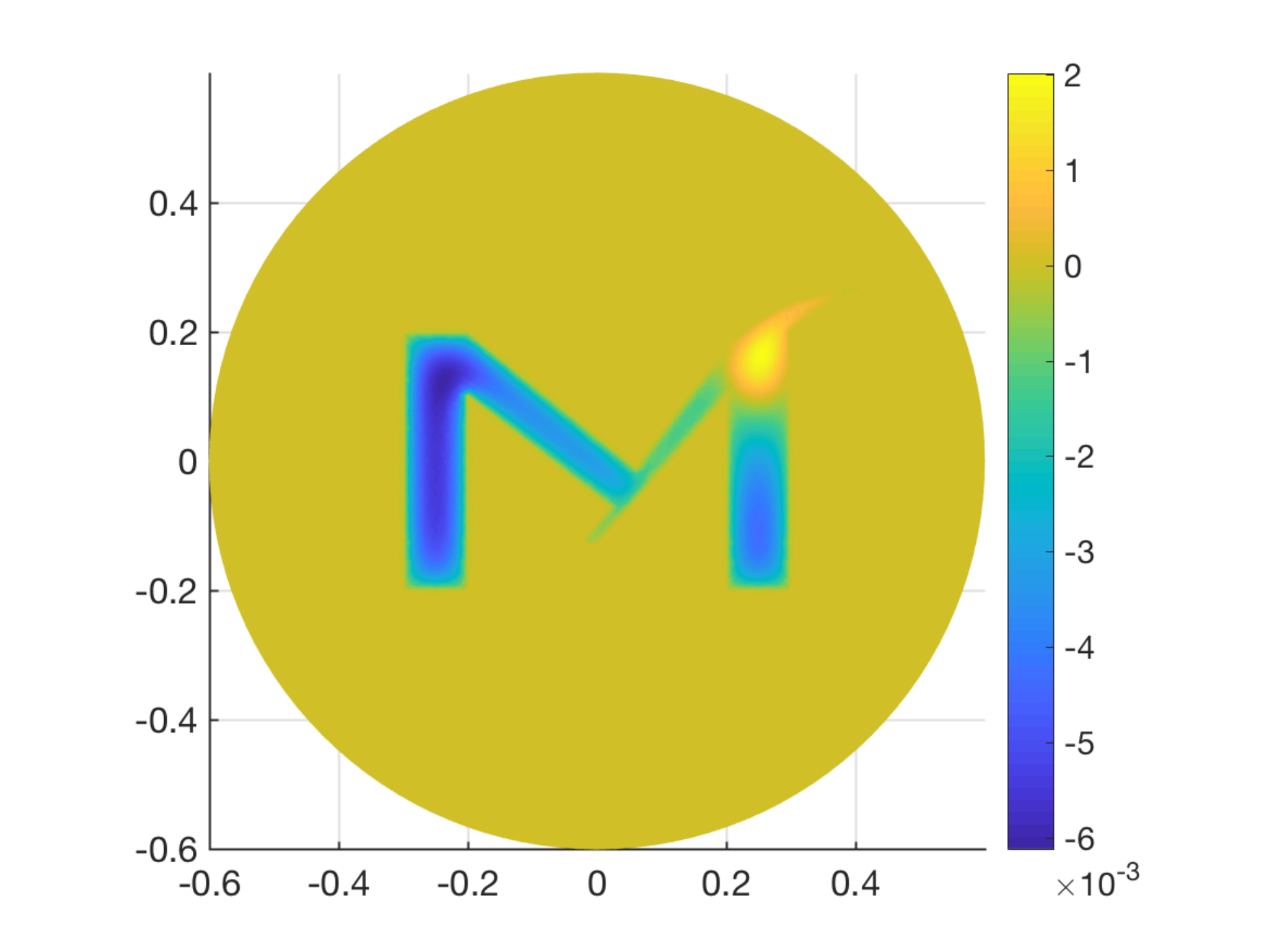}  
  \caption{\label{f:ex3s08} 
  Example 3, $s = 0.8$: Top row: Left - Desired state $u_d$; Right - Optimal state 
  ${\bar u}_h$.
   Bottom row: Left - Optimal control ${\bar z}_h$, Right - Optimal adjoint ${\bar p}_h$.}
 \end{figure}
  }
\end{example}

\vspace{0.5cm}

\noindent
{\bf Acknowledgement}:
We would like to thank Rolf Krause for suggesting to use the term ``interaction operator"
instead of ``nonlocal normal derivative". 


\bibliographystyle{plain}
\bibliography{refs}

\end{document}